\documentclass[a4paper, 11pt, reqno]{amsart}

\usepackage{subfiles}

\allowdisplaybreaks

\usepackage[latin1]{inputenc}
\usepackage[T1]{fontenc}
\usepackage[english]{babel}
\usepackage{latexsym,amsmath,amssymb,amsfonts}
\usepackage{dsfont}
\usepackage{color}
\usepackage{graphicx}
\usepackage{float}
\usepackage{esint}
\usepackage{bm}
\usepackage{mathtools}
\usepackage{enumitem}
\usepackage{hyperref}
\usepackage{cleveref}
\usepackage{bbm}
\usepackage{tikz}                        
\usepackage{pgfplots}    
\usepackage{soul}  

\usepackage{lineno} 
\usepackage[top=1in, bottom=1.25in, left=1in, right=1in]{geometry}

\newtheorem{theorem}{Theorem}[section]
\newtheorem{proposition}[theorem]{Proposition}
\newtheorem{lemma}[theorem]{Lemma}

\theoremstyle{remark}
\newtheorem{remark}[theorem]{Remark}
\theoremstyle{definition}
\newtheorem{definition}[theorem]{Definition}

\graphicspath{{Figures/}}
\DeclareGraphicsExtensions{.png,PNG,.pdf,.PDF}

\newcommand{\N}{\mathbb N}
\newcommand{\Z}{\mathbb Z}

\newcommand{\R}{\mathbb R}

\newcommand{\ho}{{\mathcal H}^{1}}
\newcommand{\hz}{{\mathcal H}^{0}}

\newcommand{\dd}{\,\mathrm{d}}

\newcommand{\dho}{\,\dd{\mathcal H}^{1}}

\newcommand{\wto}{\rightharpoonup}
\newcommand{\wtom}{\stackrel{*}{\rightharpoonup}}

\renewcommand{\o}{\Omega}

\newcommand{\f}{\mathcal{F}}
\newcommand{\g}{\mathcal{G}}

\newcommand{\MC}{C}

\newcommand{\gr}{\text{graph}}

\newcommand{\bv}{\mathrm{BV}}

\newcommand{\average}{{\mathchoice {\kern1ex\vcenter{\hrule height.4pt
				width 6pt
				depth0pt} \kern-9.7pt} {\kern1ex\vcenter{\hrule height.4pt width 4.3pt
				depth0pt}
			\kern-7pt} {} {} }}
\newcommand{\med}{\average\int}

\newcommand{\blue}[1]{{\color{black} {#1}}}

\title[Graph epitaxy with adatoms]{Two dimensional graph model for
epitaxial crystal growth with adatoms}

\author{Riccardo Cristoferi}
\address[Riccardo Cristoferi]{Department of Mathematics - IMAPP, Radboud University, Nijmegen, The Netherlands}
\email{riccardo.cristoferi@ru.nl}

\author{Gabriele Fissore}
\address[Gabriele Fissore]{Department of Mathematics - IMAPP, Radboud University, Nijmegen, The Netherlands}
\email{gabriele.fissore@ru.nl}

\date{\today}                                        

\subjclass[2020]{}

\begin{document}

\begin{abstract}
We consider a model to describe stable configurations in epitaxial growth of crystals in the two dimensional case, and in the regime of linearized elasticity.
The novelty is that the model also takes into consideration the adatom density on the surface of the film. These are behind the main mechanisms of crystal growth and formation of islands (or quantum dots).
The main result of the paper is the integral representation of the relaxed energy.
\end{abstract}

\maketitle

\section{Introduction}

The ability to grow thin films of crystal over a substrate is a technology that has applications in several areas, from surface coating, to lithography.
Practitioners developed several techniques to grow crystals over a substrate.
Vapor deposition techniques are among the most important and implemented: the substrate is immersed in a vapor, and mass transfer from the latter to the former is responsible for the growth of the crystal.
In order for the crystal to growth, two conditions need to be satisfied: the vapor has to be saturated, and the substrate is kept at a significantly lower temperature than the vapor.
The former ensures attachment of vapor atoms on the substrate, while the latter quick termalization of deposited atoms.
In particular, this implies that the entropic free energy is reduced after attachment.

In order to grow a crystal (namely, an ordered structure), attached atoms, called \emph{adatoms}, need to have sufficient energy to move from the landing location to a position of equilibrium. This depends on the type of materials used in the vapor and for the substrate. Surface diffusion of adatoms is therefore the mechanism used by thin films to growth as a crystal.

If the growth process is made in such a way that the first layers of the film arrange in the same lattice structure of the substrate, the growth is called \emph{epitaxial}. Of course, the atoms of the deposited material are stretched or compressed, since they are not in their (sometimes, stress free) natural configuration.\\

The dynamic of the crystal growth process is extremely complicated, and it is influenced by many factors. In particular, the ratio between the tendency of the adatoms to stick to the substrate and their tendency to diffuse.
Three modes of growth are defined based on this ratio: the \emph{Frank-van der Merwe growth mode}, where diffusion is stronger and thus the crystal growth layer by layer, the \emph{Volmer-Weber growth mode}, where diffusion is weaker, and therefore adatoms tend to form islands on the substrate, and an intermediate one, the \emph{Stranski-Krastanov growth mode}, where the first monolayers of the film behave like in a Frank-van der Merwe growth mode, while after a certain threshold, it starts forming islands. Here we consider the latter case. \\

In the epitaxial Stranski-Krastanov growth mode, it is observed that, after a few monolayer of material are deposited, the film accumulates too much elastic energy that it is no more energetically convenient for atoms of the film to stick to the crystalline structure of the substrate. Thus, relaxation processes are employed in order to reduce the total energy of the system. The most important ones are corrugation of the surface, and creation of defects. These are known in the literature as stress driven rearrangement instabilities (see \cite{Grin}). The former is responsible for non-flat surfaces as well as for the appearance of islands (agglomerates of atoms, also called \emph{quantum dots}) on the surface. With the latter, instead, the film introduces singularities in its crystalline structure, such as cracks and dislocations.

It is extremely important to be able to control this extremely complex process in such a way to reduce impurities as much as possible, or at least to be able to quantify them.\\

The physical literature on crystal growth is extremely vast. Here we limit ourselves to mention the pioneering work \cite{SpeTer97} by Spencer and Tersoff.

From the mathematical point of view, several investigation have been carried out, focusing on different aspects of the growth process. There are both discrete models, and continuum ones. Here we focus on these latter. In particular, the work \cite{BonCha02} by Bonnetier and Chambolle laid the foundations for rigorous mathematical investigations of stable equilibrium configurations of epitaxially strained elastic thin films in the linear elastic regime. The authors considered the two dimensional case and proved an integral representation formula for the relaxed energy with respect to the natural topology of the problem, as well as a phase field approximation.
In \cite{FonFusLeoMor07}, Fonseca, Fusco, Leoni, and Morini proved a similar result by using an independent strategy, and also investigated the regularity of configurations locally minimizing the energy.

Questions about the stability of the flat profile were investigate by Fusco and Morini in \cite{FusMor12} for the case of linear elasticity, and in \cite{Bon15} by Bonacini in the nonlinear regime. Moreover, in \cite{Bon13}, Bonacini considered the same question for the case where surface energy is anisotropic, showing, surprisingly, that the flat interface is always stable.

It was not until 2019, with the work \cite{CriFri20} by Crismale and Friedrich that the three dimensional case was considered.
Indeed, despite the existence of investigations for similar functionals in higher dimension (see the work \cite{ChaSol} by Chambolle and Solci, and \cite{BraChaSol} by Braides, Chambolle, and Solci for the study of material void) were available, all of them considered elastic energies depending on the full gradient of the displacement. On the other hand, it is known that physically compatible models for elasticity must depend on the symmetrized gradient.
The reason for such a time gap between the two and the three dimensional case was technical: it was not clear how to get compactness of a sequence of configurations with uniformly bounded energy. This required the introduction of a new functional space: GSBD, the space of Generalized Functions of Bounded Deformation, designed in the work \cite{DalMasoGSBD} by Dal Maso in 2014 specifically to address such an issue.\\

What all of the above continuum models are neglecting is the role of adatoms in the creation of equilibrium stable interfaces.
The importance of considering their effect was made clear by Specer and Tersoff in \cite{SpeTer97}, where the authors highlighted that considering the effect of adatoms, and in particular of surface segregation of several species of deposited material, will change the equilibrium configurations predicted by the model, and hopefully provide a more accurate description of those observed in experiments.

This was made even clearer in the seminal paper \cite{FriGur} by Fried and Gurtin. The manuscript unified several ad hoc investigations that focused on specific aspects on crystal growth or used specific assumptions to derive the model.
In particular, it was noted that considering adatoms will, on the one hand, add a new variable to the problem, while, on the other hand, will make the evolution equations pararbolic. Note that this is a huge mathematical advantage, since in \cite{FonFusLeoMor12} and in \cite{FonFusLeoMor18}, the authors had to add an extra term to the energy (that nevertheless has some physical interpretation) to regularize the non-parabolic evolution equations obtained from the model that does not take into consideration adatoms.

Following this direction of investigation, in \cite{CarCriDie}, the first author, together with Caroccia and Dietrich, started the study of a variational characterization of the evolution equations derived by Fried and Gurtin.
In that paper, the authors considered a variational model describing the equilibrium shape of a crystal, where the elastic energy is neglected, and the crystal can grow without the graph constraint. From the energy for regular configurations, a natural topology was identified, and a representation formula for the relaxed energy was obtained. The result highlighted the interplay between oscillations of crystal surfaces and changes in adatom density in order to lower the total energy.
The result obtained in that paper was different from previous investigations by Bouchitt\'{e} (see \cite{Bou87}), Bouchitt\'{e} and Buttazzo (see \cite{BouButt90}), and Buttazzo and Freddi (see \cite{ButtFredd}), due to the choice of the topology.

In a subsequent paper (see \cite{CarCri}), a phase field model was considered in a more general setting, to pave the way towards the analysis of the convergence of the gradient flows.\\

In this paper, we continue this line of research by considering the case where the material is deposited on a substrate, its profile can be described by a function, and the elastic energy of the film is considered, as well as the surface energy of adatoms.
The goal is to obtain a representation formula for the relaxed energy in the natural topology of the problem.
In order to develop the main ideas needed for such an investigation, this paper focus on the two dimensional case.
\blue{The main contribution of the paper is to show how the mechanism identified in \cite{CarCriDie} where oscillations of the profile interact with adatom concentration plays a role in the case where the geometry of the configuration is constrained to be a graph.
This might seem as an easier case than that treated in \cite{CarCriDie}, where the profile of the crystal was free to growth in any direction.
Nevertheless, the graph constraint poses several challenges that have to be tackled with the utmost care, in order to be properly overcome.
Indeed, we prove that the relaxed energy differs from that of \cite{CarCriDie} exactly on vertical cracks of the deposited layer.
In particular, we introduce a strategy to deal with oscillations and adatom concentration on vertical cracks, whose robustness will be tested in a forthcoming paper where we will investigate a phase-field approximation of the model and another where we will treat the three dimensional case.
}

Forthcoming papers will also consider the dynamics of the model, and the situation when multiple species of materials are deposited at the same time.


\subsection{The model}

In this section we introduce the model that we will study.
We consider the two dimensional case. This corresponds to three dimensional configurations that are constant in one direction. We work within the continuum theory of epitaxial growth.
The main assumptions of the model are the following:
\begin{itemize}
	\item[(i)] The profile of the configurations of the thin film can be described as the graph of a function;
	\item[(ii)] We neglect surface stress; 
	\item[(iii)] The exchange of atoms between the substrate and the deposited film is negligible;
	\item[(iv)] The atoms of the substrate do not change position.
\end{itemize}

The free energy of a configuration is the sum of a bulk energy and a surface energy. The former is the elastic energy due to rearrangement of the atoms of the deposited film from a stress free configuration (atoms sitting in their natural lattice position) to another disposition. The latter, instead, stems from the net work needed to create an interface with a specific density of adatoms.
We first prescribe the energy of regular configurations, and will then obtain that of more irregular configurations by relaxing the former.\\

We model the substrate as the set $\{(x,y)\in\R^2 \,:\, y\leq 0\}$. 
We consider a portion of the deposited film in a region $(a,b)\times\{y\geq 0\}$. To describe the free profile of the film, let $h:(a,b)\to\R$ be a non-negative Lipschitz function. Consider its graph
\begin{align}\label{gammah}
	\Gamma_h\coloneqq \big\{ \big(x,h(x)\big) \,:\, x\in(a,b) \big\},
\end{align} 
and its sub-graph (see Figure \ref{fig:admissible} on the left)
\begin{align}\label{omegah}
	\Omega_h\coloneqq\{(x,y)\in\R^2 \,:\, x\in(a,b),\, y<h(x)\}.
\end{align}
The set $\Omega_h\cap\{y\geq0\}$ represents the deposited film. We first introduce the surface energy. 
The adatom density will be described by a positive function $u\in L^1(\ho\llcorner\Gamma_h)$. \blue{The surface energy corresponding to such an adatom density distribution} will be
\[
\int_{\Gamma_h} \psi\big(u(\textbf{x})\big) \dho(\textbf{x}),
\]
where with $\textbf{x}$ we denote a point in $\R^2$, and $\psi:[0,+\infty)\to(0,+\infty)$ is a Borel function such that
\begin{equation}\label{eq:lower_bound_psi}
	\inf_{s\geq 0}  \psi(s) > 0.
\end{equation}
Note that such a requirement has the physical interpretation that no matter what the adatom density is, there is always an amount of energy needed to construct a profile.\\

We now discuss the elastic energy.
For each macroscopic configuration $\Omega_h$, there are several arrangements of atoms inside the thin film that produce that same profile. To each of these arrangements there is an elastic energy associated to: this energy will depend on the displacement between the actual position of each atom and its position in the natural crystal lattice
This displacement will be described by a function $v:\Omega_h\to\R^2$, and we assume it to be of class $W^{1,2}(\Omega_h;\R^2)$.
The natural crystal configuration of the crystalline substrate and that of the deposited film are represented by a function $E_0:\R\to \R^{2\times 2}$, defined as
\begin{align*}
	E_0(y)\coloneqq\begin{cases}
		t e_1 \otimes e_1\quad&\text{if}\ y\geq0 ,\\[5pt]
		0\quad&\text{if}\ y<0.
	\end{cases}
\end{align*}
Here, $t>0$ is a constant depending on the lattice of the substrate, and $\{e_1,e_2\}$ is the canonical basis of $\R^2$.
The crystalline structure of the film and the substrate might be slightly different, but we assume their difference to be very small, namely $|t|\ll1$.
This assumption allows us to work in the framework of linearized elasticity. 
In particular, the relevant object needed to compute the elastic energy is the symmetric gradient of the displacement
\begin{align*}
	E(v)\coloneqq\frac{1}{2}(\nabla v + \nabla^\top v),
\end{align*}
where $\nabla^\top v$ is the transpose of the matrix $\nabla v$.
Note that $E(v)$ is zero if $\nabla v$ is zero for any anti-symmetric matrix (for instance, a rotation matrix).

Finally, we assume that the substrate and the film share similar elastic properties, so they are described by the same positive definite elasticity tensor $\mathbb{C}$. The elastic energy density will be given by a function $W:\R^{2\times 2}\to \R$ defined as
\begin{align*}
	W(A)\coloneqq\frac{1}{2} A\cdot \mathbb{C}  [A]
	= \frac{1}{2}\sum_{i,j,m,n=1}^2 c_{ijnm} a_{ij}a_{nm},
\end{align*}
for a $2\times 2$ matrix $A=(a_{ij})_{i,j=1}^2$.
The elastic energy will then be
\begin{align*}
	\int_{\Omega_h} W\big( E(v(\textbf{x}))-E_0(y) \big)\ \text{d}\textbf{x}.
\end{align*}
Therefore, the energy of a regular configuration that we consider is given by
\begin{align}\label{fucntionalunrelaxed}
	\f(\Omega_h,v,u)\coloneqq \int_{\Omega_h} W\big( E(v(\textbf{x}))-E_0(y) \big)\ \text{d}\textbf{x}
	+ \int_{\Gamma_h} \psi\big(u(\textbf{x})\big)\dho(\textbf{x}),
\end{align}
where $h:(a,b)\to\R$ is a non-negative Lipschitz function, $u\in L^1(\ho\llcorner\Gamma_h)$, and $v\in W^{1,2}(\Omega_h;\R^2)$.
In the following, we will refer to such triples as regular admissible configurations, and we will denote it by the class $\mathcal{A}_r$ (see Definition \ref{def:reg_admiss}).


\subsection{The main result}

In order to study the relaxation of the energy $\mathcal{F}$, we need to first discuss what topology to use. This will determine the types of limiting configurations to expect, and how these effect the value of the effective energy. \blue{Here we justify the definition of the topology we use, that will be stated precisely in Definition \ref{convergence}.}

\begin{figure}
	\begin{center}
		\includegraphics[scale=1.1]{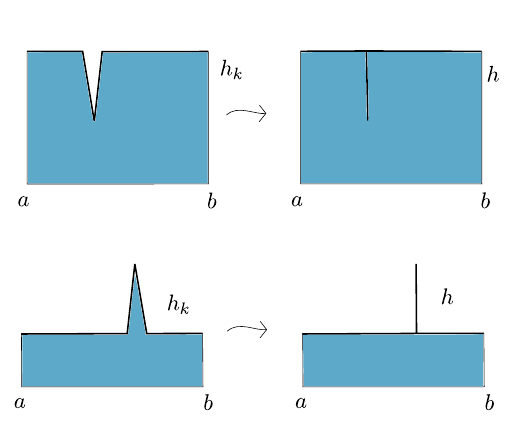}
		\caption{Two ways that a sequence of graphs can \emph{close up}: on the top by giving rise to a crack inside $\Omega_h$, while on the bottom to a crack outside $\Omega_h$. We want a topology that \emph{sees} the crack in the former case, but not in the latter.}
		\label{fig:cuts}
	\end{center}
\end{figure}

We first consider the notion of convergence for the profiles of the film.
This will be the same used in \cite{FonFusLeoMor07}. Here we give the heuristics for such a choice. There are several mechanisms that a film can use to release elastic energy. Our model allows for three of these: rearrangement of atoms inside the film, corrugation of the surface, and creation of cracks.
The topology on the profile will be concerned only with the last two.
How can a crack form? There are two mechanisms: as a fracture inside the film, or when the free profile becomes vertical, like it is depicted in Figure \ref{fig:cuts} on the top. We choose to model situations where only the latter is allowed.
Note that this forces cracks to be vertical segments touching the free profile.
What we want to avoid are configurations where cracks happen outside of the film (Figure \ref{fig:cuts} on the bottom). Thus, we need to differentiate the two situations. The right way to do it is by considering the Hausdorff convergence of the complement of the sub-graphs (the so called Hausdorff-complement topology).
We note that, in the latter case, the sets $\R^2\setminus\Omega_{h_k}$ will converge to the limiting configuration $\R^2\setminus\Omega_h$ where there is no vertical cut (see Figure \ref{fig:cuts} on the bottom).
This topology also accommodates for corrugation of the profile.\\

We now consider the convergence of the displacements. Since the energy has quadratic growth in the symmetric gradient of the displacement, the natural topology will be the weak $W^{1,2}$ topology. In particular, in order to take care of the fact that the displacements are defined in different domains (the subgraphs of the profiles), we take advantage of the fact that the complement of these latter are converging in the Hausdorff sense. Thus, local convergence in the final domain will do the job.\\

Finally, we discuss the topology for the adatom density.
In \cite{CarCriDie} the idea was to see the adatom density as a Radon measure $\mu$ concentrated on the graph describing the profile. Namely, for each $u\in L^1(\Gamma_h)$, we consider
\[
\mu\coloneqq u\ho  \llcorner \Gamma_h.
\]
This identification allows not only to consider concentration of measures, but it \blue{turns out to be the right way to model adatoms in order to} exploit the interplay between oscillations of the profile and change in adatom density. Thus, for the adatom density, the weak${}^*$ convergence of measures will be used.\\

The question we now have to address is what are the possible limiting objects that we need to consider. This is a discussion of compactness of sequences $(\Omega_{h_k},v_k, \mu_k)_k$ with uniformly bounded energy, namely such that
\[
\sup_{k\in\N} \mathcal{F}(\Omega_{h_k},v_k, \mu_k) < +\infty,
\]
We start by investigating the convergence of graphs, and the others will follow. 
Thanks to the lower bound \eqref{eq:lower_bound_psi} on the energy density $\psi$, the energy $\mathcal{F}$ is lower bounded by the length of the graph of $h_k$. Indeed, there exists $c>0$ such that
\[
\sup_{k\in\N} c \ho(\Gamma_k) \leq \sup_{k\in\N}\int_{\Gamma_{h_k}} \psi(u_k)\dho <+\infty
\]
which in turn is a lower bound on the total variation of $h_k$:
\[
\sup_{k\in\N} \ho(\Gamma_{h_k}) = \sup_{k\in\N} \int_a^b \sqrt{1+|h'_k|^2}\ \text{d}x
\geq \int_a^b |h'_k|\ \text{d}x.
\]
Thus, if a mass constraint on the area of $\Omega_k$, or a Dirichlet boundary condition at $a$ and $b$ are imposed, we get that the limiting configuration will be the sub-graph of a function $h:(a,b)\to[0,+\infty)$ of bounded variation.
In particular, since we are in the one dimensional case, such a function will have countably many jumps and countably many cuts. \\

Now, we consider the convergence of the displacement. Due to the choice of the topology, the limiting displacement will be a function $v\in W^{1,2}(\Omega_h;\R^2)$.
Note that one of the technical advantages of working in dimension two is that we can avoid having to rely on functions of bounded deformation, and use instead Sobolev functions and the free profile to describe cracks.\\

Finally, let us discuss the adatom densities. Each of them is seen as the Radon measure $u_k\ho\llcorner\Gamma_{h_k}$.
By imposing a mass constraint on the total amount of adatoms, we have that their total variation is bounded, and thus they converge (up to a subsequence), to a Radon measure $\mu$.
Noting that each $\mu_k$ is supported on the graph $\Gamma_{h_k}$, and these latter also converge in the Hausdorff sense to the graph of the limiting profile $h$, the limiting measure $\mu$ will be supported on $\Gamma_h$.\\

Therefore, the class $\mathcal{A}$ of limiting admissible configurations we will need to consider is given by the triples $(\Omega_h,v,\mu)$, where $h\in \bv(a,b)$, $v\in W^{1,2}(\Omega_h;\R^2)$, and $\mu$ is a Radon measure supported on $\Gamma_h$. Moreover, we denote by $\Gamma^c_h$ the cuts of $h$, and by $\widetilde{\Gamma}_h$ the rest of the extended graph of $h$, namely regular part and jumps (see  Figure \ref{fig:admissible} on the right, and Definition \ref{cuts} for the precise definition). \\

\blue{Thus, in light of the above discussion, given a sequence $(\Omega_{h_k},v_k,\mu_k)_k\subset\mathcal{A}_r$, we will say that $(\Omega_{h_k},v_k,\mu_k)\to (\o_h,v,\mu)\in\mathcal{A}$ if
	\begin{enumerate}
		\item[$(i)$] $\R^2\setminus\Omega_{h_k}\stackrel{H}{\rightarrow} \R^2\setminus\Omega_h$ in the Hausdorff convergence of sets;
		\item[$(ii)$] $v_k\rightharpoonup v$ weakly in $W^{1,2}_{\mathrm{loc}}(\Omega_h;\R^2)$,
		\item[$(iii)$] $\mu_k\wtom\mu$ weakly$^\ast$ in the sense of measures;
	\end{enumerate}
	as $k\to\infty$.} \\
	
\begin{figure}
	\begin{center}
		\includegraphics[scale=0.8]{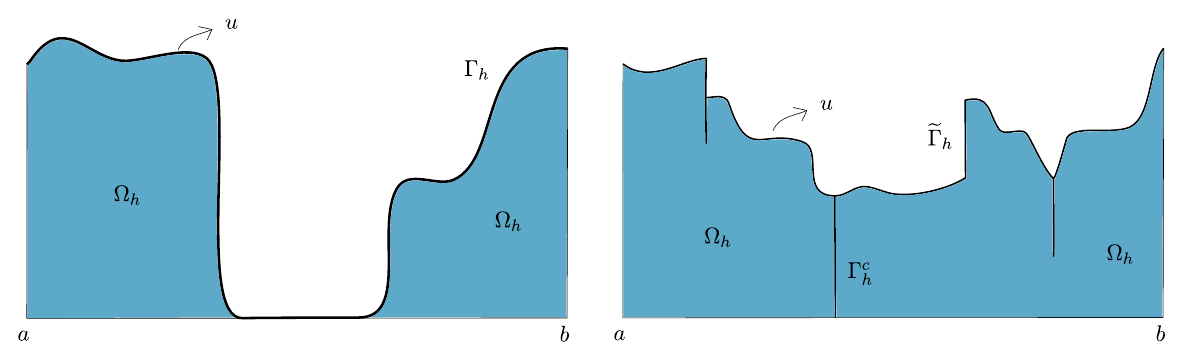}
		\caption{A regular configuration on the left, and a possible limiting configuration on the right: cracks and jumps can appear.}
		\label{fig:admissible}
	\end{center}
\end{figure}

The two main results of this paper provide representations of the relaxation of the functional $\mathcal{F}$ when a mass constraint is in force, and when it is not.

\begin{theorem}\label{thm:main_no_mass}
\blue{	Let $(\Omega_h,v,\mu)\in \mathcal{A}$, and write $\mu = u\ho\llcorner\Gamma_h +\mu^s\llcorner\Gamma_h$, 
	where $\mu^s$ is the singular part of $\mu$ with respect to $\ho\llcorner\Gamma_h$.} Then, the relaxation of the functional $\mathcal{F}$ defined in \eqref{fucntionalunrelaxed}, with respect to the above topology, is given by
\begin{align*}
	\overline{\mathcal{F}}(\Omega_h,v,\mu) = \int_{\Omega_h}
	W\big( E(v(\textbf{x}))-E_0(y) \big)\ \emph{d}\textbf{x}
	+ \int_{\widetilde{\Gamma}_h}\widetilde{\psi}\big(u(\textbf{x})\big)\dho(\textbf{x})
	+\int_{\Gamma^c_h}\psi^c\big(u(\textbf{x})\big)\dho(\textbf{x})
	+\theta\mu^s(\Gamma_h),
\end{align*}
	where $\widetilde{\psi}$ is the convex sub-additive envelope of $\psi$ (see Definition \ref{psitilde}), the function $\psi^c$ is defined as
	\[
	\psi^c(s)\coloneqq\min\{\widetilde{\psi}(r)+\widetilde{\psi}(t) \,:\, s=r+t\},
	\]
	for all $s\in[0,+\infty)$, and
	\[
	\theta \coloneqq \lim_{t\to+\infty} \frac{\widetilde{\psi}(t)}{t}
	= \lim_{t\to+\infty} \frac{\psi^c(t)}{t},
	\]
	is the common recession coefficient of $\widetilde{\psi}$ and of $\psi^c$.
\end{theorem}

\begin{theorem}\label{thm:main_mass}
	Fix $M,m>0$. Denote by $\mathcal{A}_r(m,M)$ the triples $(\Omega_h,v,\mu)\in \mathcal{A}_r$ such that 
	\[
	\int_{\Gamma_h} u(\textbf{x})  \dho(\textbf{x})=m,\quad\quad\quad
	\mathcal{L}^2\left(\Omega_h\cap\{ y\geq 0 \}\right) = M,
	\]
	and by $\mathcal{A}(m,M)$ the triples $(\Omega_h,v,\mu)\in \mathcal{A}$ such that 
	\[
	\mu(\Gamma_h)=m,\quad\quad\quad
	\mathcal{L}^2\left(\Omega_h\cap\{ y\geq 0 \}\right) = M.
	\]
	Define
	\[
	\mathcal{H}(\o_h,v,\mu)\coloneqq 
	\left\{
	\begin{array}{ll}
		\mathcal{F}(\o_h,v,\mu) & \text{ if } (\o_h,v,\mu)\in\mathcal{A}_r(m,M),\\[5pt]
		+\infty & \text{ else}.
	\end{array}
	\right.
	\]
	Then, the relaxation of $\mathcal{H}$ in the above topology is given by
	\[
	\overline{\mathcal{H}}(\o_h,v,\mu) =
	\left\{
	\begin{array}{ll}
		\mathcal{G}(\o_h,v,\mu) & \text{ if } (\o_h,v,\mu)\in\mathcal{A}(m,M),\\[5pt]
		+\infty & \text{ else},
	\end{array}
	\right.
	\]
	where $\mathcal{G}(\o_h,v,\mu)$ denotes the right-hand side of the representation formula of Theorem \ref{thm:main_no_mass}.
	Namely, the mass constraint is maintained by the relaxation procedure.
\end{theorem}

\begin{remark}
	In general, it is not possible to say more on the singular part of the measure.
\end{remark}

\begin{remark}
	The more general case, where the adatom density is vector valued (corresponding to different materials deposited on the substrate) and the surface energy is anisotropic are currently under investigation.
\end{remark}


\section{Strategy of the proof}

Now, we would like to comment on the strategy to prove the main results. First of all, in Theorem \ref{liminfinequality} we will prove the liminf inequality for the case of no mass constraint, and in Theorem \ref{limsupinequality} the limsup inequality for the case with the mass constraint. These theorems will give both Theorem \ref{thm:main_no_mass}, and Theorem \ref{thm:main_mass}.

Similarly for functional considered in \cite{FonFusLeoMor07}, the bulk and the surface terms of the energy do not interact in the relaxation process. Since the former is quite standard, we will comment on how to deal with the latter. In this lies the novelty of the paper. Our strategy relies on ideas inspired by results obtained in  \cite{CarCriDie}.
The main difference with the case treated in that paper is the graph constraint. This reflects on the fact that oscillations of the thin film profile must be in the vertical direction in order to preserve such a constraint, and that cracks can be created only in a specific way. The former term only gives technical challenges, while the latter is responsible for the different energy densities $\widetilde{\psi}$ and $\psi^c$. Despite this, note that the recession coefficients for the singular part of the measure in the two parts of the extended graph (the cuts, and the rest of the graph) agree.\\

Let us discuss the strategy for the liminf inequality for the surface terms. We avoid mentioning the fine details and focus instead on the main ideas.
Let $(h_k)_{k\in\N}$ be a sequence of Lipschitz functions such that $\R^2\setminus\Omega_{h_k}$ converge to $\R^2\setminus\Omega_h$, for some function $h$ of bounded variation.
This implies that $\Omega_{h_k}$ converges to $\Omega_h$ in $L^1$ (see Lemma  \ref{lem:Hausd_L1}).
Let $(u_k)_{k\in\N}$ the be adatom densities defined on each $\Gamma_{h_k}$, and let $\mu=u\ho\llcorner\Gamma_h+\mu^s$ be the limiting measure.
We need to prove that
\begin{equation}\label{eq:liminf_intro_1}
	\liminf_{k\to\infty} \int_{\Gamma_{h_k}} \psi\big(u_k(\textbf{x}) \big) \dho(\textbf{x})
	\geq \int_{\widetilde{\Gamma}_h}\widetilde{\psi}\big(u(\textbf{x})\big)\dho(\textbf{x})
	+ \int_{\Gamma^c_h} \psi^c\big(u(\textbf{x})\big)\dho(\textbf{x}) 
	+\theta\mu^s(\Gamma_h).
\end{equation}
The idea is to separate the contribution that the energy on the left-hand side has
on a neighborhood of each cut of $h$, and on the other part of the graph of $h$.
Despite there might be a countable number of cuts, it is just a technicality to show that we can reduce to finitely many of them (see the beginning of the proof of Theorem \ref{liminfinequality}).
\blue{Thus, let us assume that the final configuration described by $h$ has finitely many cuts. Since the energy is local, for the sake of simplicity, we will consider the case where there is one single cut. In case the measure $\mu$ has a Dirac delta at the point $P$ (see Figure \ref{fig:strategy_intro}), we want to count its contribution to the energy as part of the energy of the regular part of $\widetilde{\Gamma}_h$. For this reason, we take $\varepsilon>0$ and consider a rectangle $R_\varepsilon$ around the cut as in Figure \ref{fig:strategy_intro}. }

\begin{figure}
	\begin{center}
		\includegraphics[scale=1.1]{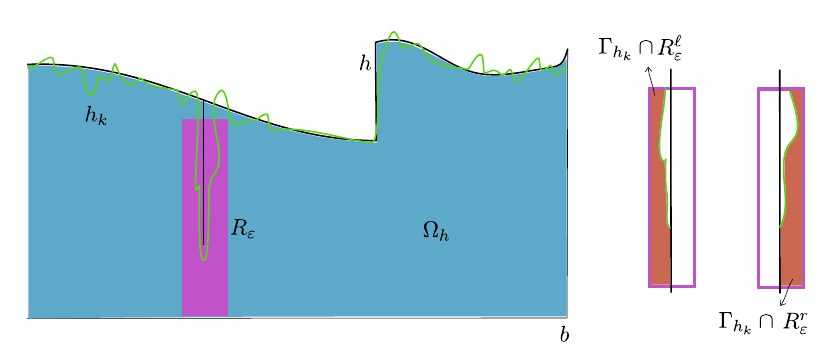}
		\caption{In order to get the liminf inequality, we separate the effects on a neighborhood $R_\varepsilon$ of the cut, and outside of it.}
		\label{fig:strategy_intro}
	\end{center}
\end{figure}

Now, we claim that
\begin{equation}\label{eq:liminf_intro_2}
	\liminf_{k\to\infty} \int_{\Gamma_{h_k}\setminus R_\varepsilon} \psi\big(u_k(\textbf{x})\big) \dho(\textbf{x})
	\geq \int_{\widetilde{\Gamma}_h\setminus R_\varepsilon}\widetilde{\psi}\big(u(\textbf{x})\big)\dho(\textbf{x})
	+\theta\mu^s(\widetilde{\Gamma}_h\setminus R_\varepsilon),
\end{equation}
and that
\begin{equation}\label{eq:liminf_intro_3}
	\liminf_{k\to\infty} \int_{\Gamma_{h_k}\cap R_\varepsilon} \psi\big(u_k(\textbf{x})\big) \dho(\textbf{x})
	\geq \int_{\Gamma^c_h\setminus R_\varepsilon} \psi^c\big(u(\textbf{x})\big)\dho(\textbf{x}) + \theta\mu^s(\Gamma^c_h\cap R_\varepsilon).
\end{equation}
Given \eqref{eq:liminf_intro_2} and \eqref{eq:liminf_intro_3}, we obtain the desired liminf inequality \eqref{eq:liminf_intro_1} by sending $\varepsilon$ to zero.

To obtain both  \eqref{eq:liminf_intro_2} and \eqref{eq:liminf_intro_3}, we rely on (a localized version of) the lower semicontinuity result proved in \cite[Theorem 5]{CarCriDie} (see Theorem \ref{cc}).
In the first case, the idea is to view the  graph of each $h_k$, and the regular and the jump part of extended graph of $h$ as ($\ho$-equivalent to) the reduced boundaries of the corresponding epigraphs.

For \eqref{eq:liminf_intro_3}, we instead have to consider the contributions of the surface energy from both sides of the crack. Therefore, we reason as follows: the rectangle $R_\varepsilon$ in Figure \ref{fig:strategy_intro} is split by the vertical line passing through the crack in two parts, one on the left and one on the right. Call them $R_\varepsilon^\ell$, and $R_\varepsilon^r$, respectively.
Then, we consider the sets $\Omega_{h_k}\cap R_\varepsilon^\ell$ and $\Omega_{h_k}\cap R_\varepsilon^r$.
Since $\R^2\setminus \Omega_{h_k}\to \R^2\setminus\Omega_h$ in the Hausdorff topology, they converge in $L^1$ to $R_\varepsilon^\ell$, and $R_\varepsilon^r$, respectively. Moreover, it holds
\[
u_k\ho\llcorner (\Gamma_{h_k}\cap R_\varepsilon^\ell)\wtom \mu^\ell = u^\ell\llcorner(\Gamma^c_h\cap R_\varepsilon) + \mu^\ell_s,
\]
\[
u_k\ho\llcorner (\Gamma_{h_k}\cap R_\varepsilon^r)\wtom \mu^r = u^r\llcorner(\Gamma^c_h\cap R_\varepsilon) + \mu^r_s.
\]
Thus, thanks to the lower semicontinuity result (see Theorem \ref{cc}), we get that
\[
\liminf_{k\to\infty} \int_{\Gamma_{h_k}\cap R_\varepsilon^\ell} \psi\big(u_k(\textbf{x})\big) \dho(\textbf{x})
\geq \int_{\Gamma^c_h\setminus R_\varepsilon} \widetilde{\psi}\big(u^\ell(\textbf{x})\big)\dho(\textbf{x}) + \theta\mu^\ell_s(\Gamma^c_h\cap R_\varepsilon)
\]
and
\[
\liminf_{k\to\infty} \int_{\Gamma_{h_k}\cap R_\varepsilon^r} \psi\big(u_k(\textbf{x})\big) \dho(\textbf{x})
\geq \int_{\Gamma^c_h\setminus R_\varepsilon} \widetilde{\psi}\big(u^r(\textbf{x})\big)\dho(\textbf{x}) + \theta\mu^r_s(\Gamma^c_h\cap R_\varepsilon).
\]
We then show that $u^\ell+u^r = u$, and $\mu^\ell_s+\mu^r_s=\mu_s$. Thus, by definition of $\psi^c$, we obtain
\[
\psi^c\big(u(\textbf{x})\big) \leq  \widetilde{\psi}\big(u^r(\textbf{x})\big)
+  \widetilde{\psi}\big(u^\ell(\textbf{x})\big).
\]
This gives \eqref{eq:liminf_intro_3}, and, in turn, the desired liminf inequality for the surface energy.\\

We now discuss the strategy for the limsup inequality for the surface energy. This is more involved, and requires several steps.
The idea is to reduce to the situation where the limiting profile $h$ is Lipschitz, and the adatom measure $\mu$ is a piecewise constant density (more precisely, it is possible to find a square grid where the density has the same value on each of the parts of the graph inside each of these squares).
In such a case, in Proposition \ref{proposition3} we construct a sequence $(\o_{h_k},v_k,\mu_k)_k$ that satisfies the mass constraints such that
\begin{equation}\label{eq:conv_energy_limsup1_intro}
	\limsup_{k\to\infty} \int_{\widetilde{\Gamma}_{h_k}} \psi\big(u_k(\textbf{x})\big)\dho(\textbf{x}) \leq \int_{\widetilde{\Gamma}_{h}}\widetilde \psi\big(u(\textbf{x})\big)\dho(\textbf{x}).
\end{equation}
Without loss of generality (see Lemma \ref{lem:char_psi_s0}), we can assume $\psi$ to be convex. Then, $\psi$ and $\widetilde{\psi}$ agree on $[0,s_0)$, for some $s_0\in(0,+\infty]$. In particular, if $s_0<+\infty$ the function $\widetilde{\psi}$ is linear on $(s_0,+\infty)$ (see Lemma \ref{lem:char_psi_s0})). Thus, in squares where $u\leq s_0$, we define $h_k$ as $h$ and $u_k$ as $u$. We just have to care about those squares $Q$ where $u>s_0$. The energy in such a square is $\widetilde{\psi}(u)\ho(\Gamma_h\cap Q)$. The idea is to write
\[
\widetilde{\psi}(u)\ho(\Gamma_h\cap Q)
= \widetilde{\psi}(r s_0)\ho(\Gamma_h\cap Q)
= r \widetilde{\psi}(s_0)\ho(\Gamma_h\cap Q)
= \psi(s_0) \left[ r\ho(\Gamma_h\cap Q)\right],
\]
for some $r>1$, where in the last step we used the fact that $\psi(s_0)=\widetilde{\psi}(s_0)$. Then, we want to obtain the quantity $r\ho(\Gamma_h\cap Q)$ as the length of an oscillating profile $h_k$ in $Q$, and define $u_k$ as $s_0$. This ensures the validity of \eqref{eq:conv_energy_limsup1_intro}.
Such a construction is done in Proposition \ref{step3limsup}, where we prove an extension of the so called \emph{wriggling lemma} (see \cite[Lemma 4]{CarCriDie}). Namely, given a Lipschitz function $f:(a,b)\to[0,+\infty)$, and a number $r>1$, there exists a sequence of graphs $f_n:(a,b)\to[0,+\infty)$ with $\ho\llcorner\Gamma_{f_n}\wtom r\ho\llcorner\Gamma_{f}$ as $n\to\infty$, such that
\[
\ho(\Gamma_{f_n}) = r \ho(\Gamma_f),
\]
and $f_n(a)=f(a)$, $f_n(b)=f(b)$, for each $n\in\N$, and satisfying other technical properties (see Proposition \ref{step3limsup} for the precise statement).
What the above inequality is using is a quantitative lack of lower semicontinuity of the perimeter. The difference with the result in \cite[Lemma 4]{CarCriDie} is that only vertical oscillations are allowed. Moreover, we also fill in details that were not fully explained in that paper.
Note that in our case, there is an additional technical difficulty to be faced: ensuring that both mass constraints are satisfied by each $(\o_{h_k},v_k,\mu_k)$ will be achieved by carefully modifying both the profile and the density. Note that modifications of the graphs have to be done in such a way that the profile is always non-negative.

In order to reduce from a general profile $(\o_h,v,\mu)\in\mathcal{A}(m,M)$ to the above case, we argue as follows.
First of all, by using averages, we prove that it suffices to consider the situation where the adatom measure $\mu$ is a piecewise constant function (see Proposition \ref{step1limsup}).
Then, we need to approximate a general profile $h\in \bv(a,b)$ with a sequence of Lipschitz profiles $(h_k)_{k\in\N}$, and corresponding piecewise constant adatom densities $(u_k)_{k\in\N}$, in such a way that
\begin{align}\label{eq:conv_en_intro}
	\lim_{k\to\infty} \int_{\widetilde{\Gamma}_{h_k}}\widetilde{\psi}\big(u_k(\textbf{x})\big)\dho(\textbf{x})
	&+\int_{\Gamma^c_{h_k}}\psi^c\big(u_k(\textbf{x})\big)\dho(\textbf{x}) \nonumber \\
	&= \int_{\widetilde{\Gamma}_{h}}\widetilde{\psi}\big(u(\textbf{x})\big)\dho(\textbf{x})
	+\int_{\Gamma^c_{h}}\psi^c\big(u(\textbf{x})\big)\dho(\textbf{x}).
\end{align}
This is done in Proposition \ref{propfinitenumbercuts}. In order to obtain the approximation of the profiles, we employ an idea by Bonnettier and Chambolle in Section 5.2 of \cite{BonCha02}, later adapted to the case of graphs in \cite[Lemma 2.7]{FonFusLeoMor07}: to use the Moreau-Yosida transform to define a Lipschitz approximation of $h$ to the left and to the right of each cut (again, we are reducing to the case of finitely many of them). To also approximate the cracks, we use a linear interpolation. As for defining the adatom density on the graph of $h_k$, we exploit the fact that the Hausdorff convergence of $\R^2\setminus\Omega_{h_k}$ to $\R^2\setminus\Omega_h$ implies that the graphs $(h_k)_{k\in\N}$ are converging in the Hausdorff topology to $h$. In particular, for $k$ large enough, the graphs of the $h_k$'s will be inside the same squares where the graph of $h$ is. This allows to define $u_k$ on the part of the graph of $h_k$ inside a square, as the value that $u$ has inside that square.
Then, the convergence of the energy required in \eqref{eq:conv_en_intro} is ensured since the length of the graph of $h_k$ inside each cube converges to the length of $h$ inside the same cube.


\section{Preliminaries}

We here introduce the main definition and basic results that will be used throughout the paper.


\subsection{Function of (pointwise) bounded variation in one dimension}

We start with functions of (pointwise) bounded variation in one dimension. A comprehensive treatment of this topic can be found in the book \cite{Leoni_Book} by Leoni.

\begin{definition}
	Let $h:(a,b)\to\R$. We say that $h$ is a function of \emph{pointwise bounded variation} in $(a,b)$ if $\mathrm{Var}(h)<+\infty$, where
	\[
	\mathrm{Var}(h) \coloneqq \sup\left\{ \sum_{i=1}^k |h(x_i) - h(x_{i-1})|  \right\},
	\]
	where the supremum is taken over all finite partitions of $(a,b)$. In this case, we write $h\in \text{BVP}(a,b)$.
\end{definition}

The main properties of functions of pointwise bounded variations that will be used in the paper are collected in the following result (see \cite[Theorem 2.17, Theorem 2.36]{Leoni_Book}).

\begin{theorem}\label{thm:properties_PBV}
	Let $h\in \emph{BPV}(a,b)$. Then, the limits
	\[
	h(x^-)\coloneqq \lim_{y\to x^-} h(y),\quad\quad\quad 
	h(x^+)\coloneqq \lim_{y\to x^+} h(y),
	\]
	exist for all $x\in(a,b)$. In particular, if we define the functions
	\[
	h^-(x)\coloneqq\min\{h(x^+),h(x^-)\} , \quad\quad\quad
	h^+(x)\coloneqq\max\{h(x^+),h(x^-)\},
	\]
	we have that there are at most countably many points $x\in(a,b)$ for which $h^-(x)$, $h^+(x)$ and $h(x)$ do not agree.
	Finally, $h$ admits a lower semi-continuous representative.
\end{theorem}

We now connect functions of pointwise bounded variation with those of bounded variation.

\begin{definition}
	Let $u\in L^1(a,b)$. We say that $u$ has \emph{bounded variation} in $(a,b)$ if there exists a Radon measure $\mu$ such that
	\[
	\int_a^b u\varphi'\, \text{d}x = - \int_{(a,b)} \varphi\, \text{d}\mu,
	\]
	for all $\varphi\in C^1_c(a,b)$. In this case, we write $u\in \bv(a,b)$, and we denote the measure $\mu$ by $\text{D}u$.
\end{definition}

The relation between functions of pointwise bounded variation and functions of bounded variation is given by the following result (see \cite[Theorem 7.3]{Leoni_Book}).

\begin{theorem}
	Let $u\in \emph{BV}(a,b)$. Then, there exists a right-continuous function $h\in \emph{BVP}(a,b)$ with $u(x)=h(x)$ for a.e. $x\in(a,b)$ such that $\mathrm{Var}(h)=|\emph{D}u|(a,b)$.
\end{theorem}

Finally, we recall that the subgraph of a function of bounded variation is a set of finite perimeter (see \cite[Theorem 14.6]{Giusti_book}), and that its reduced boundary coincides with the non cut part of the extended graph (see \cite[Theorem 4.5.9 (3)]{Federer_book}.

\begin{lemma}\label{lem:subgraph_finite_per}
	Let $h\in\bv(a,b)$. Then, the epigraph $\Omega_h$ has finite perimeter in $(a,b)\times\R$, and
	\[
	\ho( \widetilde{\Gamma}_h \,\triangle\, \partial^*\Omega_h ) = 0,
	\]
	where $\partial^*\Omega_h$ is the reduced boundary of $\Omega_h$.
\end{lemma}


\subsection{Hausdorff convergence}

We now introduce the Hausdorff metric.  

\begin{definition}
	Let $E,F\subset\R^N$. We define
	\[
	\mathrm{d}_H(E,F)\coloneqq \inf\{ r>0 \,:\, E\subset F_r,\, F\subset E_r \},
	\]
	where, for $A\subset\R^N$ and $r>0$, we set $A_r\coloneqq \{ x+y : x\in A, y\in B_r(0) \}$.
	Moreover, we say that a sequence of sets $(E_k)_k$ with $E_k\subset\R^N$ \emph{Hausdorff converges} to a set $E\subset\R^N$, and we write $E_k\stackrel{H}{\rightarrow} E$, if $\mathrm{d}_H(E_k,E)\to0$ as $k\to\infty$.
\end{definition}

In order for the Hausdorff distance to actually be a distance, we need to work with compact sets. This will also give compactness of the metric space. This latter fact is known as Blaschke Theorem (see \cite[Theorem 6.1]{AFP}).

\begin{theorem}[Blaschke Theorem]
	The family of compact sets of $\R^N$ endowed with the Hausdorff distance is a compact metric space.
\end{theorem}

The convergence of epigraphs in the Hausdorff-complement topology we use implies their $L^1$ convergence, as it was shown in \cite[Lemma 2.5]{FonFusLeoMor07}.

\begin{lemma}\label{lem:Hausd_L1}
	Let $(h_k)_k\subset\bv(a,b)$ be a sequence of lower semi-continuous functions such that
	\[
	\sup_{k\in\N}|\emph{D}h_k|(a,b)<+\infty,\quad\quad\quad
	\R^2\setminus \Omega_{h_k} \stackrel{H}{\rightarrow} \R^2\setminus A,
	\]
	for some open set $A\subset\R^2$. Then, there exists $h\in\bv(a,b)$ such that $A=\Omega_h$, $h_k\to h$ in $L^1$. Moreover, $\Omega_{h_k}\to\Omega_h$ in $L^1$.
\end{lemma}

We now relate the Hausdorff metric with the notion of Kuratowski convergence (see \cite[Theorem 6.1]{AFP}).

\begin{proposition}\label{kuratowskyconvergence}
	Let $(E_k)_k$, with $E_k\subset\R^2$, and let $E\subset\R^2$. Then, $E_k\stackrel{H}{\rightarrow} E$ if and only if the followings hold:
	\begin{itemize}
		\item[(i)] Any cluster point of a sequence $(x_k)_k$, with $x_k\in E_k$, belongs to $E$;
		\item[(ii)] For any $x\in E$, there exists $(x_k)_k$, with $x_k\in E_k$, such that $x_k\to x$.
	\end{itemize}
	These equivalent properties are those defining the so called Kuratowski convergence.
\end{proposition}


\subsection{On the surface energy}
\blue{Here we introduce all the notation and recall the result that are needed to treat the surface term.
\begin{definition}\label{sub}
	A function $\psi:[0,+\infty)\to\R$ is said to be \emph{sub-additive} if
	\begin{align*}
		\psi(s+t)\leq\psi(s)+\psi(t),
	\end{align*}
	for any $s,t\geq0$.
\end{definition}

\begin{definition}\label{psitilde}
	Let $\psi:[0,+\infty)\to\R$. 
	The \emph{convex sub-additive envelope} of $\psi$ is the function
	$\widetilde{\psi}:[0,+\infty)\to\R$ defined as
	\begin{align*}
		\widetilde{\psi}(s)\coloneqq\sup\{f(s):f:[0,+\infty) \to\R\ \emph{is convex, sub-additive and}\ f\leq\psi\}.
	\end{align*} 
	for all $s\in[0,+\infty)$.
\end{definition}

\begin{remark}
	Note that $\widetilde{\psi}$ is the greatest convex and sub-additive function that is no greater than $\psi$.
\end{remark} 

\begin{definition}\label{psic}
	Let $\psi:[0,+\infty)\to\R$.
	We define the function $\psi^c:[0,+\infty)\to\R$ as
	\[
	\psi^c(s)\coloneqq\min\{\widetilde{\psi}(r)+\widetilde{\psi}(t) \,:\, s=r+t\},
	\]
	for all $s\in[0,+\infty)$.
\end{definition}

\begin{remark}
	It is easy to see that the function $\psi^c$ is well defined. Indeed, fix $s\geq0$.
	Since $\psi$ is defined only for non-negative real numbers, by compactness there exist $a,b\geq 0$ with $s=a+b$ such that
	\[
	\psi^c(s) = \widetilde{\psi}(a)+\widetilde{\psi}(b).
	\]
	Moreover, note that $\psi^c(0)=2\widetilde{\psi}(0)$. This is consistent with the result obtained in \cite{FonFusLeoMor07}, where they consider the case $\psi\equiv1$. We will prove in Lemma \ref{lem:psic_sa} that $\psi^c$ is convex and sub-additive.
\end{remark}
}

We now recall two results on the surface energy. The first is a combination of \cite[Lemma A.11]{CarCriDie} and \cite[Lemma 2.2]{CarCri}.

\begin{definition}
	Let $\psi:\R\to\R$. We define its \emph{convex envelope} $\psi^{\text{cvx}}:\R\to\R$ as
	\[
	\psi^{\text{cvx}}(x)\coloneqq\sup\{\rho(x) : \rho\ \text{is convex and}\ \rho\leq\psi\},
	\]
	for all $x\in\R$.
\end{definition}

\begin{lemma}\label{lem:char_psi_s0}
	Let $\psi:[0,+\infty)\to(0,+\infty)$. Then
	\[
	\widetilde{\psi}=\widetilde{\psi^\text{cvx}}.
	\]
	Namely, in order to compute the convex sub-additive envelope of $\psi$, we can assume, without loss of generality, that $\psi$ is convex.
	
	Moreover, assume $\psi$ to be convex. Then, there exists $s_0\in(0,+\infty]$ such that
	\begin{align*}
		\widetilde{\psi}(s)=\begin{cases}
			\psi(s)\quad & s\leq s_0, \\[5pt]
			\theta s \quad& s> s_0,
		\end{cases}
	\end{align*}
\blue{	for some $\theta>0$.}
\end{lemma}

\begin{remark}
	Note that, if $\psi$ is differentiable at $s_0$, then $\theta=\psi'(s_0)$. In particular, if $s_0<+\infty$, it holds that $\widetilde{\psi}$ is linear in $[s_0,+\infty)$.
\end{remark}

The following result proved in \cite[Theorem 3]{CarCriDie} gives a lower bound for the surface energy.

\begin{theorem}\label{cc}
	Let $E\subset \R^N$ be a set of finite perimeter and $\mu$ be a Radon measure supported on $\partial E$.
	Let $A\subset \R^N$ be an open set with $\mu(\partial A)=0$.
	Let $(E_k)_{k\in\N}\subset \R^N$ be a sequence of sets of finite perimeter, and let $(u_k)_{k\in\N}$, with $u_k\in L^1(\partial E_k)$ be such that
	\begin{enumerate}
		\item[$(i)$] $E_k\cap A\to E\cap A$ in $L^1(\R^N)$;
		\item[$(ii)$] $u_k\ho\llcorner (\partial^\ast E_k\cap A)\wtom \mu\llcorner A$.
	\end{enumerate}
	Then,
	\begin{align*}
		\liminf_{k\to\infty}\int_{\partial^\ast E_k\cap A}\psi(u_k)\dho
		\geq \int_{\partial^\ast E\cap A} \widetilde{\psi}(u)\dho
		+ \theta\mu^s(A),
	\end{align*}
	where $\widetilde{\psi}$ is as in Definition \ref{psitilde}.
\end{theorem}


\section{Setting}

In this section we give the rigorous definitions of the objects discussed in the introduction.
We start with the set of admissible configurations.

\begin{definition}\label{def:reg_admiss}
	Let $\Omega\subset\R^2$, $v\in W^{1,2}(\Omega;\R^2)$, and $\mu$ be a Radon measure in $\R^2$.
	We say that the triple $(\Omega,v,\mu)$ is an \emph{admissible regular configurations} if there exists a Lipschitz function $h:(a,b)\to[0,+\infty)$ such that
	\[
	\Omega = \Omega_h,\quad\quad\quad \mu = u\ho\llcorner\Gamma_h,
	\]
	for some $u\in L^1(\Gamma_h)$. We denote by $\mathcal{A}_r$ the family of all admissible regular configurations.
\end{definition}
\blue{\begin{definition}\label{ArmM}
 Fix $m,M>0$. We denote by $\mathcal{A}_r(m,M)$ the triples $(\o,v,\mu)\in \mathcal{A}_r$ such that 
\[
\int_{\Gamma_h} u(\textbf{x})  \dho(\textbf{x})=m,\quad\quad\quad
\mathcal{L}^2\left(\Omega\cap\{ y\geq 0 \}\right) = M.
\]
\end{definition}}

We now define the energy for regular configurations.

\begin{definition}\label{regfunctional}
	Next, we introduce the energy for regular configurations. We define $\f:\mathcal{A}_r\to\R$ as
	\begin{align*}
		\f(\Omega,v,\mu)\coloneqq \int_\Omega W\big( E(v(\textbf{x}))-E_0(y) \big)\ \text{d}\textbf{x}
		+ \int_\Gamma \psi\big(u(\textbf{x})\big) \dho(\textbf{x}),
	\end{align*}
	for every $(\o,v,\mu)\in\mathcal{A}_r$.
\end{definition}

We now introduce the more general configurations that will be treated.

\begin{definition}\label{def:admiss}
	Let $\Omega\subset\R^2$, $v\in W^{1,2}(\Omega;\R^2)$, and $\mu$ be a Radon measure in $\R^2$.
	We say that the triple $(\Omega,v,\mu)$ is an \emph{admissible configurations} if there exists a function $h\in \bv(a,b)$ with $h\geq 0$ such that
	\[
	\Omega = \Omega_h,\quad\quad\quad \mu = u\ho\llcorner\Gamma_h +\mu^s\llcorner\Gamma_h,
	\]
	where $\mu^s$ is the singular part of $\mu$ with respect to $\ho\llcorner\Gamma_h$.
	We denote by $\mathcal{A}$ the family of all admissible configurations.
\end{definition}
\blue{\begin{definition}\label{AmM}
	Fix $m,M>0$. We denote by $\mathcal{A}(m,M)$ the triples $(\o,v,\mu)\in \mathcal{A}$ such that 
	\[
\mu(\Gamma_h)=m,\quad\quad\quad
	\mathcal{L}^2\left(\Omega\cap\{ y\geq 0 \}\right) = M.
	\]
\end{definition}}

In order to define the relaxed energy, we need to introduce some notation.

\begin{remark}
	In Theorem \ref{thm:properties_PBV} we introduced the functions $h^\pm$. Note that
	\[
	h^-(x)=\liminf_{y\to x} h(y),\quad\quad\quad h^+(x)=\limsup_{y\to x} h(y).
	\]
	In particular, if $x\in(a,b)$ is a point of continuity for $h$, then $h^-(x)=h^+(x)=h(x)$.
\end{remark}

\begin{definition}\label{cuts}
	Let $h\in \bv(a,b)$. We call
	\[
	\Gamma_h \coloneqq \left\{(x,y)\in\R^2 \,:\, x\in(a,b),\, h(x)\leq y \leq h^+(x) \,\right\}
	\]
	the \emph{extended graph} of $h$.
	Moreover, we define: \vspace{0.2cm}
	\begin{itemize}
		\item The \emph{jump part} of $\Gamma_h$ as
		\[
		\Gamma^j_h\coloneqq\{(x,y)\in\R^2 \,:\, x\in(a,b),\, h^-(x)\leq y<h^+(x)\};
		\]
		\item The \emph{cut part} of $\Gamma_h$ as
		\[
		\Gamma^c_h\coloneqq\{(x,y)\in\R^2 \,:\, x\in(a,b),\, h(x)\leq y<h^-(x)\};
		\]
		\item The \emph{regular part} of $\Gamma_h$ as
		\[
		\Gamma^r_h\coloneqq \Gamma_h\setminus (\Gamma_h^j\cup\Gamma_h^c).
		\]
	\end{itemize}
	Moreover, we introduce the notation $\widetilde{\Gamma}_h\coloneqq \Gamma^r_h\cup\Gamma^j_h$.
\end{definition}

\begin{remark}
	Note that
	\[
	\Gamma_h=\widetilde{\Gamma}_h\cup \Gamma_h^c
	=\Gamma_h^r\cup\Gamma_h^j\cup\Gamma_h^c,
	\]
	holds for every $h\in \bv(a,b)$.
	Moreover, when there is no room for confusion, we will drop the suffix $h$ in the notation above. 
\end{remark}

We now define the notion of convergence that we are going to use to study our functionals.

\begin{definition}\label{convergence}
	We say that a sequence $(\Omega_k,v_k,\mu_k)_k\subset\mathcal{A}$ converges to a configuration $(\Omega,v,\mu)\in\mathcal{A}$ if the following three conditions are satisfied:
	\begin{enumerate}
		\item[$(i)$] $\R^2\setminus\Omega_k\stackrel{H}{\rightarrow} \R^2\setminus\Omega$ in the Hausdorff convergence of sets;
		\item[$(ii)$] $v_k\rightharpoonup v$ weakly in $W^{1,2}_{\mathrm{loc}}(\Omega;\R^2)$,
		\item[$(iii)$] $\mu_k\wtom\mu$ weakly$^\ast$ in the sense of measures;
	\end{enumerate}
	as $k\to\infty$.
	We will write $(\Omega_k,v_k,\mu_k)\to (\o,v,\mu)$ to denote the above convergence. 
\end{definition}

\begin{remark}
	Note that, if $K\subset\Omega$ is a compact set, then there exists $k_0\in\N$ such that $K\subset\Omega_k$ for all $k\geq k_0$. Therefore, the convergence of the functions $v_k$'s is well defined.
\end{remark}

Now we are going to define the setting for our relaxed functional.

\begin{definition}\label{theta}
	Let $\psi:[0,+\infty)\to\R$. We define the \emph{recession coefficients} of $\widetilde{\psi}$ and $\psi^c$ as
	\begin{align*}
		\widetilde{\theta}\coloneqq\lim_{s\to+\infty}\frac{\widetilde{\psi}(s)}{s}\quad\text{and}\quad
		\theta^c\coloneqq\lim_{s\to+\infty}\frac{\psi^c(s)}{s},
	\end{align*}
\blue{respectively, where $\widetilde{\psi}$ is as in Definition \ref{psitilde} and $\psi^c$ as in Definition \ref{psic}.}
\end{definition}

In Lemma \ref{lem:thetac_thetawild} we will prove that $\widetilde{\theta} = \theta^c$. The common value will be denoted by $\theta$.
We are now in position to introduce the candidate for the relaxed energy.

\begin{definition}\label{relaxedfunctional}
	Let $\g:\mathcal{A}\to[0,+\infty)$ be the functional defined as
	\begin{align*}
		\mathcal{G}(\o,v,\mu)\coloneqq \int_\Omega
		W\big( E(v(\textbf{x}))-E_0(y) \big)\ \text{d}\textbf{x}
		+ \int_{\widetilde{\Gamma}}\widetilde{\psi}\big(u(\textbf{x})\big)\dho(\textbf{x})
		+\int_{\Gamma^c}\psi^c\big(u(\textbf{x})\big)\dho(\textbf{x})
		+\theta\mu^s(\Gamma),
	\end{align*}
	where $\theta$ is the common recession coefficient of $\widetilde{\psi}$ and $\psi^c$.
\end{definition}

\section{Technical results}

In this section we collect the main technical results that will be needed in the proof of the integral representation of the relaxation.

\begin{lemma}\label{lem:psic_sa}
	Let $\psi:[0,+\infty)\to\R$. Then, the function $\psi^c$ (see Definition \ref{psic}) is convex and sub-additive.
\end{lemma}

\begin{proof}
	\emph{Step 1.} We prove that $\psi^c$ is sub-additive. Fix $z\geq 0$. Then, by definition of $\psi^c(z)$, there exist $x,y\geq0$ with $z=x+y$ such that
	\[
	\psi^c(z)=\widetilde{\psi}(x)+\widetilde{\psi}(y).
	\]
	Thus,	
	\[
	\psi^c(z)=\widetilde{\psi}(x)+\widetilde{\psi}(y)\geq \widetilde{\psi}(x+y)=\widetilde{\psi}(z),
	\]
	where last inequality follows from the sub-additivity of $\widetilde{\psi}$.
	Moreover,
	\begin{align*}
		\psi^c(z+w)\leq\widetilde{\psi}(z)+\widetilde{\psi}(w)\leq\psi^c(z)+\psi^c(w),
	\end{align*}
	for every $z,w\geq0$. \\
	
	\emph{Step 2.} We prove that $\psi^c$ is convex. Let $z,w\geq0$ and $\lambda\in[0,1]$. By definition of $\psi^c(z)$, and of $\psi^c(w)$, there exist $z_1,z_2,w_1,w_2\geq0$ with $z=z_1+z_2$ and $w=w_1+w_2$ such that
	\[
	\psi^c(z)=\widetilde{\psi}(z_1)+\widetilde{\psi}(z_2),\quad\quad\quad
	\psi^c(w)=\widetilde{\psi}(w_1)+\widetilde{\psi}(w_2).
	\]
	Note that
	\begin{align*}
		\lambda z+(1-\lambda)w&=\lambda(z_1+z_2)+(1-\lambda)(w_1+w_2)\\[5pt]
		&=\big(\lambda z_1+(1-\lambda)w_1\big)+\big(\lambda z_2+(1-\lambda )w_2\big).
	\end{align*}
	Thus, we get that
	\begin{align*}
		\psi^c\big(\lambda z+(1-\lambda)z\big)&\leq\widetilde{\psi}\big(\lambda z_1+(1-\lambda)w_1\big)+\widetilde{\psi}\big(\lambda z_2+(1-\lambda )w_2\big)\\[5pt]
		&\leq\lambda\widetilde{\psi}(z_1)+(1-\lambda)\widetilde{\psi}(w_1)+\lambda\widetilde{\psi}(z_2)+(1-\lambda)\widetilde{\psi}(w_2)\\[5pt]
		&=\lambda\psi^c(z)+(1-\lambda)\widetilde{\psi}(w),
	\end{align*}
	where, in the second step, we used the convexity of $\widetilde{\psi}$.
\end{proof}

We now prove that the recession coefficients of $\widetilde{\psi}$ and of $\psi^c$, defined in Definition \ref{theta}, coincide.

\begin{lemma}\label{lem:thetac_thetawild}
	Let $\psi:[0,+\infty)\to\R$. Then, $\widetilde{\theta} = \theta^c$.
\end{lemma}

\begin{proof}
	We first prove that $\theta^c\leq \widetilde{\theta}$.
	Indeed, since $\psi^c(s)\leq 2\widetilde{\psi}(s/2)$, for all $s\geq0$, we have that
	\begin{align*}
		\theta^c=\lim_{s\to+\infty}\frac{\psi^c(s)}{s} \leq\lim_{s\to+\infty} \frac{2}{s}\widetilde{\psi}\Big(\frac{s}{2}\Big)=\widetilde{\theta}.
	\end{align*}
	We now prove that $\theta^c\geq \widetilde{\theta}$.
	Fix $z\geq 0$, and let $x,y\geq0$ with $z=x+y$ such that 
	\[
	\psi^c(z) = \widetilde{\psi}(x)+\widetilde{\psi}(y).
	\]
	Then, we get
	\begin{align*}
		\psi^c(z)=\widetilde{\psi}(x)+\widetilde{\psi}(y)\geq \widetilde{\psi}(z),
	\end{align*}
	where last inequality follows from the sub-additivity of $\widetilde{\psi}$.
	Therefore,
	\begin{align*}
		\theta^c=\lim_{s\to+\infty}\frac{\psi^c(s)}{s}\geq\lim_{s\to+\infty}\frac{\widetilde{\psi}(s)}{s}=\widetilde{\theta}.
	\end{align*}
	This concludes the proof.
\end{proof}

An important result that will be used several times is the following. 

\begin{lemma}\label{lem:finitely_many_cuts}
	Let $h\in\bv(a,b)$ be lower semi-continuous, and let $\varepsilon>0$. Define
	\[
	P(\varepsilon)\coloneqq \left\{\, x\in(a,b) \,:\, \exists\, y\in \Gamma_h
	\text{ s.t. } h(x)\leq y \leq h^-(x)-\varepsilon \,\right\}.
	\]
	Then, $P(\varepsilon)$ is a finite set.
\end{lemma}

\begin{proof}
	By \cite[Corollary 3.33]{AFP}, it holds that
	\[
	|\text{D}h|(a,b) = \|h'\|_{L^1(a,b)} + \sum_{x\in S} [h^+(x) - h(x)] + |\text{D}^c|(a,b),
	\]
	where $S$ denotes the set of points $x\in (a,b)$ such that $h^+(x) > h(x)$, and $\text{D}^c h$ is the Cantor part of the measure $\text{D} h$.
	we recall that from Theorem \ref{thm:properties_PBV} we have that $J_h$ is at most countable. Therefore, we obtain that
	\[
	\sum_{x\in S} [h^-(x) - h(x)]  < +\infty.
	\]
	Notice that the set $P(\varepsilon)$ corresponds to points in $S$ where the quantity  $h^-(x) - h(x)$ is at least $\varepsilon$. From the convergence of the series above, we get the desired result.
\end{proof}

We now prove a result that will be needed in the limsup inequality.

\begin{lemma}\label{lemmanonempty}
	Let $r>0$, and let $\{z_j\}_{j\in\N}$ be an enumeration of $\Z^2$. Define
	\[
	Q^j\coloneqq r \left( z_j + (0,1)^2 \right).
	\]
	Let $h\in \bv(a,b)$, and let $(h_k)_k$ be a sequence of Lipschitz functions such that $\R^2\setminus \Omega_{h_k}\stackrel{H}{\rightarrow} \R^2\setminus \Omega_h$, as $k\to\infty$.
	Then, there exists $v\in\R^2$, and $\overline{k}\in\N$ such that the grid defined as 
	\begin{align*}
		\widetilde{Q}^j\coloneqq v+Q^j
	\end{align*}
	satisfies:
	\begin{itemize}
		\item[(a)] The intersection between the graph of $h$ and the boundary of the new grid is finite, namely
		\begin{align*}
			\hz\big(\Gamma\cap (\bigcup_{j\in\N} \partial \widetilde{Q}^j)\big)<+\infty.
		\end{align*}
		\item[(b)] We have that
		\begin{align*}
			\ho(\Gamma_k\cap \widetilde{Q}^j)\neq0
			\quad\quad\text{ if and only if }\quad\quad
			\ho(\Gamma\cap \widetilde{Q}^j)\neq0,
		\end{align*}
		for every $k\geq\bar{k}$.
	\end{itemize}
\end{lemma}

\begin{proof} We first prove $(a)$. We first consider an horizontal translation. Since $h\in \bv (a,b)$, it has at most a countable number of jumps and cuts. Therefore, there is $v_1\in\R$ such that
	\begin{align*}
		\hz\big((\Gamma^j\cup \Gamma^c)\cap\big[\bigcup_{j\in\N} \partial \big((v_1,0)+Q^j\big)\big]\big)<+\infty.
	\end{align*}
	Now we need to find a suitable vertical translation. Using the coarea formula (see \cite[Theorem 3.40]{AFP}), we infer that
	\begin{align*}
		\mathrm{Per}\big(\{ x\in (a,b):h(x)>t \}\big) < +\infty,
	\end{align*}
	for almost every $t\in\R$, where $\mathrm{Per}$ denotes the perimeter.
	Since we are using the lower semi-continuous representative of $h$, the sup-level set $\{ x\in (a,b):h(x)>t \}$ is open for all $t\in\R$, which yields that
	\[
	\partial\{ x\in (a,b):h(x)>t \} = \{ x\in (a,b):h(x)=t \}.
	\]
	Thus, we obtain that
	\begin{align*}
		\hz\big(\{ x\in (a,b):h(x)=t \} \big)<+\infty,
	\end{align*}
	for almost every $t\in\R$. Let $D\subset\R$ defined as
	\begin{align*}
		D\coloneqq\big\{t>0:	\hz\big(\{ x\in (a,b):h(x)=t \} \big)=+\infty\big\}.
	\end{align*}
	By definition, we have that $|D|=0$. Let $r>0$, and, for every $t>0$, set
	\begin{align*}
		G(t)\coloneqq\{rj+t:j\in\Z\}.
	\end{align*}
	We now claim that
	\begin{align*}
		|\{t\in(0,r):G(t)\cap D\neq\emptyset\}|=0.
	\end{align*}
	First, note that if $s,t\in(0,r)$, with $s\neq t$, we have $G(t)\cap G(s)=\emptyset$. Now, define
	\begin{align*}
		D_j&\coloneqq D\cap [rj,(r+1)j],\\[5pt]
		\widetilde{D}_j&\coloneqq D_j -rj.
	\end{align*}
	By definition $\widetilde{D}_j\subset (0,r)$ and $|D_j|=|\widetilde{D}_j|=0$, for every $j\in\Z$. In conclusion, we notice that
	\begin{align*}
		\{t\in(0,r):G(t)\cap D\neq\emptyset\}=\bigcup_{j\in\Z} \widetilde{D}_j.
	\end{align*}
	The claim follows from the above equality. \\
	By proving the claim, we infer the existence of $v_2\in\R$ such that
	\begin{align*}
		\hz\big(\Gamma\cap\big[\bigcup_{j\in\N} \partial \big((0,v_2)+Q^j\big)\big]\big)<+\infty.
	\end{align*}
	In conclusion the translation $v\coloneqq (v_1,v_2)$ is the one we were looking for.\\

	We now prove part (b). Let $v\in\R^2$ be the vector found above, and let $\widetilde{Q}^j$ be the translated squares. If the graph of $h$ is contained in a single square $\widetilde{Q}^j$, then there is nothing to prove. Thus, we assume that this is not the case.
	
	Fix $j\in\N$ such that
	\[
	\ho(\Gamma\cap \widetilde{Q}^j)\neq0.
	\]
	We will prove that there exists $\bar{k}(j)\in\N$ such that
	\[
	\ho(\Gamma_k\cap \widetilde{Q}^j)\neq0.
	\]
	for all $k\geq\bar{k}(j)$. Let $x\in \Gamma\cap \widetilde{Q}^j$. By the Kuratowski convergence, there exists $(x_k)_k$ with $x_k\in \Gamma_k$ for all $k\in\N$ such that $x_k\to x$ as $k\to\infty$. Since $\widetilde{Q}^j$ is open, there exists $\bar{k}(j)\in\N$ (depending also on $x$, but this is not a problem) such that $x_k\in \Gamma_k\cap \widetilde{Q}^j$ for all $k\geq \bar{k}(j)$.
	Using the fact that the graph of $h$ is not entirely contained in the open square $Q^j$, and that the extended graph of $h_k$ is a connected curve, we obtain that
	\[
	\ho(\Gamma_k\cap \widetilde{Q}^j)\neq0
	\]
	as desired. Since $h\in\bv(a,b)$, it is bounded, and hence contained in a finite number of squares. In the following, we will also need to consider $\bar{k}_1\in\N$, the maximum of the $\bar{k}(j)$'s. \\
	
	We now prove the opposite implication. Let $j\in\N$ be such that
	\[
	\ho(\Gamma\cap \widetilde{Q}^j)=0.
	\]
	Then, by Kuratowski convergence and the fact that $\widetilde{Q}^j$ is open, we infer that there exists $\widetilde{k}(j)\in\N$ such that for all $k\geq \widetilde{k}(j)$ it holds
	\[
	\ho(\Gamma_k\cap \widetilde{Q}^j)=0.
	\]
	Again, let $\widetilde{k}_2\in\N$ be the maximum of the $\widetilde{k}(j)$'s.
	
	Setting $\bar{k}\coloneqq \max\{\bar{k}_1, \widetilde{k}_2\}$, we get the desired result.
\end{proof}

Finally, we prove a result about the so called wriggling process. This was introduced in \cite[Lemma 4]{CarCriDie} to exploit the quantitative loss of lower semi-continuity of the perimeter in order to recover the relaxed energy density from $\psi$.
The difference with this latter is that, in our case, only vertical perturbations are allowed. Moreover, we impose the oscillating profiles to stay below the given function.

\begin{proposition}\label{step3limsup}
	Let $h:[\alpha,\beta]\to\R$ be a non-negative Lipschitz function and  let  $r\geq 1$. Then, there exists a sequence of non-negative Lipschitz functions $(h_k)_k$ such that:
	\begin{enumerate}
		\item[$(i)$] $\ho(\Gamma_k)= r\ho(\Gamma)$;\vspace{0.1cm}
		\item[$(ii)$] $h(\alpha)=h_k(\alpha)$, and $ h(\beta)=h_k(\beta)$, for every $k$;\vspace{0.1cm}
		\item[$(iii)$] $h\leq h_k$, for every $k$;\vspace{0.1cm}
		\item[$(iv)$] $h_k\to h$ uniformly as $k\to\infty$;\vspace{0.1cm}
		\item[$(v)$] $\ho\llcorner\Gamma_k\wtom r\ho\llcorner\Gamma$, as $k\to\infty$,
	\end{enumerate}
	where we used the notation $\Gamma_k:=\Gamma_{h_k}$, and $\Gamma:=\Gamma_h$.
\end{proposition}

\begin{proof}
	\emph{Step 1}. Fix $\alpha\leq p\leq q\leq \beta$. We prove the existence of a sequence $(\xi_k)_k$ of Lipschitz functions $\xi_k:[p,q]\to[0,+\infty)$, that satisfies
	\begin{enumerate}
		\item[$(i')$] $\ho(\Gamma_{\xi_k})= r\ho(\Gamma)$;\vspace{0.1cm}
		\item[$(ii')$] $h(p)=\xi_k(p)$, and $ h(q)=\xi_k(q)$, for every $k$;\vspace{0.1cm}
		\item[$(iii')$] $h\leq \xi_k$, for every $k$;\vspace{0.1cm}
		\item[$(iv')$] $\xi_k\to h$ uniformly on $[p,q]$, as $k\to\infty$,
	\end{enumerate}
	
	Notice that if $r=1$ it is enough to consider the constant sequence $\xi_k=h$, for each $k$. Thus, fix $r>1$.
	Let $(\lambda_k)_k\subset(0,1)$ be an infinitesimal sequence such that $0<\lambda_k<q-p$ for each $k\in\N$, and $k\lambda_k\to\infty$ as $k\to\infty$.
	For each $k\in\N$, define the function $\eta_k\in\mathcal{C}\big([p,q]\big)$ as
	\begin{align*}
		\eta_k(x)\coloneqq
		\begin{cases}
			\displaystyle\frac{x-p}{\lambda_k}\quad &x\in[p,p+\lambda_k),\\[10pt]
			\displaystyle1\quad &x\in[p+\lambda_k,q-\lambda_k],\\[5pt]
			\displaystyle-\frac{x-q}{\lambda_k}\quad &x\in (q-\lambda_k,q].
		\end{cases}
	\end{align*} 
	For each $k\in\N$, let $t_k\geq 0$ that will be chosen later, and define the non-negative Lipschitz function $\xi_k:[p,q]\to [0,+\infty)$ as
	\begin{align}\label{wriggle}
		\xi_k(x)\coloneqq h(x)+\Big(\frac{2}{k}-\frac{1}{k}|\sin(t_k x)|\Big)\eta_k(x).
	\end{align}
	First of all, note that $\xi_k\to h$ uniformly as $k\to\infty$. Indeed, this follows from the fact that $\o_k\to\o$ as $k\to\infty$ in the Hausdorff sense.
	Moreover, from \eqref{wriggle}, we get that
	\[
	0\leq h\leq \xi_k, \quad\quad\quad h(p)=\xi_k(p),
	\quad\quad\quad h(q)=\xi_k(q).
	\]
	We claim that it is possible to chose $t_k>0$ such that $\ho(\Gamma_{\xi_k})=r\ho(\Gamma)$, for every $k\in\N$.
	In order to show that, for each $k\in\N$, let $f_k:[0,+\infty)\to(0,+\infty)$ be defined as
	\begin{align*}
		f_k(t)\coloneqq \int_p^q \sqrt{1+\partial_x H_k(x,t)^2}\,\text{d}x,
	\end{align*}
	where
	\begin{align}
		H_k(x,t)\coloneqq h(x)+\Big(\frac{2}{k}-\frac{1}{k}|\sin(tx)|\Big)\eta_k(x).
	\end{align}
	We claim that:
	\begin{enumerate}
		\item[$(a)$]  $\lim_{t\to+\infty}f_k(t)=+\infty$, for every $k\in\N$;\vspace{0.1cm}
		\item[$(b)$] $\lim_{k\to\infty} f_k(0)=\ho(\Gamma)$.
	\end{enumerate}
	Therefore, since $f_k$ is continuous for every $k\in\N$, and $r>1$, it is possible to chose $t_k>0$ such that $f_k(t_k)=\ho(\Gamma_{\xi_k})=r\ho(\Gamma)$, for every $k\in\N$.
	We now prove claim $(a)$ and $(b)$ in two separate sub-steps.\\
	
	\emph{Step 1.1}. We now prove claim $(a)$. First, notice that
	\begin{align*}
		f_k(t)=\int_p^q \sqrt{1+\partial_x H_k(x,t)^2}\,\text{d}x \geq \int_{p+\lambda_k}^{q-\lambda_k} \sqrt{1+\partial_x H_k(x,t)^2}\,\text{d}x.
	\end{align*}
	Now, fix $k\in \N$ and consider the set
	\begin{align*}
		Z_t\coloneqq \{x\in(p+\lambda_k,q-\lambda_k): \cos(t x)\geq 1/2\}.
	\end{align*}
	We now prove that
	\begin{equation}\label{eq:Zt}
		\inf_{t>0}|Z_t|>0.
	\end{equation}
	In order to do so, we first show that $|Z_n|>0$, for $n\in \N$.
	Set $I\coloneqq(p+\lambda_k,q-\lambda_k)$ and consider the function $g:I\to\{0,1\}$ defined as
	\begin{align*}
		g(x)\coloneqq \mathds{1}_{\{\cos(y)\geq 1/2\}}(x),
	\end{align*}
	and extend it periodically on $\R$. Notice that, for $n\in\N$,
	\begin{align*}
		g(nx)=\mathds{1}_{\{\cos(ny)\geq 1/2\}}(x).
	\end{align*}
	By applying the Riemann-Lebesgue Lemma, we get that
	\begin{align}\label{estimateonzn}
		|Z_n|&=|\{\cos(nx)\geq 1/2\}\cap I|
		=\int_I g(nx)\,\text{d}x\to \fint_Ig(x)\,\text{d}x>0,
	\end{align}
	as $n\to\infty$. Now, we use the above result to show \eqref{eq:Zt}.
	Let $t\in(n,n+1)$. We have that
	\[
	|Z_t|=|\{\cos(tx)\geq 1/2\}\cap I |
	\]
	and that
	\[
	\int_I g(tx)\,\text{d}x=\frac{1}{t}\int_{tI}g(z)\,\text{d}z
	\]
	As
	\begin{align}\label{gassum}
		g(z)=\sum_{m\in\Z}\mathds{1}_{\{-\frac{\pi}{3}+2m\pi\leq y\leq \frac{\pi}{3}+2m\pi\}},
	\end{align}
	we can define the following families of intervals. Set
	\begin{align*}
		\mathcal{A}_t\coloneqq\{J\subset \R: J\cap tI \neq\emptyset\}
		\quad\quad\quad
		\mathcal{B}_t\coloneqq\{J\subset \R: J\subset tI\}.
	\end{align*}
	Then, by \eqref{gassum}, we have
	\begin{align}\label{aztb}
		\frac{2\pi}{3t}\hz(\mathcal{B}_t)\leq |Z_t|\leq \frac{2\pi}{3t}\hz(\mathcal{A}_t).
	\end{align}
	Since $t\in(n,n+1)$ and by \eqref{estimateonzn} and \eqref{aztb}, we get that
	\begin{align*}
		|Z_t|&\geq \frac{2\pi}{3(n+1)} \hz(\mathcal{B}_n) = \frac{2\pi}{3(n+1)} (\hz(\mathcal{A}_n)-2) \\
		&\geq \int_I g(nx)\,\text{d}x -\frac{4\pi}{3(n+1)} 
		>C-\frac{4\pi}{3(n+1)},
	\end{align*}
	where $C>0$ is a constant independent of $n$. We conclude our claim by letting $n\to\infty$.\\

	Note that for every $t>0$, on $Z_t$ we have $\eta_k(x)=1$ and $\cos(tx)>1/2$.
	Thus, we get that
	\begin{align}\label{hogammak}
		f_k(t)
		\nonumber&\geq \int_{Z_t}  \sqrt{1+h'(x)^2+\frac{t}{k}\cos(t x)\Big[\frac{t}{k}\cos(t x)-2 \ell\Big]}\,\text{d}x \\
		&\geq \int_{Z_t}  \sqrt{1+h'(x)^2+\frac{t}{k}\cos(tx)\Big[\frac{t}{2k}-2 \ell\Big]}\,\text{d}x \nonumber \\
		&\geq \int_{Z_t}  \sqrt{1+\frac{t}{k}\cos(tx)\Big[\frac{t}{2k}-2 \ell\Big]}\,\text{d}x,
	\end{align}
	where $\ell$ is the Lipschitz constant of $h$. By choosing $t>0$ such that
	\begin{align*}
		t>4k \ell,
	\end{align*}
	from \eqref{hogammak}, and from $\cos(tx)>1/2$ on $Z_t$, we obtain
	\begin{align}\label{lowerbound}
		f_k(t)\geq \int_{Z_t}  \sqrt{1+\frac{t}{2k}\Big[\frac{t}{2k}-2 \ell\Big]}\,\text{d}x.
	\end{align}
	Thus, from \eqref{eq:Zt} and \eqref{lowerbound}, we conclude that
	\begin{align*}
		\lim_{t\to+\infty}f_k(t)=+\infty.
	\end{align*}
	
	\emph{Step 1.2} Now we prove claim $(b)$. Notice that
	\begin{align}\label{continuityfk}
		f_k(0)&=
		\int_{p}^{p+\lambda_k}\sqrt{1+\Big(h'(x)+\frac{2}{k\lambda_k}\Big)^2}\,\text{d}x
		+\int_{p+\lambda_k}^{q-\lambda_k}\sqrt{1+h'(x)^2}\,\text{d}x \nonumber \\
		&\hspace{2cm}+\int_{q-\lambda_k}^q \sqrt{1+\Big(h'(x)+\frac{2}{k\lambda_k}\Big)^2}\,\text{d}x.
	\end{align}
	Since the sequence $(\lambda_k)_k$ is such that $k\lambda_k\to\infty$, and
	$\|h'\|_{L^\infty}<+\infty$  since $h$ is Lipschitz, it holds that
	\[
	\sup_{k\in\N}\sup_{x\in[p,q]} \left|h'(x)+\frac{2}{k\lambda_k}\right| < +\infty.
	\]
	Thus, letting $k\to\infty$ in \eqref{continuityfk}, we obtain
	\begin{align*}
		\lim_{k\to\infty}f_k(0)=\ho(\Gamma).
	\end{align*}
	This concludes the proof of $(b)$.\\
	
	\emph{Step 2}. We now prove the statement of the Lemma. Fix $r>1$, otherwise the statement is trivial. For $k\in\N$, divide the interval $[\alpha,\beta]$ into $k$ subintervals $\big([\alpha^k_i,\alpha_{i+1}^k]\big)_{i=1}^k$, where $\alpha^1_k=\alpha$ and $\alpha^k_{k+1}=\beta$.
	Assume that $|\alpha_{i+1}^k - \alpha_i^k|< 2/k$.
	Thanks to Step 1, for each $k\in\N$, and each $i\in\{1,\dots,k\}$, there exists a function $\xi^k_i:[\alpha^k_i,\alpha^k_{i+1}]\to[0,+\infty)$ such that
	\[
	\xi^k_1(\alpha) = h(\alpha),\quad\quad\quad
	\xi^k_i(\alpha^k_{i+1}) = \xi^k_{i+1}(\alpha^k_{i+1}),\quad\quad\quad
	\xi^k_{k+1}(\beta) = h(\beta),
	\]
	for all $i\in\{2,\dots, k\}$, with
	\[
	\| \xi^k_i - h \|_{\mathcal{C}^0(\R)} \leq \frac{1}{k},
	\]
	and such that
	\[
	\ho(\mathrm{graph}(\xi^k_i)) = r \ho(\Gamma \llcorner [\alpha^k_i,\alpha^k_{i+1}]\times\R),
	\]
	for all $i\in\{1,\dots,k\}$, and all $k\in\N$. Define $h_k:[\alpha,\beta]\to[0,+\infty)$ as
	\begin{align*}
		h_k(x) \coloneqq \xi^k_i(x),
	\end{align*}
	for $x\in[\alpha^k_i,\alpha^k_{i+1}]$.
	Note that $h_k$ is Lipschitz, $h\leq h_k$ for all $k\in\N$, $h_k\to h$ uniformly in $k$, and
	\begin{align*}
		\ho(\Gamma_{k})=\sum_{i=1}^k\ho\big(\mathrm{graph}(\xi^k_i)\big)
		=r\sum_{i=1}^k\ho(\Gamma\llcorner [\alpha^i,\alpha^{i+1}]\times\R)
		=r\ho(\Gamma).
	\end{align*}
	It remains to prove property $(v)$. To do so, fix $\varphi\in C_c(\R^2)$ and $\varepsilon>0$.
	Thanks to the uniform continuity of $\varphi$, there exists $\bar{k\in\N}$ such that for $k\geq \bar{k}$ the following holds: if $x_i\in[\alpha^k_i,\alpha^k_{i+1}]$, then
	\begin{equation}\label{eq:unif_cont_phi_h}
		|\varphi\big(x,h_k(x)\big)-\varphi\big(x_i,h_k(x_i)\big)|\leq\varepsilon.
	\end{equation}
	Moreover, from the fact that $h_k$ is converging uniformly to the continuous function $h$, up to increasing the value of $\bar{k}$, we can also assume that
	\begin{equation}\label{eq:unif_cont_phi_h_1}
		|\varphi\big(x_i,h_k(x_i)\big) - \varphi\big(x_i,h(x_i)\big)|\leq\varepsilon.
	\end{equation}
	Using \eqref{eq:unif_cont_phi_h}, we get
	\begin{align}\label{vfirststep}
		\nonumber
		\int_{	\Gamma_k}\varphi(\textbf{x})\dho
		&- r\int_{\Gamma}\varphi(\textbf{x})\dho \nonumber\\
		&=\sum_{i=1}^k \int_{\alpha^i}^{\alpha^{i+1}}
		\Big[\varphi\big(x,h_k(x)\big)\sqrt{1+h_k'(x)^2}
		- r\varphi\big(x,h(x)\big)\sqrt{1+h'(x)^2}\Big]\text{d}x \nonumber\\
		&\leq\sum_{i=1}^k \Big[\varepsilon
		\int_{\alpha^i}^{\alpha^{i+1}}\Big(\sqrt{1+h_k'(x)^2}
		+ r\sqrt{1+h'(x)^2}\Big)\text{d}x \nonumber \\
		&\hspace{1.4cm}+\int_{\alpha^i}^{\alpha^{i+1}}
		\Big(\varphi\big(x_i,h_k(x_i)\big)\sqrt{1+h_k'(x)^2}
		- r\varphi\big(x_i,h(x_i)\big)\sqrt{1+h'(x)^2}\Big)\text{d}x\Big ] \nonumber \\
		&\leq \varepsilon \sum_{i=1}^k\int_{\alpha^i}^{\alpha^{i+1}}		
		\Big(\sqrt{1+h_k'(x)^2}+r\sqrt{1+h'(x)^2}\Big)\text{d}x\nonumber \\
		&\hspace{1.4cm}+\varphi\big(x_i,h(x_i)\big)\sum_{i=1}^k
		\Big[\int_{\alpha^i}^{\alpha^{i+1}}\big(\sqrt{1+h_k'(x)^2}
		- r\sqrt{1+h'(x)^2}\Big)\text{d}x\Big]  \nonumber \\
		&= \varepsilon \sum_{i=1}^k \int_{\alpha^i}^{\alpha^{i+1}}		
		\Big(\sqrt{1+h_k'(x)^2}+r\sqrt{1+h'(x)^2}\Big)\text{d}x,
	\end{align}
	where in the previous to last step we used \eqref{eq:unif_cont_phi_h_1}, while last step follows from $\ho(\Gamma_k)=r\ho(\Gamma)$.
	Thus, from \eqref{vfirststep} we obtain
	\begin{align*}
		\int_{	\Gamma_k}\varphi(\textbf{x})\dho-r\int_{	\Gamma}\varphi(\textbf{x})\dho\leq 2r\ho(\Gamma)\varepsilon.
	\end{align*}
	Thus, since $\varepsilon$ is arbitrary, we get that $\ho\llcorner\Gamma_k\wtom r\ho\llcorner\Gamma$ as $k\to\infty$.
\end{proof}

\begin{remark}\label{remarkwriggle}
	From the above proof, we can infer the following facts,
	\begin{enumerate}
		\item[$(i)$] Following \eqref{lowerbound},
		\begin{align*}
			r\ho(\Gamma)\geq\int_{Z_{t_k}}  \sqrt{1+h'(x)^2+\frac{t_k}{2k}\Big[\frac{t_k}{2k}-2 \ell\Big]}\text{d}x\geq \mu \sqrt{\frac{t_k}{2k}\Big[\frac{t_k}{2k}-2 \ell\Big]},
		\end{align*}
		where $\mu\coloneqq \inf_{t\geq0} |Z_t|$. This leads us to
		\begin{align*}
			\Big(\frac{t_k}{2k}\Big)^2-2 \ell\Big(\frac{t_k}{2k}\Big)\leq\frac{1}{\mu^2}r^2\ho(\Gamma)^2.
		\end{align*}
		If we solve for $t/2k$ we get
		\begin{align}\label{boundtkk}
			\frac{t_k}{k}\leq C,
		\end{align}
		where 
		\begin{align*}
			C\coloneqq 2\Big(\ell+\sqrt{\ell^2+\frac{r^2\ho(\Gamma)^2}{\mu^2}}\Big).
		\end{align*}
		\item[$(ii)$]  We claim that $t_k\to+\infty$ as $k\to\infty$. Assume by contradiction this is not the case, namely that
		\begin{align*}
			\sup_k t_k\leq \tau,
		\end{align*}
		for some $\tau>0$. 
		Thus, we have that
		\begin{align*}
			h'_{k}(x)&=h'(x)-\frac{t_k}{k}\cos(t_kx)\frac{|\sin(t_kx)|}{\sin(t_k x)}\eta_k(x)+\Big(\frac{2}{k}-\frac{1}{k}|\sin(t_kx)|\Big)\eta_k'(x)\\[5pt]
			&\leq h'(x)+\frac{\tau}{k}+\frac{2\eta_k'(x)}{k},
		\end{align*}
		for every $k$. From the inequality
		\begin{align*}
			|h_k'(x)-h'(x)|\leq\frac{\tau}{k}+\frac{2\eta_k'(x)}{k}
		\end{align*}
		we infer that
		\begin{align}\label{convergencelength}
			\ho(\Gamma_k)\to\ho(\Gamma).
		\end{align}
		From step 1 we know that
		\begin{align}\label{contradictionwriggling}
			\ho(\Gamma_k)=r\ho(\Gamma)>\ho(\Gamma),
		\end{align}
		with $r>1$ and for every $k$.
		By putting together \eqref{convergencelength} and \eqref{contradictionwriggling} we get a contradiction.
		\item[$(iii)$] From the expression of $h'_k$, we can actually choose the sequence $(\lambda_k)_k$ such that the sequence $(h_k)_k$ is uniformly Lipschitz. Indeed, on $[\alpha,\alpha+\lambda_k]$ we have
		\begin{align*}
			|h_k'(x)|\leq \ell+\frac{t_k}{k}+
			\frac{2}{k\lambda_k}.
		\end{align*}
		As $t_k/k$ is bounded and $(\lambda_k)_k$ is chosen such in such a way that $k\lambda_k\to+\infty$ as $k\to\infty$, we can conclude.
	\end{enumerate}
\end{remark}
\section{Liminf inequality}

We now present the main ideas of the proof of the  liminf inequality, contained in the following theorem. One of the issues that we take in account is the fact that our final configuration $\Gamma$, is the graph of a $\bv$ function which might have a dense cut set. In particular, this is a problem since in our argument we deal with what is happening on the left and on right of every cut in $\Gamma$. This is not doable in case the cut set is dense. One possible way to go around, is to split the energy on $\Gamma^c$. By fixing $\varepsilon>0$, since $h$ is a $\bv$ function, the cuts in $\Gamma^c$ whose length is larger then $\varepsilon$ is necessarily finite. For those amount of cuts we do the liminf inequality by using the result contained in \cite{CarCri}. Finally, for the cut part in $\Gamma^c$ with lenght smaller that $\varepsilon$, we prove that the energy there is as small as we want as $\varepsilon\to0$.

\begin{theorem}\label{liminfinequality}
	For every configuration $(\Omega,v,\mu)\in\mathcal{A}$ and for every sequence of regular configurations $(\Omega_k,v_k,\mu_k)_k\subset\mathcal{A}_r$ such that $(\Omega_k,v_k,\mu_k)\to(\Omega,v,\mu)$ as $k\to\infty$, we have
	\begin{align*}
		\g(\Omega,v,\mu)\leq \liminf_{k\to\infty}\f (\Omega_k,v_k,\mu_k).
	\end{align*}
\end{theorem}

\begin{proof}
	Fix $\varepsilon>0$ and consider the set
	\begin{align*}
		C_\varepsilon\coloneqq \{\blue{\textbf{x}=(x,y)}\in\Gamma^c:h^-(x)-y<\varepsilon\}.
	\end{align*}
    By a standard measure theory argument, it is possible to choose $\varepsilon$ such that $\mu(\Gamma\cap\partial C_\varepsilon)=0$.
	As a consequence, from Lemma \ref{lem:finitely_many_cuts}, we have that $\Gamma^c\setminus C_\varepsilon$ consists of a finite number of vertical segments, whose projections  on the $x$-axes corresponds to the set $(x^i)_{i=1}^N$. \blue{Recalling the definition of $\Gamma^c$ (see Definition \ref{cuts}), it holds that $C_\varepsilon$ is monotonically converging to the empty set, as $\varepsilon\to0$. Therefore, we get that}
	\begin{equation}\label{eq:conv_liminf}
		\mu(C_\varepsilon)\to0,\quad\quad\quad
		\mu(\Gamma^c\setminus C_\varepsilon)\to \mu(\Gamma^c),
	\end{equation}
	as $\varepsilon\to0$.
	\blue{Let $\delta=\delta(\varepsilon)>0$ such that we have $\delta<|x^i-x^j|$, for every $i,j=1,\dots,N$.}
	As we have a finite number of cuts, in order to simplify the notation, we do the following construction as we had only one cut point, and then we repeat it for each other one.
	
	Fix $i\in{1,\dots, N}$. Since $\R^2\setminus\o_k\stackrel{H}{\rightarrow}\R^2\setminus \o$, for every cut point $\big(x^i,h(x^i)\big)$, there is a sequence of the form $\big(x_k,h_k(x_k)\big)_k$ such that $(x_k)_k\subset(x^i-\delta,x^i+\delta)$ and $\big(x_k,h_k(x_k)\big)\to\big(x^i,h(x^i)\big)$ as $k\to\infty$. Indeed, by Proposition \ref{kuratowskyconvergence} there is a sequence $(x_k,y_k)_k\subset \R^2\setminus \o_k$ such that $(x_k,y_k)\to \big(x^i,h(x^i)\big)$. By definition, we have that $h_k(x_k)\leq y_k$ , up to a subsequence (not relabelled), we have that $\big(x_k,h_k(x_k)\big)\to(x^i,z^i)$, for some $z^i\in\R$. We would like to have $z^i=h(x^i)$. If we had $z^i> h(x^i)$, then
	\begin{align*}
		\lim_{k\to\infty}h_k(x_k)\leq h(x^i)<z^i,
	\end{align*}
	which contradicts our convergence above. Vice versa, if $z^i<h(x^i)$, then $(x^i,z^i)\notin \R^2\setminus\o$. In conclusion we have $z^i=h(x^i)$ and thus $\big(x_k,h_k(x_k)\big)\to \big(x^i,h(x^i)\big)$, as $k\to\infty$.\\

	Around each vertical cut, we set, for each $k\in\N$ (see Figure \ref{fig:liminf}),
	\[
	R_k^\ell\coloneqq (x^i-\delta,x_k)\times(0,h^-(x^i)-\varepsilon), \quad\quad
	R_k^r\coloneqq (x_k,x^i+\delta)\times(0,h^-(x^i)-\varepsilon),
	\]
	and
	\[
	\mathcal{R}_\delta^\varepsilon\coloneqq R_k^\ell\cup R_k^r\cup\big[\{x_k\}\times (0,h^-(x^i)-\varepsilon)\big].
	\]
\blue{Thanks to the existence of the right and left limits of $h$ at every point (see Theorem \ref{thm:properties_PBV}), up to further reducing $\delta$, we can assume that 
\begin{equation*}
\mathcal{R}_\delta^\varepsilon\cap \Gamma = \{x_k\}\times (0,h^-(x^i)-\varepsilon).
\end{equation*}
}
	Now we split the energy in the following way. Take any $(\Omega_k,v_k,\mu_k)_k\subset\mathcal{A}_r$ such that $(\Omega_k,v_k,\mu_k)\to(\Omega,v,\mu)$ as $k\to\infty$. We have
	\begin{align}\label{liminfbegin2}
		\nonumber\liminf_{k\to\infty}\Big[\int_{\Omega_k} W\big(E(v_k)&-E_0(y)\big)\,\text{d}\textbf{x}+\int_{\Gamma_k}\psi(u_k)\dho\Big]\\[5pt]
		\nonumber\geq&\liminf_{k\to\infty}\int_{\Omega_k} W\big(E(v_k)-E_0(y)\big)\,\text{d}\textbf{x}+\liminf_{k\to\infty}\int_{\Gamma_k\setminus 	\mathcal{R}_\delta^\varepsilon}\psi(u_k)\dho\\[5pt]
		&\hspace{1cm}+\liminf_{k\to\infty}\int_{\Gamma_k\cap 	\mathcal{R}_\delta^\varepsilon}\psi(u_k)\dho.
	\end{align}
	We are going to estimate each term on the right-hand side of \eqref{liminfbegin2} separately.

	\begin{figure}
		\includegraphics[scale=1]{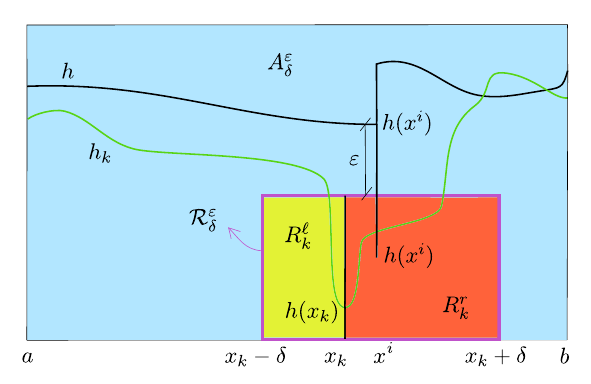}
		\caption{The rectangles we are using for the estimate of the liminf. In particular, the set $A_\delta^\varepsilon$ is the light blue, while the boundary of the rectangle $\mathcal{R}_\delta^\varepsilon$ is the one in purple.}
		\label{fig:liminf}
	\end{figure}

	\emph{Step 1}. Here we estimate the bulk term on the right-hand side of \eqref{liminfbegin2}. Since $v_k\wto v$ in $W^{1,2}_\mathrm{loc}(\Omega;\R^2)$ as $k\to\infty$, for every compactly contained set $K\subset\Omega$, we get
	\begin{align*}
		\nonumber\liminf_{k\to\infty}\int_{\Omega_k}W\big(E(v_k)-E_0(y)\big)\,\text{d}\textbf{x}&\geq \liminf_{k\to\infty}\int_{K}W\big(E(v_k)-E_0(y)\big)\,\text{d}\textbf{x}\\[5pt]
		&\geq\int_{K}W\big(E(v)-E_0(y)\big)\,\text{d}\textbf{x},
	\end{align*}
	as $E(\cdot)$ is linear and $W(\cdot)$ is convex. Since $K$ is arbitrary, we can conclude by taking an increasing sequence $(K_j)_j$ of sets compactly contained in $\Omega$ with $|\o\setminus K_k|\to 0$ as $k\to\infty$.
	Thus, by using the Monotone Convergence Theorem
	\begin{equation}\label{liminf1}
		\liminf_{k\to\infty}\int_{\Omega_k}W\big(E(v_k)-E_0(y)\big)\,\text{d}\textbf{x}
		\geq\int_{\Omega}W\big(E(v)-E_0(y)\big)\,\text{d}\textbf{x},
	\end{equation}
	we get the liminf for the bulk term. \\
	
	\emph{Step 2}. For the second term on the right-hand side of \eqref{liminfbegin2}, we would like to apply Theorem \ref{cc}. 
	Fix $\varepsilon>0$. By knowing that for each $k\in\N$ we have $|h_k|\leq M$, we define the open set
	\begin{align*}
		A_\delta^\varepsilon\coloneqq \big([a,b]\times [0,M]\big)\setminus \overline{	\mathcal{R}}_\delta^\varepsilon.
	\end{align*}
	We have that $A_\delta^\varepsilon\cap\o_k\to A_\delta^\varepsilon\cap \o$ in $L^1$ as $k\to\infty$. From Lemma \ref{lem:subgraph_finite_per}, we have that
	\begin{align*}
		\ho\big((\partial^\ast \o\cap A_\delta^\varepsilon)\Delta \widetilde{\Gamma}\big)=0.
	\end{align*}
	By definition, we can write
	\begin{align*}
		u_k\ho\llcorner(\partial\o_k\cap A_\delta^\varepsilon)\wtom \mu\llcorner	A_\delta^\varepsilon=u\ho\llcorner\widetilde{\Gamma}+\mu^s\llcorner A_\delta^\varepsilon+u\ho\llcorner C_\varepsilon,
	\end{align*}
	as $k\to\infty$, and, by applying Theorem \ref{cc}, we have
	\begin{align}\label{liminf2}
		\liminf_{k\to\infty}\int_{\partial\o_k\cap A_\delta^\varepsilon}\psi(u_k)\dho\geq \int_{\widetilde{\Gamma}\cap A_\delta^\varepsilon}\widetilde{\psi}(u)\dho+\theta\mu^s(A_\delta^\varepsilon)+\theta \int_{C_\varepsilon}u\dho,
	\end{align}
	as desired.\\

	\emph{Step 3}.  We now deal with the third term on the right-hand side of \eqref{liminfbegin2}. 
	Define
	\begin{align} \label{eq:def_Ek}
		E_k^\ell\coloneqq \o_k\cap R_k^\ell\quad\text{and}\quad
		E_k^r\coloneqq \o_k\cap R_k^r.
	\end{align}
	Using Lemma \ref{lem:Hausd_L1} we obtain that
	\begin{align*}
		E_k^\ell\to R^\ell&\coloneqq (x^i-\delta,x^i)\times(0,h^-(x^i)-\varepsilon),\\[5pt]
		E_k^\ell\to R^r&\coloneqq (x^i,x^i+\delta)\times(0,h^-(x^i)-\varepsilon),
	\end{align*}
	as $k\to\infty$ in $L^1$. Note that, for every $k$ large enough, both $E_k^\ell\neq\emptyset$ and $E_k^r\neq\emptyset$. Furthermore, notice that
	\begin{align*}
		\partial  E_k^\ell \cap R^\ell &=\big(\Gamma_k\cap R_k^\ell\big)\cup\big[ \{x_k\}\times\big(0,h_k(x_k)\big)\big],\\[5pt]
		\partial  E_k^r \cap R^r &=\big(\Gamma_k\cap R_k^r\big)\cup\big[ \{x_k\}\times\big(0,h_k(x_k)\big) \big].
	\end{align*}
	We now define the densities
	\begin{align*}
		u_k^\ell (\textbf{x})&\coloneqq\begin{cases}
			u_k(\textbf{x})\quad& \textbf{x}\in\Gamma_k\cap R_k^\ell, \\[5pt]
			0\quad&\textbf{x}\in\{x_k\}\times(0,h_k(x_k)),
		\end{cases}\\
		u_k^r (\textbf{x})&\coloneqq\begin{cases}
			u_k(\textbf{x})\quad &\textbf{x}\in\Gamma_k\cap R_k^r,\\[5pt]
			0\quad&\textbf{x}\in\{x_k\}\times(0,h_k(x_k)).
		\end{cases}
	\end{align*}
	We now prove that that
	\begin{align*}
		\mu_k^\ell\coloneqq u_k^\ell\ho\llcorner(\partial E_k^\ell\cap R^\ell)\wtom\mu^\ell&\coloneqq
		f\ho\llcorner(\Gamma^c\setminus C_\varepsilon)+(\mu^\ell)^s,\\[5pt]
		\mu_k^r\coloneqq	u_k^r\ho\llcorner(\partial E_k^r\cap R^r)\wtom\mu^r&\coloneqq
		g\ho\llcorner(\Gamma^c\setminus C_\varepsilon)+(\mu^r)^s,
	\end{align*}
	for some $f,g\in L^1(\Gamma^c\setminus C_\varepsilon)$ such that
	\begin{align}\label{fgu}
		f+g=u_{|\Gamma^c\setminus C_\varepsilon},
	\end{align}
	and
	\begin{align}\label{fgusing}
		(\mu^\ell)^s+(\mu^r)^s=\mu^s,
	\end{align}
	where $(\mu^\ell)^s$ and $(\mu^r)^s$ are supported in $\Gamma^c\setminus C_\varepsilon$.
	Notice that
	\[
	\mu_k^\ell\big(\{x_k\}\times(0,h_k(x_k))\big)=\mu_k^r\big(\{x_k\}\times(0,h_k(x_k))\big)=0
	\]
	holds for every $k\in\N$.
	By definition we have $\mu_k^\ell+\mu_k^r=\mu_k$, for every $k\in\N$.
	Moreover, for every set $A$, measurable with respect to $\mu_k$ (thus also for $\mu_k^\ell$ and $\mu_k^r$), we have
	\begin{align*}
		\mu_k^\ell(A)\leq \mu_k(A)=\int_{\Gamma_k\cap A}u_k\dho =||u_k||_{L^1(\Gamma_k\cap A)}\leq L,
	\end{align*}
	where $L$ is a constant independent of $A$, and is given by the fact that the sequence $(\mu_k)_k$ is weakly$^\ast$ converging. The same bound for $\mu_k^r$ also holds. We have that, up to a subsequence (not relabelled), there are two Radon measures $\mu^\ell$ and $\mu^r$ such that
	\begin{align*}
		\mu_k^\ell\wtom \mu^\ell\quad\text{and}\quad
		\mu_k^r\wtom \mu^r,
	\end{align*}
	as $k\to\infty$.\\
	
	We claim that $\mathrm{supp}(\mu^\ell)\subset \Gamma^c\setminus C_\varepsilon$ and $\mathrm{supp}(\mu^r)\subset \Gamma^c\setminus C_\varepsilon$. Indeed, take any set $A$ such that $\mu\big((\Gamma^c\setminus C_\varepsilon)\cap\partial A\big)=0$ and $A\cap (\Gamma^c\setminus C_\varepsilon)=\emptyset$. Then $\mu\big((\Gamma^c\setminus C_\varepsilon)\cap A\big)=0$. If we had $\mu^\ell\big((\Gamma^c\setminus C_\varepsilon)\cap A\big)>\mu\big((\Gamma^c\setminus C_\varepsilon)\cap A\big)$, we would have
	\begin{align*}
		\mu\big((\Gamma^c\setminus C_\varepsilon)\cap A\big)=\lim_{k\to\infty}\mu_k\big((\Gamma^c\setminus C_\varepsilon)\cap A\big)\geq\lim_{k\to\infty}\mu_k^\ell\big((\Gamma^c\setminus C_\varepsilon)\cap A\big)=\mu^\ell\big((\Gamma^c\setminus C_\varepsilon)\cap A\big),
	\end{align*}
	and this implies that $\mu^\ell\big((\Gamma^c\setminus C_\varepsilon)\cap A\big)=0$. Thus $\mu^\ell\leq\mu$ and if $\mu\big((\Gamma^c\setminus C_\varepsilon)\cap A\big)=0$, then also $ \mu^\ell\big((\Gamma^c\setminus C_\varepsilon)\cap A\big)=0$. 
	As the same holds for $\mu^r$, we conclude our claim.
	
	Then, there are $f,g\in L^1(\Gamma^c\setminus C_\varepsilon)$ for which we can write
	\begin{align*}
		\mu^\ell=f\ho\llcorner(\Gamma^c\setminus C_\varepsilon)+\	(\mu^\ell)^s\quad\text{and}\quad
		\mu^r=	g\ho\llcorner(\Gamma^c\setminus C_\varepsilon)+\	(\mu^r)^s,
	\end{align*}
	with $	(\mu^\ell)^s$ and $	(\mu^r)^s$ are singular measures with respect to $f\ho\llcorner(\Gamma^c\setminus C_\varepsilon)$ and $g\ho\llcorner(\Gamma^c\setminus C_\varepsilon)$ respectively.\
	We now prove that $\mu=\mu^\ell+\mu^r$.
	Notice that for every $\varphi\in C_c(\R^2)$,
	\begin{align*}
		\int_{\partial  E_k^\ell \cup \partial  E_k^r}\varphi \,\text{d}\mu_k\to \int_{\Gamma^c\setminus C_\varepsilon}\varphi \,\text{d}\mu,
	\end{align*}
	as $k\to\infty$, from the fact that $\mu_k\wtom\mu$. 
	On the other hand we have
	\begin{align*}
		\int_{\partial  E_k^\ell \cup \partial  E_k^r}\varphi \,\text{d}\mu_k&=\int_{\partial  E_k^\ell \cup \partial  E_k^r}\varphi \,\text{d}(\mu_k^\ell+\mu_k^r)
		=\int_{\partial  E_k^\ell }\varphi \,\text{d}\mu_k^\ell+\int_{\partial  E_k^r }\varphi \,\text{d}\mu_k^r\\[5pt]
		&\underset{k\to\infty}{\longrightarrow} \int_{\Gamma^c\setminus C_\varepsilon}\varphi \,\text{d}\mu^\ell+\int_{\Gamma^c\setminus C_\varepsilon}\varphi \,\text{d}\mu^r.
	\end{align*}
	Since $\varphi\in C_c(\R^2)$ is arbitrary, we get $\mu=\mu^\ell+\mu^s$ . In particular, we obtain \eqref{fgu} and \eqref{fgusing}.\\
	
	We now prove the convergence of the energy. Set
	\begin{align*}
		S_k\coloneqq \{x_k\}\times \big(0,h_k(x_k)\big)\quad\text{and}\quad
		S\coloneqq\{x^i\}\times \big(0,h(x^i)\big).
	\end{align*}
	We notice that $\ho(S_k)\to\ho(S)$ as $k\to\infty$. In particular, this implies that
	\begin{align}\label{sktos}
		\lim_{k\to\infty} \int_{S_k}\psi(0) \dho= \int_S \psi(0)\dho.
	\end{align}
	
	Now, we want to apply Theorem \ref{cc}.
	Recalling Definition \ref{eq:def_Ek} of the sets $E^\ell_k$ and $E^r_k$, we obtain
	\begin{align*}
		&\liminf_{k\to\infty}
		\int_{\Gamma_k\cap \mathcal{R}_\delta^\varepsilon}\psi(u_k)\dho
		+ 2\int_{S}\psi(0) \dho \\[5pt]
		&\hspace{0.3cm}=\liminf_{k\to\infty}
		\left[ \int_{\Gamma_k\cap \mathcal{R}_\delta^\varepsilon}\psi(u_k)\dho
		+ 2\int_{S_k}\psi(0) \dho \right]\\[5pt]
		&\hspace{0.3cm}=\liminf_{k\to\infty} \left[ \int_{\partial E_k^\ell\cap \mathcal{R}_\delta^\varepsilon}\psi(u_k^\ell)\dho
		+\int_{\partial E_k^r\cap \mathcal{R}_\delta^\varepsilon}\psi(u_k^r)\dho \right]\\[5pt]	
		&\hspace{0.3cm}\geq\liminf_{k\to\infty}\int_{\partial E_k^\ell\cap \mathcal{R}_\delta^\varepsilon}\psi(u_k^\ell)\dho
		+\liminf_{k\to\infty}\int_{\partial E_k^r\cap \mathcal{R}_\delta^\varepsilon}\psi(u_k^r)\dho\\[5pt]
		&\hspace{0.3cm}\geq \int_{\partial R^\ell\cap \mathcal{R}_\delta^\varepsilon }\widetilde{\psi}(f)\dho +\theta(\mu^\ell)^s( \partial R^\ell\cap \mathcal{R}_\delta^\varepsilon   )
		+ \int_{\partial R^r\cap \mathcal{R}_\delta^\varepsilon}\widetilde{\psi}(g)\dho
		+\theta(\mu^r)^s( \partial R^r\cap \mathcal{R}_\delta^\varepsilon)\\[5pt]
		&\hspace{0.3cm}= \int_{\Gamma^c\setminus C_\varepsilon}\widetilde{\psi}(f)\dho+\theta(\mu^\ell)^s( \Gamma^c\setminus C_\varepsilon)
		+ \int_{\Gamma^c\setminus C_\varepsilon}\widetilde{\psi}(g)\dho
		+\theta(\mu^r)^s( \Gamma^c\setminus C_\varepsilon)
		+2\int_S\psi(0)\dho \\[5pt]
		&\hspace{0.3cm}\geq\int_{\Gamma^c\setminus C_\varepsilon}\psi^c(u)\dho+\theta\mu^s(\Gamma^c\setminus C_\varepsilon)+2\int_S\psi(0)\dho,
	\end{align*}
	where the last inequality follows from \eqref{fgu} together with the definition of $\psi^c$. Thus,
	\begin{equation}\label{liminf3}
		\liminf_{k\to\infty}
		\int_{\Gamma_k\cap \mathcal{R}_\delta^\varepsilon}\psi(u_k)\dho\geq 
		\int_{\Gamma^c\setminus C_\varepsilon}\psi^c(u)\dho+\theta\mu^s(\Gamma^c\setminus C_\varepsilon),
	\end{equation}
	for all $\varepsilon>0$\\
	
	\emph{Step 5}. Using \eqref{liminfbegin2}, \eqref{liminf1}, \eqref{liminf2} and \eqref{liminf3} we obtain
	\begin{align*}
		\liminf_{k\to\infty}\Big[\int_{\Omega_k} W\big(E(v_k)&-E_0(y)\big)\,\text{d}\textbf{x}+\int_{\Gamma_k}\psi(u_k)\dho\Big]
		\geq\int_{\Omega}W\big(E(v)-E_0(y)\big)\,\text{d}\textbf{x} \\[5pt]
		&+\int_{\widetilde{\Gamma}\cap A_\delta^\varepsilon}\widetilde{\psi}(u)\dho+\theta\mu^s(A_\delta^\varepsilon)+\theta \int_{C_\varepsilon}u\dho \\[5pt]
		&+\int_{\Gamma^c\setminus C_\varepsilon}\psi^c(u)\dho+\theta\mu^s(\Gamma^c\setminus C_\varepsilon).
	\end{align*}
	By letting $\varepsilon\to0$, and using \eqref{eq:conv_liminf} we get the desired liminf inequality.
\end{proof}

\section{Limsup inequality}

The goal of this section is to prove the limsup inequality for the mass constrained problem.\blue{ We recall that the classes $\mathcal{A}_r(m,M)$ and $\mathcal{A}_r(m,M)$ are given in Definitions \ref{ArmM} and \ref{AmM} respectively.}

\begin{theorem}\label{limsupinequality}
	Let $m,M>0$. Let $(\Omega,v,\mu)\in\mathcal{A}(m,M)$. Then, there exists a sequence of regular configurations $(\Omega_k,v_k,\mu_k)_k\subset\mathcal{A}_r(m,M)$ such that
	\begin{align*}
		\limsup_{k\to\infty}\f (\Omega_k,v_k,\mu_k) \leq \g(\Omega,v,\mu),
	\end{align*}
	and with $(\Omega_k,v_k,\mu_k)\to(\Omega,v,\mu)$ as $k\to\infty$.
\end{theorem}

The proof is long and therefore it will be divided in several steps, each proved in a separate result.
\blue{In particular, we will need to work with a specific class of piecewise constant functions, that we introduce here.

\begin{definition}
	Let $h\in\bv(a,b)$, and $\delta>0$. We say that a finite family $(Q^j)_{j=1}^N$ of open and pairwise disjoint rectangles is \emph{$\delta$-admissible covering} for $\Gamma$, if
	\begin{itemize}
		\item[(i)] The side lengths of each $Q^j$ is less than $\delta$;
		\item [(ii)]It holds
		\[
		\Gamma \subset \bigcup_{i=1}^N \overline{Q^j};
		\]
		\item[(iii)] $\ho(\Gamma\cap\partial Q^j)=0$  for all $j=1,\dots, N$.
	\end{itemize}
\end{definition}

A simple result that will be use repeatedly without mentioning it is the following (see (a) of \ref{lemmanonempty}).

\begin{lemma}
	Let $h\in\bv(a,b)$, and $\delta>0$. Then, there exists a $\delta$-admissible covering for $\Gamma$.
\end{lemma}

\begin{definition}\label{gridconstant}
	Let $h\in\bv(a,b)$, and $\delta>0$. A function $u\in L^1(\Gamma)$ is called \emph{$\delta$-grid constant} if there exists a $\delta$-admissible covering for $\Gamma$, such that $u_{|Q^j\cap \Gamma}=u^j\in\R$, for every $j=1,\dots, N$.
	Moreover, we say that $u\in L^1(\Gamma)$ is \emph{grid constant} if there exists $\delta>0$ such that it is $\delta$-grid constant.
\end{definition}

We are now in position to explain the steps of the strategy that we will use in order to prove Theorem \ref{limsupinequality}.}

\begin{enumerate}
	\item[Step 1:] For any configuration $(\o,v,\mu)\in\mathcal{A}(m,M)$, we find a sequence $(u_k)_k\subset L^1(\Gamma)$ where each $u_k$ is a grid constant function, such that $\mu_k:=u_k\ho\llcorner\Gamma\wtom \mu$ as $k\to\infty$,
	$(\Omega,v,\mu_k)\in \mathcal{A}(m,M)$ for all $k\in\N$, and
	\begin{align*}
		\lim_{k\to\infty}\g(\o,v,\mu_k)\leq\g(\o,v,\mu).
	\end{align*}
	This will be proved in Proposition \ref{step1limsup};
	\item[Step 2:] Let $(\o,v,\mu)\in\mathcal{A}(m,M)$, be such that $\mu=u\ho\llcorner\Gamma$, and $u\in L^1(\Gamma)$ is grid constant. In Proposition \ref{propfinitenumbercuts}, we construct a sequence $\big(\Omega_k,v_k,\mu_k\big)_k\subset\mathcal{A}_r(m,M)$, where $\mu_k=u_k\ho\llcorner\Gamma_k$ and $u_k$ is grid constant, such that $(\o_k,v_k,\mu_k)\to (\o,v,\mu)$ as $k\to\infty$, and
	\begin{align*}
		\lim_{k\to\infty}\g(\Omega_k,v_k,\mu_k)=\g(\o,v,\mu);
	\end{align*}
	\item[Step 3:] For every configuration $(\o,v,\mu)\in\mathcal{A}_r$, with each $u_k$ grid constant, in Proposition \ref{proposition3} we build a sequence $\big(\o_k,v_k,\mu_k\big)_k\subset\mathcal{A}_r$ with $(\o_k,v_k,\mu_k)\to(\o,v,\mu)$ as $k\to\infty$, such that
	\begin{align*}
		\lim_{k\to\infty}\f(\o_k,v_k,\mu_k)=\g(\o,v,\mu);
	\end{align*}
	\item[Step 4:] From Propositions \ref{step1limsup}, \ref{propfinitenumbercuts} and \ref{proposition3} and a diagonalization argument we get the limsup inequality.
\end{enumerate}

\begin{remark}
	Using Theorem \ref{limsupinequality} with Theorem \ref{liminfinequality}, we have proved Theorem \ref{thm:main_no_mass} and Theorem \ref{thm:main_mass}.
\end{remark}

We now carry on Step 1: approximate any admissible configuration with a sequence of configurations where the density is grid constant.

\begin{proposition}\label{step1limsup}
	Let $(\o,v,\mu)\in\mathcal{A}(m,M)$. Then, there exists a sequence $(u_k)_k\subset L^1(\Gamma)$, with $u_k\in L^1(\Gamma)$ grid constant, such that $(\o,v,\mu_k)\to(\o,v,\mu)$, as $k\to\infty$, and
	\begin{align*}
		\lim_{k\to\infty}\g(\o,v,\mu_k)\leq\g(\o,v,\mu),
	\end{align*}
	where $\mu_k:=u_k\ho\llcorner \Gamma$. Moreover, $(\o,v,\mu_k)\in\mathcal{A}(m,M)$.
\end{proposition}

\begin{proof} \emph{Step 1.} Given $(\Omega,v,\mu)\in\mathcal{A}(m,M)$, with $\mu=u\ho\llcorner\Gamma+\mu^s$, we would like to approximate $\mu^s$ with a finite number of Dirac deltas. Given $k\in\N$, consider an $1/k$-admissible covering of $\Gamma$. Let $Q^1,\dots, Q^{N_k}$ be those cubes that intersect with $\Gamma$. For each $i=1,\dots,N_k$, let $x^i_k\in Q^i\cap \Gamma$.
	%
	%
	We define
	\begin{align*}
		m^i_k\coloneqq \mu^s(Q_k^i)
	\end{align*} 
	and set
	\begin{align*}
		\mu_k\coloneqq u\ho\llcorner \Gamma+ \sum_{i=1}^{N_k}m^i_k\delta_{x^i_k},
	\end{align*}
	where, for every $k\in\N$, $N_k$ is finite.
	It is possible to see that $\mu_k(\Gamma)=m$ and $\mu_k\wtom \mu$ as $k\to\infty$. Furthermore, the fact that $\mu^s(\Gamma)=\sum_{i=1}^{N_k}m^i_k$,
	for every $k\in\N$, implies that
	\begin{align*}
		\g(\o,v,\mu_k)=	\g(\o,v,\mu),
	\end{align*}
	for every $k\in\N.$\\

	\emph{Step 2.}
	Now, consider $(\o,v,\mu)\in\mathcal{A}(m,M)$, with $\mu=u\ho\llcorner\Gamma + \sum_{i=1}^{N}m^i\delta_{x_i}$, with $x_i\in \Gamma$ and $m^i> 0$ as defined in step $1$, for every $i=1,\dots N$. 
	We now construct an admissible covering in order to define a suitable density on $\Gamma$.\\
	For $k\in\N$, consider $(Q_k^j)_{j=1}^{L_k}$, an $1/k$-admissible covering for $\Gamma$. Consider the covering of $\Gamma$ given by
	\begin{align}\label{def:reccover}
		\Big(\bigcup_{i=1}^N Q(x^i, 1/k)\Big) \cup \Big[\Big(\bigcup_{j=1}^{L_k} Q_k^j\Big)\setminus	\Big(\bigcup_{i=1}^N Q(x^i, 1/k)\Big)\Big]
	\end{align}
	We notice that $\big(\bigcup_{j=1}^{L_k} Q_k^j\big)\setminus	\big(\bigcup_{i=1}^N Q(x^i, 1/k)\big)$ can be divided $N_k$ rectangles whose sides does not exceed $1/k$. Thus, up to a further subdivision in rectangles, we consider \eqref{def:reccover} as a $1/k$-admissible covering of $\Gamma$. In order to simplify the notation, we denote as $Q_k^j$ any rectangle contained in \eqref{def:reccover}. Furthermore, by reordering the rectangles in \eqref{def:reccover}, we assume that for $j=1,\dots,N$, $Q_k^j\subset\bigcup_{i=1}^N Q(x^i, 1/k)$ and for $j=N+1,\dots,N+N_k$, we have $Q_k^j\subset\big(\bigcup_{j=1}^{L_k} Q_k^j\big)\setminus	\big(\bigcup_{i=1}^N Q(x^i, 1/k)\big)$.\\
	Fix $\varepsilon>0$. Since
	\begin{align*}
		\lim_{k\to\infty}\frac{\mu^s({\widetilde{\Gamma}\cap Q_k^j})}{\ho({\widetilde{\Gamma}\cap Q_k^j})}= +\infty,\quad\text{and}\quad
		\lim_{k\to\infty}
		\frac{\mu^s({\Gamma^c\cap Q_k^j})}{\ho({\Gamma^c\cap Q_k^j})}=+\infty,
	\end{align*}
	for all $j=1,\dots,N$, there is $\bar{k}\in\N$ such that, for every $k\geq \bar{k}$, we have
	\begin{align}\label{eq:prop1.5widepsi}
		\Bigg|\frac{\ho(\Gamma\cap Q_k^j)}{\mu^s(\Gamma\cap Q_k^j)}\widetilde{\psi}\Bigg(\frac{\mu^s(\Gamma\cap Q_k^j)}{\ho(\Gamma\cap Q_k^j)}\Bigg)-\theta\Bigg|<\varepsilon
	\end{align}
	and
	\begin{align}\label{eq:prop1.5psic}
		\Bigg|\frac{\ho(\Gamma\cap Q_k^j)}{\mu^s(\Gamma\cap Q_k^j)}\psi^c\Bigg(\frac{\mu^s(\Gamma\cap Q_k^j)}{\ho(\Gamma\cap Q_k^j)}\Bigg)-\theta\Bigg|<\varepsilon.
	\end{align}
	We now define a density on $\Gamma$. For $\textbf{x}\in Q_k^j$, we define $u_k:\Gamma\to \R$ as
	\begin{align*}
		u_k(\textbf{x})\coloneqq
		\begin{cases}
			\displaystyle\frac{\mu(\widetilde{\Gamma}\cap Q^j_k)}{\ho(\widetilde{\Gamma}\cap Q^j_k)}& \text{ if }  \textbf{x}\in\widetilde{\Gamma},\ \widetilde{\Gamma}\cap Q^j_k \neq\emptyset, \\[15pt]
			\displaystyle	\frac{\mu(\Gamma^c\cap Q^j_k)}{\ho(\Gamma^c\cap Q^j_k)} & \text{ if }  \textbf{x}\in\Gamma^c,\ \Gamma^c\cap Q^j_k \neq\emptyset.
		\end{cases}
	\end{align*}
	Note that the function $u_k\in L^1(\Gamma)$ is $1/k$-grid constant by definition. For each $k\in\N$, define the measure
	\begin{align}\label{mukstep1prop1}
		\mu_k\coloneqq u_k\ho\llcorner\Gamma.
	\end{align}
	By definition, it follows directly that the mass constrained is satisfied, namely that $(\o,v,\mu_k)\in\mathcal{A}(m,M)$.\\
	
	\emph{Step 3.} We now prove that $\mu_k\wtom\mu$ as $k\to\infty$.
	Take $\varphi\in C_c(\R^2)$.
	Fix $\varepsilon>0$. Using the uniform continuity of $\varphi$, there exists $\bar{k}\in\N$ such that for every $k\geq \bar{k}$ we have that
	\[
	|\varphi(\mathbf{x})-\varphi(\mathbf{x}^i_k)|<\varepsilon,
	\]
	for every $\mathbf{x}\in Q_k^j$, where $\mathbf{x}^i_k$ is the intersection point of the diagonals of $Q_k^j$.
	First, we write
	\begin{align}\label{triangstep1}
		\nonumber\Big|\int_\Gamma\varphi\ \text{d}\mu_k&-\int_\Gamma\varphi\ \text{d}\mu\Big|\leq
		\Big|\int_{\widetilde{\Gamma}}\varphi\ \text{d}\mu_k-\int_{\widetilde{\Gamma}}\varphi\ \text{d}\mu\Big|
		+ \Big|\int_{\Gamma^c}\varphi\ \text{d}\mu_k - \int_{\Gamma^c}\varphi\ \text{d}\mu\Big| \\[5pt]
		&\leq \sum_{j=1}^{N+N_k} \Big|\int_{\widetilde{\Gamma}\cap Q_k^j}\varphi\ \text{d}\mu_k
		-\int_{\widetilde{\Gamma}\cap Q_k^j}\varphi\ \text{d}\mu\Big|
		+\sum_{j=1}^{N+N_k} \Big|\int_{\Gamma^c\cap Q_k^j}\varphi\ \text{d}\mu_k - \int_{\Gamma^c\cap Q_k^j}\varphi\ \text{d}\mu\Big|,
	\end{align}
	and we estimate the two terms on the right-hand side of \eqref{triangstep1} separately. We have that
	\begin{align}\label{eq:estimate_limsup_1_1}
		\sum_{j=1}^{N+N_k} \Big|\int_{\widetilde{\Gamma}\cap  Q_k^j}\varphi\ \text{d}\mu_k
		&-\int_{\widetilde{\Gamma}\cap  Q_k^j}\varphi\ \text{d}\mu\Big|
		\leq \sum_{j=1}^{N+N_k} \Big[\int_{\widetilde{\Gamma}\cap  Q_k^j}
		|\varphi(\mathbf{x})-\varphi(\mathbf{x}^i_k)|\ \text{d}\mu_k \nonumber \\[5pt]
		&\hspace{0.5cm}+\int_{\widetilde{\Gamma}\cap  Q_k^j}
		|\varphi(\mathbf{x})-\varphi(\mathbf{x}^i_k)|\ \text{d}\mu
		+|\varphi(\mathbf{x}^i_k)| \big|\mu(\widetilde{\Gamma}\cap  Q_k^j)
		- \mu_k(\widetilde{\Gamma}\cap  Q_k^j) \big| \Big] \nonumber \\[5pt]
		&\leq 2m\varepsilon\|\varphi\|_{\mathcal{C}^0(\R^2)},	
	\end{align}
	where we used the fact that $\mu(\widetilde{\Gamma}\cap  Q_k^j)= \mu_k(\widetilde{\Gamma}\cap  Q_k^j)$ for each $j=1,\dots,M+N_k$ and every $k\in\N$, by definition of $\mu_k$.
	Using similar computations, we also get that the second term on the right-hand side of \eqref{triangstep1} can be estimated as
	\begin{align}\label{eq:estimate_limsup_1_2}
		\sum_{j=1}^{N+N_k} \Big|\int_{\Gamma^c\cap Q_k^j}\varphi\ \text{d}\mu_k-\int_{\Gamma^c\cap Q_k^j}\varphi\ \text{d}\mu\Big| \leq 2m\varepsilon\|\varphi\|_{\mathcal{C}^0(\R^2)},	
	\end{align}
	Finally, from \eqref{triangstep1}, \eqref{eq:estimate_limsup_1_1} and \eqref{eq:estimate_limsup_1_2}, we get
	\begin{align*}
		\Big|\int_\Gamma\varphi\ \text{d}\mu_k-\int_\Gamma\varphi\ \text{d}\mu\Big|
		\leq 4m\varepsilon \|\varphi\|_{\mathcal{C}^0(\R^2)}.
	\end{align*}
	As $\varepsilon>0$ is arbitrary, we can conclude that $\mu_k\wtom\mu$ as $k\to\infty$.\\
	
	\emph{Step 4.} We now prove the convergence of the energy. We will prove that
	\[
	\limsup_{k\to\infty}\g(\o,v,\mu_k) \leq \g(\o,v,\mu).
	\]
	Since the bulk term of the energy is unchanged, we estimate the other contributions. 
	We have that
	\begin{align}\label{eq:limsup_step1_1}
		\nonumber \int_{\widetilde{\Gamma}}&\widetilde{\psi}(u_k)\dho
		+\int_{\Gamma^c}\psi^c(u_k)\dho
		= \sum_{j=1}^{N+N_k} \Bigg[\int_{\widetilde{\Gamma}\cap Q_k^j}\widetilde{\psi}(u_k)\dho
		+\int_{\Gamma^c\cap Q_k^j}\psi^c(u_k)\dho  \Bigg] \\[5pt]
		\nonumber&=\sum_{j=1}^{N+N_k} \Bigg[\ho({\widetilde{\Gamma}\cap Q_k^j})\widetilde{\psi}\Bigg(\frac{\mu(\widetilde{\Gamma}\cap Q_k^j)}{\ho(\widetilde{\Gamma}\cap Q_k^j)}\Bigg)
		+ \ho({\Gamma^c\cap Q_k^j})\psi^c\Bigg(\frac{\mu(\Gamma^c\cap Q_k^j)}{\ho(\Gamma^c\cap Q_k^j)}\Bigg) \Bigg]\\[5pt]
		&=\sum_{j=1}^{N} \Bigg[ \ho({\widetilde{\Gamma}\cap Q_k^j})
		\widetilde{\psi}\Bigg(\med_{\widetilde{\Gamma}\cap Q_k^j}u\dho
		+ \frac{\mu^s({\widetilde{\Gamma}\cap Q_k^j})}{\ho({\widetilde{\Gamma}\cap Q_k^j})} \Bigg) \nonumber \\[5pt]
		&\hspace{1cm}+ \ho({\Gamma^c\cap Q_k^j})
		\psi^c\Bigg(\med_{\Gamma^c\cap Q_k^j}u\dho
		+ \frac{\mu^s(\Gamma^c\cap Q_k^j)}{\ho({\Gamma^c\cap Q_k^j})}\Bigg) \Bigg]\nonumber \\[5pt]
		&\nonumber\hspace{1cm}+\sum_{j=N}^{N+N_k}\Bigg[\ho({\widetilde{\Gamma}\cap Q_k^j})
		\widetilde{\psi}\Big(\med_{\widetilde{\Gamma}\cap Q_k^j}u\dho\Big)+\ho({\Gamma^c\cap Q_k^j})
		\psi^c\Big(\med_{\Gamma^c\cap Q_k^j}u\dho\Big)\Bigg]\\[5pt]
		&\leq\sum_{j=1}^{N} \Bigg[ \ho({\widetilde{\Gamma}\cap Q_k^j})\widetilde{\psi}\Big(\med_{\widetilde{\Gamma}\cap Q_k^j}u\dho\Big)
		+\ho({\widetilde{\Gamma}\cap Q_k^j})\widetilde{\psi}\Bigg(\frac{\mu^s({\widetilde{\Gamma}\cap Q_k^j})}{\ho({\widetilde{\Gamma}\cap Q_k^j})}\Bigg)\nonumber\\[5pt]
		&\hspace{1cm}+\ho(\Gamma^c\cap Q_k^j)\psi^c\Big(\med_{\Gamma^c\cap Q_k^j}u\dho\Big)
		+ \ho(\Gamma^c\cap Q_k^j)\psi^c\Bigg(\frac{\mu^s(\Gamma^c\cap Q_k^j)}{\ho(\Gamma^c\cap Q_k^j)}\Bigg)  \Bigg] \nonumber \\[5pt]
		&\nonumber\hspace{1cm}+\sum_{j=N}^{N+N_k}\Bigg[\ho({\widetilde{\Gamma}\cap Q_k^j})
		\widetilde{\psi}\Big(\med_{\widetilde{\Gamma}\cap Q_k^j}u\dho\Big)+\ho({\Gamma^c\cap Q_k^j})
		\psi^c\Big(\med_{\Gamma^c\cap Q_k^j}u\dho\Big)\Bigg]\\[5pt]
		&\leq\sum_{j=1}^{N} \Bigg[\int_{\widetilde{\Gamma}\cap Q_k^j}
		\widetilde{\psi}(u)\dho 
		+\ho({\widetilde{\Gamma}\cap Q_k^j})\widetilde{\psi} \Bigg(\frac{\mu^s({\widetilde{\Gamma}\cap Q_k^j})}{\ho({\widetilde{\Gamma}\cap Q_k^j})}\Bigg)\nonumber\\[5pt]
		&\hspace{1cm}+\int_{\Gamma^c\cap Q_k^j}\psi^c(u)\dho
		+\ho(\Gamma^c\cap Q_k^j)\psi^c\Bigg(\frac{\mu^s(\Gamma^c\cap Q_k^j)}{\ho(\Gamma^c\cap Q_k^j)}\Bigg) \Bigg]\nonumber \\[5pt]
		&\nonumber\hspace{1cm}+\sum_{j=N}^{N+N_k}\Bigg[
		\int_{\widetilde{\Gamma}\cap Q_k^j}\widetilde{\psi}(u)\dho+
		\int_{	\Gamma^c\cap Q_k^j}\psi^c(u)\dho\Bigg]\\[5pt]
		&=\sum_{j=1}^{N} \Bigg[\int_{\widetilde{\Gamma}\cap Q_k^j}\widetilde{\psi}(u)\dho 
		+ \mu^s({\widetilde{\Gamma}\cap Q_k^j})\frac{\ho({\widetilde{\Gamma}\cap Q_k^j})}{\mu^s({\widetilde{\Gamma}\cap Q_k^j})}\widetilde{\psi}\Bigg(\frac{\mu^s({\widetilde{\Gamma}\cap Q_k^j})}{\ho({\widetilde{\Gamma}\cap Q_k^j})}\Bigg) \nonumber\\[5pt]
		&\nonumber\hspace{1cm} +\int_{\Gamma^c\cap Q_k^j}\psi^c(u)\dho
		+ \mu^s(\Gamma^c\cap Q_k^j)
		\frac{\ho(\Gamma^c\cap Q_k^j)}{\mu^s(\Gamma^c\cap Q_k^j)}	
		\psi^c\Bigg(\frac{\mu^s(\Gamma^c\cap Q_k^j)}{\ho(\Gamma^c\cap Q_k^j)}\Bigg)\Bigg]\\[5pt]
		&\hspace{1cm}+\sum_{j=N}^{N+N_k}\Bigg[
		\int_{\widetilde{\Gamma}\cap Q_k^j}\widetilde{\psi}(u)\dho+
		\int_{	\Gamma^c\cap Q_k^j}\psi^c(u)\dho\Bigg],
	\end{align}
	where in the first inequality we used the sub-additivity of $\widetilde{\psi}$ and $\psi^c$, while in the previous to last step we used Jensen's inequality.
	
	By construction, we have that \eqref{eq:prop1.5widepsi} and \eqref{eq:prop1.5psic} hold. Thus, from \eqref{eq:limsup_step1_1}, we obtain
	\begin{align}\label{eq:step4finale}
		\nonumber\int_{\widetilde{\Gamma}}\widetilde{\psi}(u_k)&\dho
		+\int_{\Gamma^c}\psi^c(u_k)\dho\\[5pt]
		\nonumber&\leq\sum_{j=1}^{N}\Bigg[\int_{\widetilde{\Gamma}\cap Q_k^j}\widetilde{\psi}(u)\dho+\mu^s(\widetilde{\Gamma}\cap Q_k^j)(\theta +\varepsilon)+\int_{	\Gamma^c\cap Q_k^j}\psi^c(u)+\mu^s(\Gamma^c\cap Q_k^j)(\theta+\varepsilon)\Bigg]\\[5pt]
		\nonumber&\hspace{1cm}+\sum_{j=N}^{N+N_k}\Bigg[
		\int_{\widetilde{\Gamma}\cap Q_k^j}\widetilde{\psi}(u)\dho+
		\int_{	\Gamma^c\cap Q_k^j}\psi^c(u)\dho\Bigg] \\[5pt]
		\nonumber&=\sum_{j=1}^{N+N_k}\Bigg[\int_{\widetilde{\Gamma}\cap Q_k^j}\widetilde{\psi}(u)\dho+\int_{	\Gamma^c\cap Q_k^j}\psi^c(u)\dho+\theta\mu^s(\Gamma\cap Q_k^j)+\varepsilon\mu^s(\Gamma\cap Q_k^j)\Bigg]\\[5pt]
		&\leq\int_{\widetilde{\Gamma}}\widetilde{\psi}(u)\dho+\int_{	\Gamma^c}\psi^c(u)\dho+\theta \mu^s(\Gamma)+\varepsilon\mu^s(\Gamma).
	\end{align}
	From \eqref{eq:step4finale}, since $\varepsilon$ is arbitrary, we can conclude
	\[
	\limsup_{k\to\infty} \g(\o,v,\mu_k)\leq\g(\o,v,\mu).
	\]
	This concludes the proof.
\end{proof}

We proceed our analysis with the second step, which will allows us to reduce to the case of a Lipschitz profile and a grid constant adatom density.

\begin{proposition}\label{propfinitenumbercuts}
	Let $(\o,v,\mu)\subset \mathcal{A}(m,M)$ be such that $u$ is grid constant. Then, there exists a sequence $(\o_k,v_k,\mu_k)_k\subset\mathcal{A}_r(m,M)$, where $\mu_k=u_k\ho\llcorner\Gamma_k$ with each $u_k$ grid constant, such that
	\begin{align*}
		\lim_{k\to\infty}\g(\o_k,v_k,\mu_k)=\g(\o,v,\mu),
	\end{align*}
	and $(\o_k,v_k,\mu_k)\to(\o, v,\mu)$, as $k\to\infty$.
\end{proposition}

\begin{proof}
	The strategy of the proof is the following. In \emph{Step 1} we show that it suffices to build the required sequence in case $h$ has finitely many cut points. In \emph{Step 2} we build the recovery sequence. Finally in \emph{Step 3} we show the convergence of the energy.\\
	
	\emph{Step 1.} In this first step we are going to show that it suffices to prove the result in the case $h$ has a finite number of cuts.
	Namely, we prove that there exist sequences $(\Omega_{g_k}, w_k, \nu_k)_k\subset \mathcal{A}(m,M)$ where each $g_k$ has a finite number of cuts, and $\nu_k$ is grid constant, such that
	\[
	\lim_{k\to\infty}\g(\o_{g_k},w_k,\nu_k)=\g(\o,v,\mu),
	\]
	and $(\Omega_{g_k}, w_k, \nu_k)\to (\Omega,v,\mu)$ as $k\to\infty$.
	
	The following construction is inspired by \cite[Theorem 2.8]{FonFusLeoMor07}.
	For $k\in\N$, define (see Figure \ref{fig:limsup_approx})
	\begin{align*}
		\hat{g}_k(x)\coloneqq\min\{\max\{h^-(x)-1/k,0\},h(x)\},
	\end{align*}
	for every $x\in(a,b)$. It is possible to see that, for each $k$, the function $	\hat{g}_k$ is lower semicontinuous, of bounded variation, and such that $\hat{g}_k\leq h$.
	Moreover, thanks to Lemma \ref{lem:finitely_many_cuts}, we have that $\hat{g}_k$ has finitely many cuts.
	We then define
	\begin{align}\label{trickmassconstraint}
		g_k(x) \coloneqq \hat{g}_k(x)+\varepsilon_k,
	\end{align}
	for each $k$, where
	\begin{align*}
		\varepsilon_k\coloneqq\frac{1}{b-a}\Big(M-\int_a^b	\hat{g}_k(x)\ \text{d}x\Big)>0.
	\end{align*}
	Set $\Gamma_k\coloneqq \Gamma_{g_k}$, and note that
	\begin{equation}\label{eq:ho_gk}
		\lim_{k\to\infty}\ho(\Gamma_k) = \ho(\Gamma).
	\end{equation}
	We now need, for each $k\in\N$, to define the displacement $v_k$ and the adatom density $u_k$.
	For the former, by fixing a $y_0<0$ such that $v(\cdot,y_0)\in W^{1,2}\big((a,b);\R^2\big)$, we define
	\begin{align}\label{displacementrestriction}
		w_k(\textbf{x})\coloneqq \begin{cases}
			\displaystyle	v(x,y-\varepsilon_k)\quad&\text{if}\ y>y_0+\varepsilon_k,\\[5pt] 
			v(x,y_0)\quad&\text{if}\ y_0<y\leq y_0+\varepsilon_k,\\[5pt]
			v(x,y)\quad&\text{if}\ y\leq y_0.
		\end{cases}
	\end{align}

For $k\in\N\setminus\{0\}$, and $\mathbf{x}\in\Gamma_k$, we define
	\begin{align*}
		z_k(\mathbf{x})\coloneqq\begin{cases}
			\displaystyle	u\big(x,y+1/k\big)&\quad \text{if}\ (x,y+1/k)\in\widetilde{\Gamma},\ \text{and}\ h(x)>1/k, \\[10pt]
			u(x,y)&\quad \text{if}\ \textbf{x}\in\Gamma^c,\\[10pt]
			u(x,0)&\quad \text{if}\ h(x)=0.
		\end{cases}
	\end{align*}
\blue{
	For each $k\in\N\setminus\{0\}$, we then define the measure
\[
\nu_k\coloneqq (z_k+r_k)\ho\llcorner \Gamma_k,
\]
where
	\begin{align*}
		r_k\coloneqq \frac{1}{\ho(\Gamma_k)}\left[\int_{\Gamma}u\dho -\int_{\Gamma_k}z_k\dho\right].
	\end{align*}}
	We notice that, by using \eqref{eq:ho_gk},
	\begin{equation}\label{eq:rk}
		\lim_{k\to\infty}r_k=0.
	\end{equation}
	
	\begin{figure}
		\includegraphics[scale=1]{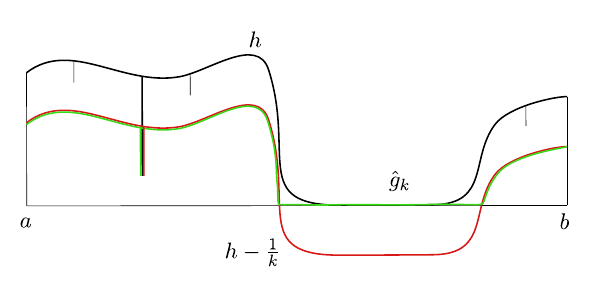}
		\caption{In order to reduce to a finite number of cuts, we do the following: first, we shift down by $1/k$ the regular part of the graph of $h$ (not the cuts), getting the red graph. In this process, some parts of the graph might have gone below zero. Thus, we get the function $\hat{g}_k$ by cutting them, and by adding the remaining part of the original cuts.}
		\label{fig:limsup_approx}
	\end{figure}

	\emph{Step 1.1} \blue{Note that, by definition, the sequences $(g_k)_k$ and $(\nu_k)_k$ satisfy the mass and the density constraint as in Theorem \ref{thm:main_mass}, and thus $(\Omega_{g_k}, w_k, \nu_k)_k\subset \mathcal{A}(m,M)$.}
	
	\emph{Step 1.2} We now prove that $(\Omega_{g_k}, w_k, \nu_k)\to (\Omega,v,\mu)$ as $k\to\infty$.
	By using the definition, it is possible to see that $\R^2\setminus\Omega_{g_k}\stackrel{H}{\rightarrow}\R^2\setminus\Omega$, and $w_k\wto v$ in $W^{1,2}_{\text{loc}}(\Omega;\R^2)$ as $k\to\infty$. In particular, we have that $\ho(\Gamma_k)\to \ho(\Gamma)$ as $k\to\infty$.\\
	We now prove that $\nu_k\wtom \mu$ as $k\to\infty$. Take any $\varphi\in   C_c(\R^2)$ and fix $\varepsilon>0$. By the uniform continuity of $\varphi$ we find $\delta>0$ such that, if $|(x,y-1/k)-(x,y)|<\delta$, we have
	\begin{align*}
		|\varphi(x,y-1/k)-\varphi(x,y)|<\varepsilon.
	\end{align*}
	Then, for $k$ large enough,
	\begin{align*}
		\Big|\int_{\Gamma_k}\varphi &z_k\dho+\int_{	\Gamma_k}r_k\varphi\dho-\int_{\Gamma} \varphi u \dho\Big|\\[5pt]
		&\leq\Big|\int_{\widetilde{\Gamma}\cap \{h>1/k\}}\varphi(x,y-1/k) u\dho+\int_{\widetilde{\Gamma}\cap \{h=0\}}\varphi u(x,0) \dho-\int_{\widetilde{\Gamma}} \varphi u \dho\Big|\\[5pt]
		&\hspace{1cm}+\Big|\int_{\Gamma_k^c}\varphi u\dho-\int_{\Gamma^c} \varphi u \dho\Big|+||\varphi||_{\mathcal{C}^0(\R^2)}\ho(	\Gamma_k)r_k \\[5pt]
		&\leq\varepsilon||u||_{L^1(\widetilde{\Gamma})}+\Big|\int_{\widetilde{\Gamma}\cap \{h>1/k\}}\varphi u\dho-\int_{\widetilde{\Gamma}\cap \{h>0\}}\varphi u\dho\Big|\\[5pt]
		&\hspace{1cm}+\Big|\int_{\Gamma^c\setminus \Gamma^c_k}\varphi u\dho\Big|+||\varphi||_{\mathcal{C}^0(\R^2)}\ho(	\Gamma_k)r_k .
	\end{align*}
	Here we notice that $\Gamma^c\setminus \Gamma_k^c\to\emptyset$, $r_k\to0$ and that $\widetilde{\Gamma}\cap\{h>1/k\}\to\widetilde{\Gamma}\cap\{h>0\}$ as $k\to\infty$. From these considerations, as $\varepsilon$ is arbitrary, we infer that $\nu_k\wtom\mu$ as $k\to\infty$.\\
	
	\emph{Step 1.3} Finally, we prove the convergence of the energy. First, by a standard argument, we can reduce to the case $u\in L^\infty(\Gamma)$.
	Thus, we have
	\begin{align}\label{step1energyest}
		\nonumber|\g(\Omega_k,w_k,\nu_k)-\g(\Omega,v,\mu)|&\leq \Big|\int_{\o_k}W\big(E(w_k)-E_0(y)\big)\ \text{d}\textbf{x}-\int_{\o}W\big(E(v)-E_0(y)\big)\ \text{d}\textbf{x}\Big| \\[5pt]
		\nonumber&\hspace{1cm}+\Big|\int_{\widetilde{\Gamma}_k}\widetilde{\psi}(z_k+r_k)\dho-\int_{\widetilde{\Gamma}}\widetilde{\psi}(u)\dho\Big| \\[5pt]
		&\hspace{1cm}+\Big|\int_{\Gamma_k^c}\psi^c(z_k+r_k)\dho-\int_{\Gamma^c}\psi^c(u)\dho\Big|.
	\end{align}
	Regarding the bulk term on the right-hand side of \eqref{step1energyest}, since $w_k\rightharpoonup v$ in $W^{1,2}_{\text{loc}}(\Omega;\R^2)$ as $k\to\infty$, we have that
	\begin{align*}
		\lim_{k\to\infty}E(w_k)= E(v).
	\end{align*}
	Remember that, by construction, $\o_k\subset \o$. From the fact that $\o_k\to\o$ in $L^1$ as $k\to\infty$, we can find $\bar{k}\in\N$ such that for every $k\geq \bar{k}$, we have $|\o\setminus\o_k|<\varepsilon$. Then, for $k\geq\bar{k}$, we have
	\begin{align}\label{step1bulkest}
		\nonumber
		\Big|\int_{\o_k}W\big(E(w_k)-E_0(y)\big)\ \text{d}\textbf{x}&-\int_\o W\big(E(v)-E_0(y)\big)\ \text{d}\textbf{x}\Big| \\[5pt]
		\nonumber\hspace{1cm}&\leq\int_{\o_k}\big| W\big(E(w_k)-E_0(y)\big)-W\big(E(v)-E_0(y)\big)\big|\ \text{d}\textbf{x}\\[5pt]
		&\hspace{1cm}+\Big|\int_{\o\setminus\o_k}W\big(E(v)-E_0(y)\big)\ \text{d}\textbf{x}\Big|.
	\end{align}
	Notice that the first term on the right-hand side of \eqref{step1bulkest} is zero, whereas, by Dominated Convergence Theorem, we can conclude that the second term is going to zero as $k\to\infty$. \\
	
	We now consider the surface terms on the right-hand side of \eqref{step1energyest}. From \eqref{eq:rk}, we can choose $k$ large enough so that $r_k\leq 1$. Since $u\in L^\infty(\Gamma)$, we have that $\widetilde{\psi}$ and $\psi^c$ are uniformly continuous in $[0,\|u\|_{L^\infty}+1]$. Then, for every $\varepsilon>0$, there is $\bar{k}\in\N$ such that, for every $k\geq\bar{k}$,
	\begin{align}\label{ulinfunifcont}
		|\widetilde{\psi}(u+r_k)-\widetilde{\psi}(u)|<\varepsilon\quad\text{and}\quad|\psi^c(u+r_k)-\psi^c(u)|<\varepsilon.
	\end{align}
	For the first term, we get
	\begin{align}\label{step2energyregular}
		\nonumber\Big|\int_{\widetilde{\Gamma}_k}\widetilde{\psi}(z_k+r_k)&\dho-\int_{\widetilde{\Gamma}}\widetilde{\psi}(u)\dho\Big|\\[5pt]
		&=\Big|\int_{\widetilde{\Gamma}}\big[\widetilde{\psi}(u+r_k)-\widetilde{\psi}(u)\big]\dho\Big|+\Big|\int_{\widetilde{\Gamma}\cap\{0<h<1/k\}}\widetilde{\psi}(z_k+r_k)\dho \Big|.
	\end{align}
	Now we use \eqref{ulinfunifcont}, together with
	\[
	\widetilde{\Gamma}\cap\{0<h<1/k\}\to\emptyset,
	\] 
	and we conclude the convergence to $0$ of the surface term in \eqref{step2energyregular}, as $k\to\infty$.
	Regarding the second surface term on the right-hand side of \eqref{step1energyest}, we have that
	\begin{align}\label{step2energycut}
		\nonumber\Big| \int_{\Gamma_k^c}\psi^c(u+&r_k)\dho-\int_{\Gamma^c}\psi^c(u)\dho\Big|
		\leq \Big|\int_{	\Gamma^c}\big[\psi^c(u+r_k)-\psi^c(u)\big]\dho\Big|\\[5pt]
		&\hspace{1cm}+\Big|\int_{	\Gamma^c\cap \{h^-(x)-1/k<y<h^-(x)\}}\psi^c(u)\dho\Big|
	\end{align}
	From \eqref{ulinfunifcont} and since
	\[
	\Gamma^c\cap \{h^-(x)-1/k<y<h^-(x)\}\to\emptyset
	\]
	for $k\to\infty$, we conclude our estimate on the cut part.\\
	
	By putting together \eqref{step1bulkest}, \eqref{step2energyregular} and  \eqref{step2energycut} in \eqref{step1energyest}, we  get that
	\begin{align*}
		\lim_{k\to\infty} \g(\Omega_k,w_k,\nu_k)=\g(\Omega,v,\mu).
	\end{align*}
	\vspace*{0.4cm}
	
	\emph{Step $2.$} Now, consider $h\in$ $\bv(a,b)$ with a finite number of cuts. Let $(c^i)_{i=1}^n\subset (a,b)$ be the orthogonal projection on the $x$-axes of the cuts. Set
	\begin{align}\label{step2epsilon}
		\varepsilon_0\coloneqq \min\{|c^i-c^j|:\quad i\neq j=1,\dots,n\}.
	\end{align}
	In order to lighten the notation, and since we are considering a function $h$ which has a finite number of cut points, we can work as $h$ had a single cut and then repeating the following construction for the general case. So let $c$ be the cut point of $h$.
	
	The idea of the construction is to use a Yosida-Moreau transform far from the cut point $a<c<b$ and, around the cut, we use an interpolation in $[c-\varepsilon_0/k,c+\varepsilon_0/k]$ in order to get the Hausdorff convergence to the vertical cut.
	We need to apply \blue{the Yosida-Moreau transform of $h$ with maximal slope $k$} beforehand because we need the mass constraint to be satisfied, as we want to use the same procedure as in \eqref{trickmassconstraint}, which requires a sequence that lies below $h$.\blue{ Moreover, since we use \blue{the Yosida-Moreau transform of $h$ with maximal slope $k$} far form the cut point, thanks to \cite[Lemma 2.7]{FonFusLeoMor07}, we have the Hausdorff convergence to our configuration as well as the convergence of the length of the graph.}
	
	We define, for each $k\in \N$,  $h^\ell_k:(a,c)\mapsto [0,\infty)$ as \blue{the Yosida-Moreau transform of $h$ with maximal slope $k$} on $(a,c)$ and $h^r_k:(c,b)\mapsto[0,\infty)$ as the Yosida transform of $h$ on $(c,b)$. Namely
	\begin{align*}
		h^\ell_k(x)&\coloneqq \inf\{h(z)+k|x-z|:z\in(a,c)\}, \\[5pt]
		h^r_k(x)&\coloneqq \inf\{h(z)+k|x-z|:z\in(c,b)\}. 
	\end{align*}
	We have that both $h^\ell_k$ and $h^r_k$ are $k$-Lipschitz functions such that $h^\ell_k\leq h$ and $h^r_k\leq h$. Furthermore, by \cite[Lemma 2.7]{FonFusLeoMor07} we have that $\o_{h^\ell_k}\to \o\cap\big[(a,c)\times\R\big]$ and $\o_{h^r_k}\to \o\cap\big[(c,b)\times\R\big]$ as $k\to\infty$, together with their convergence of the length of their respective graph, namely
	\begin{align*}
		\ho(\Gamma_{h^\ell_k})&\to\ho\big(\Gamma\cap((a,c)\times\R)\big), \\[5pt]
		\ho(\Gamma_{h^r_k})&\to\ho\big(\Gamma\cap((c,b)\times\R)\big),
	\end{align*}
	as $k\to\infty$. We can also extend by continuity $h^\ell_k$ and $h^r_k$ at $c$, as we have both right and left limit of $h$ at $c$. We are going to use the following notation
	\begin{align*}
		S_k&\coloneqq \Big[c-\frac{\varepsilon_0}{k},c+\frac{\varepsilon_0}{k}\Big]\times\R,\\[10pt]
		S_k^\ell&\coloneqq \Big[c-\frac{\varepsilon_0}{k},c\Big]\times\R,\\[10pt]
		S_k^r&\coloneqq \Big[c,c+\frac{\varepsilon_0}{k}\Big]\times\R,
	\end{align*}
	where $\varepsilon_0$ is defined in \eqref{step2epsilon}.
	The definition of our sequence $(h_k)_k$ uses the definition of $h^\ell_k$ and $h^r_k$  outside $S_k$ whereas in $S_k$ we have a linear interpolation from the cut point $(c,h(c))$ and the points $(c-\varepsilon_0/k,h^\ell_k(c-\varepsilon_0/k)$ and $( c+\varepsilon_0/k,h^r_k(c+\varepsilon_0/k))$. We define our Lipschitz sequence as
	\begin{align*}
		\hat{h}_k(x)\coloneqq\begin{cases}
			\displaystyle h^\ell_k(x)\quad&x\in	(a,c-\varepsilon_0/k),\\[10pt]
			\displaystyle m_k^\ell x+q_k^\ell\quad&x\in S_k^\ell,\\[10pt]
			\displaystyle m_k^r x+q_k^r\quad&x\in S_k^r, \\[10pt] h^r_k(x)\quad&x\in	(c+\varepsilon_0/k,b),
		\end{cases}
	\end{align*}
	with suitable coefficients $m^\ell_k,q^\ell_k,m^r_k,q^r_k\in\R$ such that we have linear interpolation from $\big(c-\varepsilon_0/k,h_k^\ell(c-\varepsilon_0/k)\big)$ and $\big(c+\varepsilon_0/k,h_k^r(c+\varepsilon_0/k)\big)$ to the point $(c,h(c))$.  Notice that, by definition, $\hat{h}_k(c)=\hat{h}(c)$ and $h_k$ is continuous. Moreover, thanks to Theorem \ref{thm:properties_PBV}, for $k$ large enough, it holds that $\hat{h}_k\leq h$. Now, following the same path as in \eqref{trickmassconstraint}, we set
	\begin{align*}
		h_k(x)\coloneqq \hat{h}_k(x)+\varepsilon_k,
	\end{align*}
	where
	\begin{align*}
		\varepsilon_k\coloneqq\frac{1}{b-a}\Big(M-\int_a^b	\hat{h}_k(x)\ \text{d}x\Big).
	\end{align*}
	We then have that the sequence $(h_k)_k$ satisfies the mass constraint, namely,
	\begin{align*}
		\int_a^bh_k(x)\ \text{d}x=M.
	\end{align*}
	
	\emph{Step $2.1.$} For every  $k\in\N$, let $\o_k$ be the sub-graph of $h_k$. We prove that $\R^2\setminus\Omega_k\stackrel{H}{\to}\R^2\setminus\Omega$ as $k\to\infty$. We use again the equivalence of the Hausdorff convergence with the Kuratowski convergence (see Proposition \ref{kuratowskyconvergence}).
	Take $\bar{\textbf{x}}=(\bar{x},\bar{y})\in\R^2\setminus\Omega$. We first want to prove that there exists a sequence $(x_k,y_k)_k\subset \R^2\setminus\o_k$ such that $(x_k,y_k)\to\bar{x}$. Then, we have different cases depending on whether $\bar{\textbf{x}}\in S_k$ or not.
	In case $\bar{\textbf{x}}\notin S_k$, as the sequence $(h_k)_k$ is defined as the Yosida-Moreau transform of $h$, away from the cut point we can use Lemma $2.7$ of \cite{FonFusLeoMor07} and we have already the Hausdorff convergence desired.
	
	Next we deal the case in which $\bar{\textbf{x}}\in S_k$. If $\bar{x}=c$ and $\bar{y}\leq h^-(c)$, consider the sequence 
	\begin{align*}
		(x_k,y_k)\coloneqq\Bigg(\frac{\bar{y}-q^\ell_k}{m^\ell_k},\bar{y}\Bigg),
	\end{align*}
	for every $k\in\N$. We obtain $(x_k,y_k)\to(c,\bar{y})$ as $k\to\infty$.
	
	In case $\bar{x}=c$ and $\bar{y}\geq h^-(c)$ or in case $\bar{x}\neq c$, it is enough to consider the constant sequence $(x_k,y_k)\coloneqq(c,\bar{y})$, since by definition $h_k\leq h$ and thus we have that $(\bar{x},\bar{y})\in\R^2\setminus\o_k$,  for every $k\in\N$. \\
	
	We are left to check the second condition of the Kuratowski convergence. Take a sequence $(x_k,y_k)_k\subset\R^2\setminus (\Omega_k\cap S_k)$ and suppose that $(x_k,y_k)\to(x,y)$ as $k\to\infty$. We need to prove that $(x,y)\in\R^2\setminus\Omega$. Since $(x_k,y_k)\in S_k$ and the vertical strip $S_k$ is shrinking to the vertical line $c\times\R$, then we must have that $x=c$ thus the point $(x,y)\in\R^2\setminus\Omega$.
	
	In case our sequence $(x_k,y_k)_k$ is laying both in $\R^2\setminus(\o_k\cap S_k)$ and in $\R^2\setminus(\o_k\setminus S_k)$, as it is converging, it is enough consider $k$ large enough and we get that $(x_k,y_k)$ is only in one of the two sets. Then we can proceed as before.
	
	Thus, we can conclude that $\R^2\setminus\Omega_k\stackrel{H}{\rightarrow}\R^2\setminus\Omega$ as $k\to\infty$.  \\
	
	\emph{Step $2.2.$} We are going to define a density on $\Gamma_k$. Since $u$ is grid constant we can consider a family of squares $(Q^j)_{j\in J}$, with $J=\{1,\dots N\}$, such that on each square $Q^j$ we have
	\begin{align*}
		u_{|Q^j\cap \Gamma}=u^j\in\R.
	\end{align*}
	We now define two index sets
	\begin{align}\label{indexakbk}
		A_k\coloneqq\big\{j\in J:Q^j\cap S_k=\emptyset\big\},\quad\quad
		B_k\coloneqq J \setminus A_k.
	\end{align}
	In order to define what follows, we recall Lemma \ref{lemmanonempty}. The density is then defined as $u_k:\Gamma_k\to\R$ 
	\begin{align*}
		u_k(\textbf{x})\coloneqq\begin{cases}
			\displaystyle u^j\frac{\ho(\widetilde{\Gamma}\cap Q^j)}{\ho(\Gamma_k\cap Q^j)}\quad &\textbf{x}\in \Gamma_k\cap Q^j,\ j\in A_k, \\[15pt]
			\displaystyle a^j\frac{\ho(\Gamma^c\cap Q^j)}{\ho(\Gamma_k\cap Q^j\cap S_k^\ell)}\quad &\textbf{x}\in \Gamma_k\cap Q^j\cap	S_k^\ell,\ j\in B_k, \\[15pt]
			\displaystyle b^j\frac{\ho(\Gamma^c\cap Q^j)}{\ho(\Gamma_k\cap Q^j\cap S_k^r)}\quad &\textbf{x}\in \Gamma_k\cap Q^j\cap	S_k^r,\ j\in B_k,\\[15pt]
			\displaystyle
			u^j\frac{\ho\big((\widetilde{\Gamma}\cap Q^j)\setminus S_k\big)}{\ho\big((\Gamma_k\cap Q^j)\setminus S_k\big)}\quad &\textbf{x}\in (\Gamma_k\cap Q^j)\setminus S_k,\ j\in B_k,
		\end{cases}
	\end{align*}
	where $a^j$, $b^j$ are such that
	\begin{align}\label{step2ajbjuj}
		a^j+b^j=u^j
	\end{align}
	and
	\begin{align}\label{psiabu}
		\psi^c(u^j)=\widetilde{\psi}(a^j)+\widetilde{\psi}(b^j).
	\end{align}
	As the size of the squares is fixed, we take $k$ large enough such that the vertical strip $S_k$ is contained in a single vertical column of squares.
	
	For each $k\in\N$, define the measure $\mu_k\coloneqq u_k\ho\llcorner\Gamma_k$. We have that $\mu_k$ satisfies the density constraint. Indeed,
	\begin{align*}
		\int_{\Gamma_k}u_k&\dho=\sum_{j\in A_k}\int_{\Gamma_k\cap Q^j}u^j\frac{\ho(\widetilde{\Gamma}\cap Q^j)}{\ho(\Gamma_k\cap Q^j)}\dho \\[5pt]
		&\hspace{0.5cm}+ \sum_{j\in B_k}\Big(\int_{\Gamma_k\cap Q^j\cap S_k^\ell} a^j\frac{\ho(\Gamma^c\cap Q^j)}{\ho(\Gamma_k\cap Q^j\cap S_k^\ell)}\dho+\int_{\Gamma_k\cap Q^j\cap S_k^r}b^j\frac{\ho(\Gamma^c\cap Q^j)}{\ho(\Gamma_k\cap Q^j\cap S_k^r)}\dho\\[5pt]
		&\hspace{0.5cm}+\int_{(\Gamma_k\cap Q^j)\setminus S_k}u_j\frac{\ho\big((\widetilde{\Gamma}\cap Q^j)\setminus S_k\big)}{\ho\big((\Gamma_k\cap Q^j)\setminus S_k\big)}\dho\Big)\\[5pt]
		&=\sum_{j\in A_k}u^j\ho(\widetilde{\Gamma}\cap Q^j)+\sum_{j\in B_k}\Big(a_j\ho(\Gamma^c\cap Q^j)+b_j\ho(\Gamma^c\cap Q^j)\\[5pt]
		&\hspace{0.5cm}+u^j\ho\big((\widetilde{\Gamma}\cap Q^j)\setminus S_k\big)\Big)\\[5pt]
		&=\sum_{j=1}^N\int_{\Gamma\cap Q^j}u^j\dho=m,
	\end{align*}
	where in the previous to last step we used \eqref{step2ajbjuj}.\\

	\emph{Step $2.3.$} We prove that $\mu_k\wtom \mu$. Take any $\varphi\in\MC_c(\R^2)$. For every $\varepsilon>0$, we can find $\bar{k}\in \N$ such that for every $k\geq \bar{k}$ we have $|\varphi(\textbf{x})-\varphi(\textbf{x}^j)|\leq\varepsilon$ for all $\textbf{x}\in Q^j$, where $\textbf{x}^j$ denotes the center of the square $Q^j$.\
	From Lemma \ref{lemmanonempty} we have
	\begin{align}\label{densitysk}
		\nonumber\Big|\int_{\Gamma_k}\varphi u_k \dho-\int_{\Gamma}\varphi u\dho\Big|&\leq \sum_{A_k} \Big|\int_{\Gamma_k\cap Q^j}\varphi u_k \dho-\int_{\widetilde{\Gamma}\cap Q^j}\varphi u^j\dho\Big| \\[5pt]
		\nonumber&\hspace{1cm}+\sum_{B_k}\Big(\Big|\int_{\Gamma_k\cap Q^j\cap S_k}\varphi u_k \dho-\int_{\Gamma\cap Q^j\cap S_k}\varphi u^j\dho\Big|\\[5pt]
		&\hspace{1cm}+\Big|\int_{(\Gamma_k\cap Q^j)\setminus S_k}\varphi u_k\dho-\int_{(\widetilde{\Gamma}\cap Q^j)\setminus S_k} \varphi u^j\dho\Big|\Big).
	\end{align}
	We now compute first the sum over the indexes in $A_k$ on the right-hand side of \eqref{densitysk}. By summing and subtracting $\varphi(\textbf{x}^j)$ inside each of the integral, it holds that
	\begin{align}\label{ikestimate}
		\nonumber\sum_{j\in A_k} \Big|\int_{\Gamma_k\cap Q^j}\varphi u^j\frac{\ho(\widetilde{\Gamma}\cap Q^j)}{\ho(\Gamma_k\cap Q^j)} \dho-\int_{\widetilde{\Gamma}\cap Q^j}\varphi u^j\dho\Big|&\leq 2\sum_{j\in A_k} \big|\varphi(\textbf{x})-\varphi(\textbf{x}^j)\big| |u^j| \ho(\widetilde{\Gamma}\cap Q^j)\\[5pt]
		\nonumber&\leq2\varepsilon\sum_{j\in A_k}\ho(\widetilde{\Gamma}\cap Q^j)|u^j|\\[5pt]
		&\leq 2\varepsilon ||u||_{L^1(\widetilde{\Gamma})}.
	\end{align}
	We now estimate the sum over $B_k$ on the right-hand side of \eqref{densitysk}. Note that, up taking a larger $\bar{k}\in\N$, we can assume that
	\[
	\sum_{j\in B_j}|\ho(\Gamma\cap Q^j\cap S_k)-\ho(\Gamma\cap Q^j)|\leq 4\varepsilon,
	\]
	for all $k\geq\bar{k}$.
	Bearing in mind that for every $j\in\N$ it holds $a^j+b^j=u^j$, we get
	\begin{align}\label{jkestimate}
		\nonumber\sum_{j\in B_k}\Big|\int_{\Gamma_k\cap Q^j\cap S_k}&\varphi u_k \dho-\int_{\Gamma\cap Q^j\cap S_k}\varphi u^j\dho\Big|\\[5pt]
		\nonumber&=\sum_{j\in B_k}\Big|\int_{\Gamma_k\cap Q^j\cap S_k^\ell}\varphi a^j\frac{\ho(\Gamma^c\cap Q^j)}{\ho(\Gamma_k\cap Q^j\cap S_k^\ell)}\dho\\[5pt]
		\nonumber&\hspace{1cm}+\int_{\Gamma_k\cap Q^j\cap S_k^r}\varphi b^j\frac{\ho(\Gamma^c\cap Q^j)}{\ho(\Gamma_k\cap Q^j\cap S_k^r)}\dho\\[5pt]
		\nonumber&\hspace{1cm}-\int_{\Gamma\cap Q^j\cap S_k}\varphi u^j\dho\Big|\\[5pt]
		&\leq\nonumber 2\varepsilon\sum_{j\in B_k} u^j\big(\ho(\Gamma^c\cap Q^j)+\ho(\Gamma\cap Q^j\cap S_k)\big)\\[5pt]
		\nonumber&\hspace{1cm}+|\varphi(\textbf{x}^j)| u^j \Big| \ho(\Gamma^c\cap Q^j)-\ho(\Gamma\cap Q^j\cap S_k)\big|\\[5pt]
		&\leq 2\varepsilon(2||u||_{L^\infty(\Gamma)} +4\varepsilon)+4\varepsilon\|\varphi\|_{\mathcal{C}^0(\R^2)}\|u\|_{L^\infty(\Gamma)}.
	\end{align}
	
	In the same way, we can obtain the estimate for last two terms of the sum over $B_k$ on the right-hand side of \eqref{densitysk},
	\begin{align}\label{whatremeains}
		\sum_{j\in B_k}\Big|\int_{(\Gamma_k\cap Q^j)\setminus S_k}\varphi u^j\dho-\int_{(\widetilde{\Gamma}\cap Q^j)\setminus S_k} \varphi u^j\dho\Big|\leq C\varepsilon ||u||_{L^1(\widetilde{\Gamma})},
	\end{align}
	for some constant $C>0$.
	In conclusion, if we put together \eqref{densitysk}, \eqref{ikestimate}, \eqref{jkestimate},  \eqref{whatremeains}, we obtain that
	\begin{align*}
		\Big|\int_{\Gamma_k}\varphi u_k \dho-\int_{\Gamma}\varphi u\dho\Big|< C'\varepsilon,
	\end{align*} 
	with $C'>0$.
	Since $\varepsilon$ is arbitrary, we get that $\mu_k\wtom\mu$ as $k\to\infty$. \\
	
	\emph{Step $2.5.$}
Arguing as in \eqref{displacementrestriction}, we can define the displacement sequence $(v_k)_k$, \blue{with $v_k\in W^{1,2}(\Omega_k;\R^2)$ }such that $v_k\wto v$ in $W^{1,2}_{\text{loc}}(\Omega;\R^2)$ as $k\to\infty$. \\
	
	\emph{Step $2.6.$} It remains to prove the convergence of the energy. By using the index sets in \eqref{indexakbk}, we have that
	\begin{align}\label{jkestimateenergy}
		\nonumber
		|\g(\Omega_k,v_k,\mu_k)-\g(\Omega,v,\mu)|&\leq \Big|\int_{\o_k}W\big(E(v_k)-E_0(y)\big)\ \text{d}\textbf{x}-\int_\o W\big(E(v)-E_0(y)\big)\ \text{d}\textbf{x}\Big|\\[5pt]
		\nonumber&\hspace{1cm}+\sum_{j\in A_k} \Big|\int_{\Gamma_k\cap Q^j}\widetilde{\psi}(u^j)\dho-\int_{\widetilde{\Gamma}\cap Q^j} \widetilde{\psi}(u^j)\dho\Big| \\[5pt]
		\nonumber&\hspace{1cm}+\sum_{j\in B_k} \Big|\int_{\Gamma_k\cap Q^j\cap S_k}\widetilde{\psi}(u^j)\dho-\int_{\Gamma^c\cap Q_k^j}\psi^c(u^j)\dho\Big|\\[5pt]
		&\hspace{1cm}+\sum_{j\in B_k}\Big|\int_{(\Gamma_k\cap Q^j)\setminus S_k}\widetilde{\psi}(u_j)\dho-\int_{(\widetilde{\Gamma}\cap Q^j)\setminus S_k} \widetilde{\psi}(u^j)\dho\Big|.
	\end{align}
	We will estimate the four terms on the right-hand side of \eqref{jkestimateenergy} separately. For the bulk term, we can use the same method as in \eqref{step1bulkest} and we conclude that
	\begin{equation}\label{step2bulkest}
		\begin{aligned}
			\Big|\int_{\o_k}W\big(E(v_k)-E_0(y)\big)\ \text{d}\textbf{x}-\int_\o W\big(E(v)-E_0(y)\big)\ \text{d}\textbf{x}\Big|
			\to 0,
		\end{aligned}
	\end{equation}
	as $k\to\infty$.\\
	We now consider the first sum on the right hand side of \eqref{jkestimateenergy}.We have that
	\begin{align}\label{ikenergyestimatereg}
		\nonumber&\sum_{j\in A_k} \Big|\int_{\Gamma_k\cap Q^j}\widetilde{\psi}\Bigg(u^j\frac{\ho(\widetilde{\Gamma}\cap Q^j)}{\ho(\Gamma_k\cap Q^j)}\Bigg)\dho-\int_{\widetilde{\Gamma}\cap Q^j} \widetilde{\psi}(u^j)\dho\Big|\\[5pt]
		\nonumber&\hspace{1cm}\leq\sum_{j\in A_k}\Big|\widetilde{\psi}\Bigg(u^j\frac{\ho(\widetilde{\Gamma}\cap Q^j)}{\ho(\Gamma_k\cap Q^j)}\Bigg)\ho(\Gamma_k\cap Q^j)-\widetilde{\psi}(u^j)\ho(\Gamma_k\cap Q^j)\Big|\\[5pt]
		&\hspace{2cm}+\sum_{j\in A_k}\Big|\widetilde{\psi}(u^j)\ho(\Gamma_k\cap Q^j)-\widetilde{\psi}(u^j)\ho(\widetilde{\Gamma}\cap Q^j)\Big|.
	\end{align}
	From the fact that $\widetilde{\psi}$ is continuous and since
	$\ho(\Gamma_k\cap Q^j)\to\ho(\widetilde{\Gamma}\cap Q^j)$ as $k\to\infty$, for every $\varepsilon>0$, there is $\bar{k}\in\N$ such that for every $k\geq \bar{k}$ we have
	\begin{align*}
		|\ho(\Gamma_k\cap Q^j)-\ho(\widetilde{\Gamma}\cap Q^j)|<\varepsilon.
	\end{align*}
	and
	\begin{align*}
		\Bigg|\widetilde{\psi}\Bigg(u^j\frac{\ho(\widetilde{\Gamma}\cap Q^j)}{\ho(\Gamma_k\cap Q^j)}\Bigg)-\widetilde{\psi}(u^j)\Bigg|<\varepsilon.
	\end{align*} 
	Then, from \eqref{ikenergyestimatereg} we have that
	\begin{align}\label{ikenergyestimatereg2}
		\nonumber&\sum_{j\in A_k} \Big|\int_{\Gamma_k\cap Q^j}\widetilde{\psi}\Bigg(u^j\frac{\ho(\widetilde{\Gamma}\cap Q^j)}{\ho(\Gamma_k\cap Q^j)}\Bigg)\dho-\int_{\widetilde{\Gamma}\cap Q^j} \widetilde{\psi}(u^j)\dho\Big|\\[5pt]
		&\hspace{1cm}\leq\varepsilon\sum_{j\in A_k}\ho(\Gamma_k\cap Q^j)+\varepsilon\sum_{j\in A_k}\widetilde{\psi}(u^j).
	\end{align}
	As $\varepsilon$ is arbitrary, we can conclude our estimate.\\
	
	Regarding the second sum on the right-hand side of \eqref{jkestimateenergy}, we use the a similar method as in \eqref{jkestimate}.  Now, for the first two terms can be estimated as follows,
	\begin{align}\label{jkestimatecuts}
		\nonumber&\sum_{j\in B_k} \Big|\int_{\Gamma_k\cap Q^j\cap S_k}\widetilde{\psi}(u^j)\dho-\int_{\Gamma^c\cap Q^j}\psi^c(u^j)\dho\Big| \\[5pt]
		\nonumber&\hspace{1cm}= \sum_{j\in B_k}\big|\widetilde{\psi}\Bigg( a^j\frac{\ho(\Gamma^c\cap Q^j)}{\ho(\Gamma_k\cap Q^j\cap S_k^\ell)}\Bigg)\ho(\Gamma_k\cap Q^j\cap S_k^\ell)\\[5pt]
		&\hspace{2cm}+\widetilde{\psi}\Bigg( b^j\frac{\ho(\Gamma^c\cap Q^j)}{\ho(\Gamma_k\cap Q^j\cap S_k^r)}\Bigg)\ho(\Gamma_k\cap Q^j\cap S_k^r)
		-\psi^c(u^j)\ho(\Gamma^c\cap Q^j)\big|.
	\end{align}
	By using the same argument that led us to \eqref{ikenergyestimatereg2}, consider $\varepsilon>0$ as before, then, for $k$ large enough, we have
	\begin{align*}
		\nonumber&|\ho(\Gamma_k\cap Q^j\cap S_k^\ell)-\ho(\Gamma^c\cap Q^j)|<\varepsilon,\\[5pt]
		&|\ho(\Gamma_k\cap Q^j\cap S^r_k)-\ho(\Gamma^c\cap Q^j)|<\varepsilon,
	\end{align*}
	and, by the continuity of $\widetilde{\psi}$,
	\begin{align*}
		\nonumber&\Bigg|\widetilde{\psi}\Bigg( a^j\frac{\ho(\Gamma^c\cap Q^j)}{\ho(\Gamma_k\cap Q^j\cap S_k^\ell)}\Bigg)-\widetilde{\psi}(a^j)\Bigg|<\varepsilon,\\[5pt]
		&\Bigg|\widetilde{\psi}\Bigg( b^j\frac{\ho(\Gamma^c\cap Q^j)}{\ho(\Gamma_k\cap Q^j\cap S_k^r)}\Bigg)-\widetilde{\psi}(b^j)\Bigg|<\varepsilon.
	\end{align*}
	As a consequence, from \eqref{jkestimatecuts} we get
	\begin{align}\label{jkestimatecuts2}
		\nonumber&\sum_{j\in B_k} \Big|\int_{\Gamma_k\cap Q^j\cap S_k}\widetilde{\psi}(u^j)\dho-\int_{\Gamma^c\cap Q^j}\psi^c(u^j)\dho\Big| \\[5pt]
		\nonumber&\hspace{1cm}\leq\varepsilon\sum_{j\in B_k}\big(\ho(\Gamma_k\cap Q^j\cap S_k^\ell)+\ho(\Gamma_k\cap Q^j\cap S_k^r)\big)\\[5pt]
		\nonumber&\hspace{1.5cm}+\sum_{j\in B_k}|\widetilde{\psi}(a_j)\ho(\Gamma_k\cap Q^j\cap S_k^\ell)+\widetilde{\psi}(b_j)\ho(\Gamma_k\cap Q^j\cap S_k^r)-\psi^c(u^j)\ho(\Gamma^c\cap Q^j)|\\[5pt]
		\nonumber&\hspace{1cm}=\varepsilon\sum_{j\in B_k}\ho(\Gamma_k\cap Q^j\cap S_k)+\varepsilon\sum_{j\in B_k}\big(\widetilde{\psi}(a^j)+\widetilde{\psi}(b^j)\big)\\[5pt]
		&\hspace{1.5cm}+\sum_{j\in B_k} |\widetilde{\psi}(a^j)+\widetilde{\psi}(b^j)-\psi^c(u^j)|\ho(\Gamma^c\cap Q^j)
	\end{align}
	Now, we conclude our estimate by using \eqref{psiabu} and the fact that $\varepsilon$ is arbitrary.\\

	The third sum in the right hand side of \eqref{jkestimateenergy} can be treated in the same way as before. 
	Consider $\varepsilon>0$ as above, then, for $k$ large enough we have
	\begin{align*}
		|\ho\big((\widetilde{\Gamma}\cap Q^j)\setminus S_k\big)-\ho(\widetilde{\Gamma}\cap Q^j)|<\varepsilon
	\end{align*} 
	and
	\begin{align*}
		\Bigg|\widetilde{\psi}\Bigg(u_j\frac{\ho\big((\widetilde{\Gamma}\cap Q^j)\setminus S_k\big)}{\ho\big((\Gamma_k\cap Q^j)\setminus S_k\big)}\Bigg)-\widetilde{\psi}(u^j)\Bigg|<\varepsilon.
	\end{align*}
	Thus, we have
	\begin{align}\label{finalestimateenergy}
		\nonumber	&\sum_{j\in B_k} \Big|\int_{(\Gamma_k\cap Q^j)\setminus S_k}\widetilde{\psi}\Bigg(u_j\frac{\ho\big((\widetilde{\Gamma}\cap Q^j)\setminus S_k\big)}{\ho\big((\Gamma_k\cap Q^j)\setminus S_k\big)}\Bigg)\dho-\int_{(\widetilde{\Gamma}\cap Q^j)\setminus S_k} \widetilde{\psi}(u^j)\dho\Big|\\[5pt]
		\nonumber	
		\nonumber&\hspace{1cm}\leq\varepsilon\sum_{j\in B_k}\ho\big((\Gamma_k\cap Q^j)\setminus S_k\big)+\varepsilon\sum_{j\in B_k}\widetilde{\psi}(u^j)\ho\big((\Gamma_k\cap Q^j)\setminus S_k\big)\\[5pt]
		&\hspace{2cm}+ \sum_{j\in B_k}\big|\widetilde{\psi}(u^j)\ho\big((\Gamma_k\cap Q^j)\setminus S_k\big)-\widetilde{\psi}(u^j)\ho(\widetilde{\Gamma}\cap Q^j)\big|.
	\end{align}
	Since $\varepsilon$ is arbitrary and from the fact that $\ho\big((\Gamma_k\cap Q^j)\setminus S_k\big)\to\ho(\widetilde{\Gamma}\cap Q^j)$ as $k\to\infty$, we can conclude the last estimate. \\

	By putting together \eqref{step2bulkest},  \eqref{ikenergyestimatereg2}, \eqref{jkestimatecuts2} and \eqref{finalestimateenergy} in \eqref{jkestimateenergy},  we conclude that
	\begin{align*}
		\lim_{k\to\infty}\g(\Omega_k,v_k,\mu_k)=\g(\Omega,v,\mu).
	\end{align*}
\end{proof}

\begin{proposition}\label{proposition3}
	Let $(\Omega,v,\mu)$ be such that $h$ is a non-negative Lipschitz function, $v\in W^{1,2}(\Omega;\R^2)$ and $\mu=u\ho\llcorner\Gamma$, with $u\in L^1(\Gamma)$ a grid constant density. Then, there is a sequence $(\o_k,v_k,\mu_k)_k\subset \mathcal{A}_r(m,M)$, with  $\mu_k=u_k\ho\llcorner \Gamma_k$ and $u_k\in L^1(\ho\llcorner\Gamma_k)$ grid constant, such that 
	\begin{align*}
		\lim_{k\to\infty}\f(\o_k,v_k,\mu_k)=\g(\o,v,\mu),
	\end{align*}
	and $(\o_k,v_k,\mu_k)\to(\o, v,\mu)$, as $k\to\infty$.
\end{proposition}

\begin{proof}
	\emph{Step 1}. Denote by $\psi^{\text{cvx}}$ the convex envelope of $\psi$, namely,
	\begin{align*}
		\psi^{\text{cvx}}\coloneqq\{\rho: \rho\ \text{is convex and}\ \rho\leq\psi\}.
	\end{align*}
	It is well known (see, for instance, \cite[Theorem 5.32 and Remark 5.33]{FL_book}) that for any given density $w\in L^1(\Gamma_g)$, with $g$ a Lipschitz function, then there is a sequence $(w_m)_m\subset L^1(\Gamma_g)$ such that $w_m\wto w$ in $L^1(\Gamma_g)$ and
	\begin{align*}
		\lim_{m\to\infty}\int_{\Gamma_g}\psi(w_m)\dho=\int_{	\Gamma_g}\psi^{\text{cvx}}(w)\dho.
	\end{align*}
	In particular, $w_m\ho\llcorner \Gamma_g\wtom w\ho\llcorner\Gamma_g$ as $k\to\infty$. Therefore, if we prove the statement of the proposition for $\psi$ convex we also have it for $\psi$ Borel. Thus, from now on, in order to enlighten the notation, we will assume $\psi$ to be a convex function. \\
	
	\emph{Step 2}. Take any configuration $(\Omega,v,\mu)$, where $h$ is a Lipschitz function, $v\in W^{1,2}(\Omega;\R^2)$ and $\mu=u\ho\llcorner\Gamma$ is a grid constant density. Then, we can consider a finite grid of open squares $(Q^j)_{j\in J}$ such that
	\begin{align*}
		u_{|Q^j\cap \Gamma}=u^j\in\R.
	\end{align*}
	for each $j\in J$. By construction, there are finitely many points $a=x^0<x^1<\dots<x^n=b$ such that $u=u^i\in\R$ on
	\begin{align*}
		\text{graph}(h)\cap [(x^i,x^{i+1})\times \R],
	\end{align*}
	for every $i=0,\dots,n$ (see Figure \ref{fig:limsup_wriggling}). 
	\begin{figure}
		\includegraphics[scale=1]{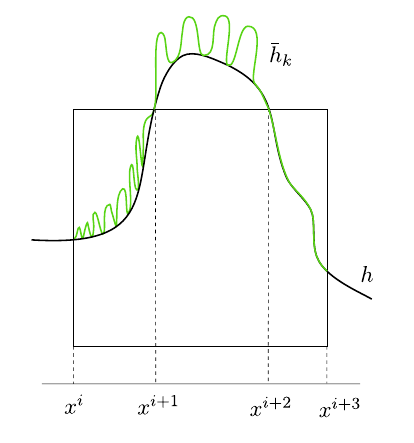}
		\caption{On each interval $[x^i,x^{i+1}]$, depending on whether $u^j>s_0$ or not, we will apply the wriggling process and change the density to $s_0$, or do not change anything.}
		\label{fig:limsup_wriggling}
	\end{figure}
	
	Define the index sets
	\begin{align}\label{ajbj}
		\nonumber A&\coloneqq\{i=1,\dots,n: u^i\leq s_0 \},\\
		B&\coloneqq \{i=1,\dots,n\}\setminus A,
	\end{align}
	where $s_0$ is given by Lemma \ref{lem:char_psi_s0}. In such a way, we are going to apply the wriggling process for $i\in B$.
	By Lemma \ref{step3limsup}, for every $i\in B$, we choose $r^i> 1$ such that
	\begin{align*}
		u^i=r^is_0.
	\end{align*} 
	and we have, on each interval $(x^i,x^{i+1})$, a Lipschitz sequence $(\bar{h}^i_k)_k$,  that verifies the following properties:
	\begin{enumerate}
		\item[$(i)$] $\ho(\Gamma_k^i)= r^i\ho(\Gamma\cap [(x^i,x^{i+1})\times \R])$, where $\Gamma_k^i\coloneqq\text{graph}(\bar{h}_k^i)$,\vspace{0.1cm}
		\item[$(ii)$] $h(x^i)=\bar{h}_k^i(x^i)$, and $ h(x^{i+1})=\bar{h}_k^i(x^{i+1})$,\vspace{0.1cm}
		\item[$(iii)$] $h_{|(x^i,x^{i+1})}\leq \bar{h}_k^i$,\vspace{0.1cm}
		\item[$(iv)$] $\bar{h}_k^i\to h_{|(x^i,x^{i+1})}$ uniformly as $k\to\infty$,\vspace{0.1cm}
		\item[$(v)$] $\ho\llcorner\Gamma_k^i\wtom r^i\ho\llcorner(\Gamma\cap (x^i,x^{i+1})\times \R)$, as $k\to\infty$.
	\end{enumerate}
	Then, we define the Lipschitz sequence $(\bar{h}_k)_k$ as
	\begin{align*}
		\bar{h}_{k|(x^i,x^{i+1})}\coloneqq\begin{cases}
			\bar{h}_k^i\quad&u^i>s_0,\\[5pt]
			h_{|(x^i,x^{i+1})}\quad& u^i\leq s_0,
		\end{cases}
	\end{align*}
	By setting $\overline{\Gamma}_k\coloneqq\text{graph}(\bar{h}_k)$, we define the density $\bar{u}_k$ on $\overline{\Gamma}_k$ as
	\begin{align*}
		\bar{u}_{k|(x^i,x^{i+1})\times\R}\coloneqq\begin{cases}
			s_0\quad&u^i>s_0\\[5pt]
			u^i\quad&u^i\leq s_0,
		\end{cases}
	\end{align*}
	We have that the sequence $(\bar{u}_k)_k$ define above satisfies the density constraint.
	Indeed, by considering the index set defined in \eqref{ajbj}, we have
	\begin{align*}
		\int_{\overline{\Gamma}_k}\bar{u}_k\dho&=\sum_{i\in A} \int_{\overline{\Gamma}_k\cap [(x^i,x^{i+1})\times \R]}	u^i\dho+\sum_{i\in B}\int_{\overline{\Gamma}_k\cap [(x^i,x^{i+1})\times \R]}s_0\dho \\[5pt]
		&=\sum_{i\in A}u^i\ho(\overline{\Gamma}_k\cap [(x^i,x^{i+1})\times \R])+\sum_{i\in B}s_0\ho(\overline{\Gamma}_k\cap [(x^i,x^{i+1})\times \R])\\[5pt]
		&=\sum_{i\in A}u^i\ho(\Gamma\cap [(x^i,x^{i+1})\times \R])+\sum_{i\in B}s_0r^i\ho(\Gamma\cap [(x^i,x^{i+1})\times \R])\\[5pt]
		&=\sum_{i=1 }^n \int_{\Gamma\cap [(x^i,x^{i+1})\times \R]}u^i\dho\\[5pt]
		&=m,
	\end{align*}
	where in the third to last step we used the fact that 
	\begin{align}\label{localwrigglecondition}
		\ho(\Gamma_k\cap [(x^i,x^{i+1})\times \R])=r^i\ho(\Gamma\cap[(x^i,x^{i+1})\times \R]),
	\end{align}
	for every $i\in B$.\\
	
	\emph{Step 3}. Since in general $h\leq \bar{h}_k$, we have that $M=|\o|\leq|\overline{\o}_k|$, where $\overline{\o}_k$ is the sub-graph of $\bar{h}_k$, for each $k\in\N$. In order to fix the mass constraint we set
	\begin{align*}
		\gamma_k\coloneqq \frac{M}{|	\overline{\o}_k|}\leq1,
	\end{align*}
	and we have that $\gamma_k\to 1$ as $k\to\infty$. Define, for each $k\in\N$,
	\begin{align*}
		h_k\coloneqq \gamma_k \bar{h}_k.
	\end{align*}
	Now the sequence $(h_k)_k$ satisfies the mass constraint, indeed
	\begin{align*}
		\int_a^b h_k\ \text{d}x=\int_a^b \gamma_k \bar{h}_k\ \text{d}x=\gamma_k|\overline{\o}_k|=M.
	\end{align*}
	Now, let $\Gamma_k\coloneqq\gr(h_k).$
	Since in general, for every $k\in\N$, $\ho(\Gamma_k)\leq\ho(\overline{\Gamma}_k)$, we need to adjust the density constraint. By knowing that 
	\begin{align*}
		\int_{\overline{\Gamma}_k}\bar{u}_k\dho=m,
	\end{align*}
	we need to define a new sequence of density $(u_k)_k$ on $\Gamma_k$ such that, for every $k\in\N$,
	\begin{align*}
		\int_{\Gamma_k}u_k\dho=m.
	\end{align*}
	Thus we set, for each $k\in\N$,
	\begin{align*}
		u_k\coloneqq\frac{\bar{u}_k}{t_k},
	\end{align*}
	with
	\begin{align*}
		t_k\coloneqq\frac{\ho(\Gamma_k)}{\ho(\overline{\Gamma}_k)} \leq1.
	\end{align*}
	Notice that $t_k\to1$ as $k\to\infty$. We have that the sequence $(u_k)_k$ satisfies the density constraint. Indeed,
	\begin{align*}
		\int_{\Gamma_k}u_k\dho=\frac{\bar{u}_k}{t_k}\ho(\Gamma_k)=\bar{u}_k\ho(\overline{\Gamma}_k)=\int_{\overline{\Gamma}_k}\bar{u}_k\dho=m.
	\end{align*}

	\emph{Step 3}. We now prove the convergence of the density, namely $u_k\ho\llcorner\Gamma_k\wtom u\ho\llcorner\Gamma$. To do so, we first prove that $\bar{u}_k\ho\llcorner\Gamma_k\wtom u\ho\llcorner\Gamma$, and then we conclude by triangle inequality.
	
	Take any $\varphi\in C_c(\R^2)$ and consider $\varepsilon>0$. We can find $\delta>0$ such that, if $\textbf{x},\textbf{y}\in\R^2$ satisfy
	\begin{align*}
		|\textbf{y}-\textbf{x}|<\delta,
	\end{align*}
	then
	\begin{align}\label{step3unifcont}
		|\varphi(\textbf{y})-\varphi(\textbf{x})|<\varepsilon.
	\end{align}
	Up to refining the intervals $(x^i,x^{i+1})$, we can assume that
	\begin{align*}
		|x^i-x^{i+1}|<\frac{\delta}{\sqrt{2}}.
	\end{align*}
	Let $K>0$ such that for every $k\in\N$ we have $h_k\leq K$ and $h\leq K$. This is possible, as our sequence is uniformly bounded by definition and $h$ is bounded. Consider a finite partition of $[0,K]$ given by $y^0=0,y^1,\dots,y^m=K$, such that for every $l=1,\dots,m$ we have
	\begin{align*}
		|y^l-y^{l+1}|<\frac{\delta}{\sqrt{2}}.
	\end{align*}
	Moreover, for every $l$, consider $\bar{y}^l\in[y^l,y^{l+1}]$. Then, from \eqref{step3unifcont}, for every $\textbf{x}\in[x^i,x^{i+1}]\times[y^l,y^{l+1}]$, we have
	\begin{align*}
		|\varphi(\textbf{x})-\varphi(\bar{x}^i,\bar{y}^l)|<\varepsilon.
	\end{align*}
	We then have
	\begin{align*}
		\nonumber\Big|\int_{\overline{\Gamma}_k}\bar{u}_k\varphi \dho-\int_{	\Gamma}u\varphi \dho\Big| &=\Big|\sum_{i\in A}\int_{\overline{\Gamma}_k\cap [(x^i,x^{i+1})\times \R]}u^i\varphi \dho+\sum_{i\in B}\int_{\overline{\Gamma}_k\cap [(x^i,x^{i+1})\times \R]}s_0\varphi \dho\\[5pt]
		\nonumber&\hspace{1cm}-\sum_{i=0}^{n}\int_{\Gamma\cap [(x^i,x^{i+1})\times \R]}u^i\varphi \dho\Big| \\[5pt]
		\nonumber&=\sum_{i\in B}\Big|\int_{\overline{\Gamma}_k\cap [(x^i,x^{i+1})\times \R]}s_0\varphi \dho-\int_{\Gamma\cap [(x^i,x^{i+1})\times \R]}u^i\varphi \dho\Big|\\[5pt]
		\nonumber&=\sum_{l=0}^{m}\sum_{i \in B}\Big|\int_{\overline{\Gamma}_k\cap [(x^i,x^{i+1})\times (y^l,y^{l+1})]}s_0\big[\varphi(\textbf{x})-\varphi\big(\bar{x}^i,\bar{y}^l\big)\big]\dho\Big|\\[5pt]
		\nonumber&\hspace{1cm}+\sum_{l=0}^{m}\sum_{i\in B}\Big|\int_{\Gamma\cap [(x^i,x^{i+1})\times(y^l,y^{l+1})]}u^i \big[\varphi(\textbf{x})-\varphi\big(\bar{x}^i,\bar{y}^l\big)\big]\dho  \Big|\\[5pt]
		\nonumber&\hspace{1cm}+\sum_{l=0}^{m}\sum_{i \in B}\big|s_0\varphi(\bar{x}^i,\bar{y}^l)\ho(\overline{\Gamma}_k\cap [(x^i,x^{i+1})\times (y^l,y^{l+1})])\\[5pt]
		&\hspace{1cm}-u^i\varphi(\bar{x}^i,\bar{y}^l)\ho(\Gamma\cap [(x^i,x^{i+1})\times (y^l,y^{l+1})])\big|\\[5pt]
		&\leq\varepsilon s_0 \sum_{l=0}^{m}\sum_{i \in B}\ho(\overline{\Gamma}_k\cap [(x^i,x^{i+1})\times(y^l,y^{l+1})])\\[5pt]
		&\hspace{1cm}+\varepsilon u^i\sum_{l=0}^{m}\sum_{i \in B}\ho(\Gamma\cap [(x^i,x^{i+1})\times(y^l,y^{l+1})])\\[5pt]
		&\hspace{1cm}+||\varphi||_{\mathcal{C}^0(\R^2)}\sum_{l=0}^{m}\sum_{i \in B}|s_0\ho(\overline{\Gamma}_k\cap [(x^i,x^{i+1})\times(y^l,y^{l+1})])\\[5pt]
		&\hspace{1cm}-u^i\ho(\Gamma\cap [(x^i,x^{i+1})\times(y^l,y^{l+1})])|\\[5pt]
		&\leq\varepsilon s_0 \sum_{i \in B}\ho(\overline{\Gamma}_k\cap [(x^i,x^{i+1})\times\R])\\[5pt]
		&\hspace{1cm}+\varepsilon u^i\sum_{i \in B}\ho(\Gamma\cap [(x^i,x^{i+1})\times\R])\\[5pt]
		&\hspace{1cm}+||\varphi||_{\mathcal{C}^0(\R^2)}\sum_{i \in B}|s_0\ho(\overline{\Gamma}_k\cap [(x^i,x^{i+1})\times\R])\\[5pt]
		&\hspace{1cm}-u^i\ho(\Gamma\cap [(x^i,x^{i+1})\times\R])|\\[5pt]
	\end{align*}
	Now, by using condition \eqref{localwrigglecondition} we get
	\begin{align}\label{gammagkestimate}
		\Big|\int_{\overline{\Gamma}_k}\bar{u}_k\varphi \dho-\int_{	\Gamma}u\varphi \dho\Big|\leq 2\varepsilon||u||_{L^1(\Gamma)},
	\end{align}
	where we can conclude as $\varepsilon$ was arbitrary.\\
	
	In order to prove that
	$u_k\ho\llcorner\Gamma_k\wtom u\ho\llcorner\Gamma$, we can use \eqref{gammagkestimate} together with the triangle inequality and the following estimates. We fix $\varphi$ and $\varepsilon$ as in \eqref{step3unifcont}, so we have
	\begin{align}\label{limsupconvergencedensity}
		\nonumber\Big|\int_{\Gamma_k}u_k\varphi \dho&-\int_{	\overline{\Gamma}_k}\bar{u}_k\varphi \dho\Big|\\[5pt]\nonumber&=\Big|\int_a^b\Big(\frac{\bar{u}_k}{t_k}\varphi(x,h_k(x))\sqrt{1+\gamma_k^2\bar{h}_k'(x)^2}-\bar{u}_k\varphi(x,\bar{h}_k(x))\sqrt{1+\bar{h}_k'(x)^2}\Big)\ \text{d}x\Big| \\[5pt]\nonumber
		&\leq\Big|\int_a^b\Big[\Big(\frac{1}{t_k}-1\Big)\bar{u}_k\varphi(x,h_k(x))\sqrt{1+\gamma_k^2\bar{h}_k'(x)^2}\\[5pt]
		&\hspace{1cm}+\bar{u}_k\varphi(x,h_k(x))\sqrt{1+\gamma_k^2\bar{h}_k'(x)^2}-\bar{u}_k\varphi(x,\bar{h}_k(x))\sqrt{1+\bar{h}_k'(x)^2}\Big]\ \text{d}x\Big|.
	\end{align}
	Regarding the first term on the right hand side of \eqref{limsupconvergencedensity}, we have that the sequence $(\bar{h}_k)_k$ is uniformly Lipschitz, as stated in Remark \ref{remarkwriggle}. Then there is $L>0$ such that $|\bar{h}_k'|\leq L$. Furthermore, we have that, for every $k\in\N$, $|\bar{u}_k|\leq C$, with $C>0$, and we get
	\begin{align}\label{estimate1densitylimsup}
		\Big|\int_{a}^{b}\Big(\frac{1}{t_k}-1\Big)\bar{u}_k\varphi(x,h_k(x))\sqrt{1+\gamma_k^2\bar{h}_k'(x)^2}\ \text{d}x\Big|\leq \Big|\frac{1}{t_k}-1\Big|C||\varphi||_{\mathcal{C}^0(\R^2)}\sqrt{1+\gamma_k^2L^2},
	\end{align}
	Now, we estimate the remaining two terms on the right-hand side of \eqref{limsupconvergencedensity}. Let $\varepsilon'>0$. There is $k'\in\N$ such that for $k\geq k'$ we have 
	\begin{align*}
		|\gamma_k-1|&\leq\varepsilon'.
	\end{align*}
	Since the function $x\mapsto\sqrt{1+x^2}$ is Lipschitz we have
	\begin{align}\label{squarerootlips}
		\big|\sqrt{1+\gamma_k^2\bar{h}_k'(x)^2}-\sqrt{1+\bar{h}_k'(x)^2}\big|
		\leq2|\gamma_k\bar{h}'_k(x)-\bar{h}'_k(x)|
		\leq2L|\gamma_k-1|
		\leq 2L\varepsilon'.
	\end{align}
	Thus we have
	\begin{align}\label{estimate2densitylimsup}
		\nonumber
		\int_a^b\Big| \bar{u}_k\varphi(x,&h_k(x))\sqrt{1+\gamma_k^2\bar{h}_k'(x)^2}-\bar{u}_k\varphi(x,\bar{h}_k(x))\sqrt{1+\bar{h}_k'(x)^2}\Big|\ \text{d}x\\[5pt]
		\nonumber&\leq \int_a^b\Big|\bar{u}_k\varphi(x,h_k(x))\sqrt{1+\gamma_k^2\bar{h}_k'(x)^2}-\bar{u}_k\varphi(x,h_k(x))\sqrt{1+\bar{h}_k'(x)^2}\Big|\ \text{d}x\\[5pt]
		&\hspace{1cm}+\int_a^b\Big|\bar{u}_k\varphi(x,h_k(x))\sqrt{1+\bar{h}_k'(x)^2}-\bar{u}_k\varphi(x,\bar{h}_k(x))\sqrt{1+\bar{h}_k'(x)^2}\Big|\ \text{d}x.
	\end{align}
	Then, the first term on the right-hand side of \eqref{estimate2densitylimsup} can be estimated by using \eqref{squarerootlips} and we get
	\begin{align}\label{34step1}
		\int_a^b\Big|\bar{u}_k\varphi(x,h_k(x))\sqrt{1+\gamma_k^2\bar{h}_k'(x)^2}&-\bar{u}_k\varphi(x,h_k(x))\sqrt{1+\bar{h}_k'(x)^2}\Big|\ \text{d}x 
		\leq K'\varepsilon',
	\end{align}
	where $K'\coloneqq2LC(b-a)||\varphi||_{\mathcal{C}^0(\R^2)}$. \\
	The second term on the right-hand side of \eqref{estimate2densitylimsup} is estimated by using the uniform continuity of $\varphi$. Since there is $C'>0$ such that $|h_k|<C'$, for every $k\in\N$, we also have
	\begin{align*}
		|h_k(x)-\bar{h}_k(x)|&=|\gamma_k-1||\bar{h}_k(x)|\leq\varepsilon'C'.
	\end{align*}
	As a consequence, by using a similar approach as in \eqref{step3unifcont},  we get
	\begin{align}\label{34step2}
		\int_a^b\Big|\bar{u}_k\varphi(x,h_k(x))\sqrt{1+\bar{h}_k'(x)^2}-\bar{u}_k\varphi(x,\bar{h}_k(x))\sqrt{1+\bar{h}_k'(x)^2}\Big|\ \text{d}x\leq K''\varepsilon,
	\end{align}
	where $K''\coloneqq(b-a)C\sqrt{1+L^2}$.
	
	By putting \eqref{34step1} and \eqref{34step2} in \eqref{squarerootlips}, we get that
	\begin{align}\label{kprimoksecondo}
		\int_a^b\Big| \bar{u}_k\varphi(x,h_k(x))&\sqrt{1+\gamma_k^2\bar{h}_k'(x)^2}-\bar{u}_k\varphi(x,\bar{h}_k(x))\sqrt{1+\bar{h}_k'(x)^2}\Big|\ \text{d}x\leq K'\varepsilon'+ K''\varepsilon.
	\end{align}
	Now, by putting \eqref{estimate1densitylimsup} and \eqref{kprimoksecondo} in \eqref{limsupconvergencedensity} we get
	\begin{align}\label{secondsteptraingleineq}
		\Big|\int_{\Gamma_k}u_k\varphi \dho-\int_{	\overline{\Gamma}_k}\bar{u}_k\varphi \dho\Big|\leq K'\varepsilon'+ K''\varepsilon+\Big|\frac{1}{t_k}-1\Big|C||\varphi||_{\mathcal{C}^0(\R^2)}\sqrt{1+\gamma_k^2L^2}.
	\end{align}
	Finally, by using \eqref{gammagkestimate} and \eqref{secondsteptraingleineq} we get
	\begin{align*}
		\Big|\int_{\Gamma_k}u_k\varphi \dho-\int_{	\Gamma}u\varphi \dho\Big| &\leq\Big|\int_{\Gamma_k}u_k\varphi \dho-\int_{\overline{\Gamma}_k}\bar{u}_k\varphi \dho\Big| +\Big|\int_{\overline{\Gamma}_k}\bar{u}_k\varphi \dho-\int_{	\Gamma}u\varphi \dho\Big|\\[5pt]
		&\leq 2\varepsilon||u||_{L^1(\Gamma)}+K'\varepsilon'+ K''\varepsilon+\Big|\frac{1}{t_k}-1\Big|C||\varphi||_{\mathcal{C}^0(\R^2)}\sqrt{1+\gamma_k^2L^2}.
	\end{align*}
	we can conclude since $\varepsilon$ and $\varepsilon'$ were arbitrary and by letting $k\to\infty$.\\

	\emph{Step 4}. Regarding the displacement, set
	\begin{align*}
		v_k(x,y)\coloneqq v(x,\gamma_k y).
	\end{align*}
	The definition of the $v_k$'s is well posed, indeed $(x,\gamma_ky)\in\o_{k}$ if and only if $y\leq \bar{h}_k(x)$. In particular $h\leq \bar{h}_k$, hence $v(x,\gamma_k y)$ is well defined at every point. Notice that, since $h_k\geq 0$, we have that for $y\leq0$ it holds $v_k=v$. Thus, denote the bounded open set
	\begin{align*}
		\o^+\coloneqq \o\cap \{y>0\},
	\end{align*}
	and note that the set
	\begin{align*}
		\o^+_k\coloneqq\{(x,\gamma_ky):(x,y)\in\o^+\}.
	\end{align*}
	is also open and bounded. \\

	We now prove that $v_k\wto v$ in $W^{1,2}_\text{loc}(\o;\R^2)$, as $k\to\infty$. Indeed, take  $\varphi\in\mathcal{C}_c(\R^2)$. Fix $\varepsilon>0$ and since $\varphi$ is uniformly continuous, we have that $|\varphi(\textbf{x})-\varphi(\textbf{y})|<\varepsilon$, every time $|\textbf{x}-\textbf{y}|<\delta$ for some $\delta>0$. In particular, since $\gamma_k\to1$, if $k$ is large enough, we have
	\begin{align*}
		\Big|\varphi\Big(x,\frac{y}{\gamma_k}\Big)-\varphi(x,y)\Big|<\varepsilon.
	\end{align*}
	By using the above fact, we get
	\begin{align*}
		\Big|\int_{\R^2}v_k\varphi\  \text{d}\textbf{x}-\int_{\R^2}v\varphi\  \text{d}\textbf{x} \Big|&=\Big|\int_{\o_k^+}v_k\varphi\  \text{d}\textbf{x}-\int_{\o^+}v\varphi\ \text{d}\textbf{x} \Big| \\[5pt]
		&=\Big|\frac{1}{\gamma_k}\int_{\o^+}v(x,y)\varphi\Big(x,\frac{y}{\gamma_k}\Big)\ \text{d}x\text{d}y-\int_{\o^+}v(x,y)\varphi(x,y)\ \text{d}x\text{d}y\Big|\\[5pt]
		&\leq\frac{1}{\gamma_k}\Big|\int_{\o^+}v(x,y)\Big[\varphi\Big(x,\frac{y}{\gamma_k}\Big)-\varphi(x,y)\Big]\ \text{d}x\text{d}y\Big|\\[5pt]
		&\hspace{1cm}+\Big(\frac{1}{\gamma_k}-1\Big)\int_{\o^+}v(x,y)\varphi(x,y)\ \text{d}x\text{d}y\Big|\\[5pt]
		&\leq\frac{\varepsilon}{\gamma_k}||v||_{L^1(\o)}+\Big|\frac{1}{\gamma_k}-1\Big|||v||_{L^2(\o)}||\varphi||_{L^2(\o)}.
	\end{align*}
	By letting $\varepsilon\to0$ and $k\to\infty$ we conclude the first estimate. Here, we used the Sobolev embedding for $W^{1,2}(\o^+;\R^2)$.
	
	Now we prove the convergence of the gradient.  First we note that the gradients are uniformly bounded, namely it can be verified that
	\begin{align*}
		||\nabla v_k||_{L^2(\o)}\leq C ||\nabla v||_{L^2(\o)},
	\end{align*}
	for some positive uniform constant $C>0$.
	Thus, we have
	\begin{align*}
		\Big|\int_{\R^2}\nabla  v_k\cdot \nabla\varphi\  \text{d}\textbf{x}-\int_{\R^2}\nabla v\cdot\nabla\varphi\  \text{d}\textbf{x}\Big|&=\Big|\int_{\o_k^+}\nabla v_k\cdot \nabla\varphi\  \text{d}\textbf{x}-\int_{\o^+}\nabla v\cdot \nabla\varphi\  \text{d}\textbf{x}\Big|\\[5pt]
		&=\frac{1}{\lambda_k}\int_{\o^+}\partial_xv(x,y)\partial_x\varphi\Big(x,\frac{y}{\lambda_k}\Big)\ \text{d}x\text{d}y\\[5pt]
		&\hspace{1cm}+\int_{\o^+}\partial_yv(x,y)\partial_y\varphi\Big(x,\frac{y}{\lambda_k}\Big)\ \text{d}x\text{d}y,
	\end{align*}
	and, from similar estimates as before, together with the uniform boundedness of the gradients, we can conclude that $v_k\wto v$ in $W^{1,2}(\o^+;\R^2)$, as $k\to\infty$.\\
	
	\emph{Step 5}. It remains to prove the convergence of the energy. Set $\mu_k\coloneqq u_k\ho\llcorner\Gamma_k$. We have
	\begin{align}\label{limsupstepzero}
		\nonumber	\f(\o_k,v_k,\mu_k)-\g(\o,v,\mu)
		&= \int_{\o_k}W\big(E(v_k)-E_0(y)\big)\ \text{d}\textbf{x}-\int_{\o}W\big(E(v)-E_0(y)\big)\ \text{d}\textbf{x} \\[5pt]
		&\hspace{1cm}+\int_{\Gamma_k}\psi(u_k)\dho-\int_{\Gamma}\widetilde{\psi}(u)\dho 
	\end{align}
	\emph{Step 5.1} We now prove the convergence of the bulk term in \eqref{limsupstepzero}.
	\begin{align}\label{bulktermlimsup}
		\nonumber	\int_{\o_k}W\big(E(v_k)&-E_0(y)\big)\ \text{d}\textbf{x}-\int_{\o}W\big(E(v)-E_0(y)\big)\ \text{d}\textbf{x}\\[5pt] 
		\nonumber	&=\int_{\o_k}W\big(E(v(x,\gamma_ky))-E_0(y)\big)\ \text{d}\textbf{x}-\int_{\o}W\big(E(v)-E_0(y)\big)\ \text{d}\textbf{x}\\[5pt]
		\nonumber
		&=\frac{1}{\gamma_k}\Big[\int_{\overline{\o}_k}W\big(E(v(x,z))-E_0\Big(\frac{z}{\gamma_k}\Big)\big)\ \text{d}x\text{d}z-\int_{\o}W\big(E(v)-E_0(z)\big)\ \text{d}x\text{d}z\Big]\\[5pt]
		&\hspace{1cm}+\Big(\frac{1}{\gamma_k}-1\Big)\int_{\o}W\big(E(v)-E_0(z)\big)\ \text{d}x\text{d}z.
	\end{align}
	By noticing that $E_0(z)=E_0(z/\gamma_k)$,
	fix $\varepsilon'>0$  such that, if $k$ is large enough, $|\o_k\setminus\o|\leq\varepsilon'$. In the first two terms on the right-hand side of \eqref{bulktermlimsup}, we have that, for every $k$, $\overline{\o}\subset\o_k$, and then we can proceed as in \eqref{step1bulkest}, and we get
	\begin{align*}
		\frac{1}{\gamma_k}\Big[\int_{\o_k}W\big(E(v(x,z))&-E_0(z)\big)\ \text{d}x\text{d}z-\int_{\o}W\big(E(v)-E_0(z)\big)\ \text{d}x\text{d}z\Big] \\[5pt]
		&=\frac{1}{\gamma_k}\int_{\o_k\setminus\o}W\big(E(v)-E_0(y)\big)\ \text{d}x\text{d}z.
	\end{align*}
	From here we conclude by Dominated Convergence Theorem. Notice that
	the second term on the right-hand side of \eqref{bulktermlimsup} is going to zero, since $\gamma_k\to1$ as $k\to\infty$. \\
	From here we conclude the convergence of the bulk term in \eqref{limsupstepzero}.\\
	
	\emph{Step 5.2} We now consider the surface terms in \eqref{limsupstepzero}.
	Using the index sets defined in \eqref{ajbj}, we get
	\begin{align}\label{finalenergy}
		\nonumber\int_{\Gamma_k}\psi(u_k)\dho&=\sum_{i\in       A}\int_{\Gamma_k\cap [(x^i,x^{i+1})\times \R]}\psi\Big(\frac{u^j}{t_k}\Big)\dho+\sum_{i\in B}\int_{\Gamma_k\cap [(x^i,x^{i+1})\times \R]}\psi\Big(\frac{s_0}{t_k}\Big)\dho.
	\end{align}
	By using the fact that $\psi$ is continuous (as we are in the convexity assumption stated in Step $1$) and from the fact that, for every $i\in B$,
	\begin{align*}
		\psi\Big(\frac{s_0}{t_k}\Big)\ho(\Gamma_k\cap [(x^i,x^{i+1})\times \R])= r^it_k\psi\Big(\frac{s_0}{t_k}\Big)\ho(\Gamma\cap [(x^i,x^{i+1})\times \R]),
	\end{align*}
	we get
	\begin{align*}
		\lim_{k\to\infty}\int_{\Gamma_k}&\psi(u_k)\dho\\[5pt]
		&=\lim_{k\to\infty}\Big[\sum_{i\in A}\psi\Big(\frac{u^j}{t_k}\Big)\ho(\Gamma\cap [(x^i,x^{i+1})\times \R])+\sum_{i\in B} r^it_k\psi\Big(\frac{s_0}{t_k}\Big)\ho(\Gamma\cap [(x^i,x^{i+1})\times \R])\Big]\\[5pt]
		&= \sum_{i\in A} \psi(u^j)\ho(\Gamma\cap [(x^i,x^{i+1})\times \R])+\sum_{i\in B}r^i \psi(s_0)\ho(\Gamma\cap [(x^i,x^{i+1})\times \R])\\[5pt]
		&=\sum_{i\in A} \widetilde{\psi}(u^j)\ho(\Gamma\cap [(x^i,x^{i+1})\times \R])+\sum_{i\in B} \widetilde{\psi}(u^j)\ho(\Gamma\cap [(x^i,x^{i+1})\times \R])\\[5pt]
		&=\int_{	\Gamma}\widetilde{\psi}(u^j)\dho.
	\end{align*}
	This concludes the estimate for the surface term in \eqref{limsupstepzero}. \\ 
	
	\emph{Step 6}. By putting all the steps together, we then conclude that
	\begin{align*}
		\lim_{k\to\infty}\f(\o_k,v_k,\mu_k)=\g(\o,v,\mu).
	\end{align*}
	This completes the proof of Theorem \ref{limsupinequality}.
\end{proof}

\bibliographystyle{siam}
\def\url#1{}
\bibliography{bibliography}

\end{document}